\documentclass[10pt]{article}
\usepackage{tikz}
\usepackage{ucs}
\usepackage{amssymb}
\usepackage{amsthm}
\usepackage{amsmath}
\usepackage{latexsym}
\usepackage[cp1251]{inputenc}
\usepackage[english]{babel}
\usepackage{graphicx}
\usepackage{wrapfig}
\usepackage{caption}
\usepackage{subcaption}
\usepackage{txfonts}
\usepackage{lmodern}
\usepackage{mathrsfs}
\usepackage{algorithm}
\usepackage{multicol}
\usepackage{enumerate}
\usepackage{titlesec}
\usepackage{paralist}
\usepackage{mathtools}
\usepackage[hidelinks]{hyperref}
\usepackage{listings}
\usepackage{changes}

\usepackage{xcolor} 
\definecolor{listinggray}{gray}{0.9}
\definecolor{lbcolor}{rgb}{0.9,0.9,0.9}
\definecolor{Darkgreen}{rgb}{0,0.4,0}
\newtheorem{thm}{Theorem}
\newtheorem{lemma}[thm]{Lemma}
\newtheorem{claim}[thm]{Claim}
\newtheorem{conj}[thm]{Conjecture}

{}
\newtheorem{theorem}[thm]{Theorem}

\newtheorem{example}[thm]{Example}
\newtheorem*{example*}{Example}

\newtheorem{remark}[thm]{Remark}
\newtheorem*{definition*}{Definition}
\newtheorem*{remark*}{Remark}
\newtheorem{question}[thm]{Question}

\titleformat{\paragraph}
{\normalfont\normalsize\bfseries}{\theparagraph}{1em}{}
\titlespacing*{\paragraph}
{0pt}{3.25ex plus 1ex minus .2ex}{1.5ex plus .2ex}

\newcommand*{\myproofname}{Proof}

\usepackage{wasysym}

\usepackage[top=2cm, bottom=2cm, left=2cm, right=2cm]{geometry}

\def\qed{\hfill\ifhmode\unskip\nobreak\fi\qquad\ifmmode\Box\else\hfill$\Box$\fi}

\title{Partition subcubic planar graphs into independent sets}
\date{\today}

\author{
Xujun Liu\thanks{Department of Foundational Mathematics, Xi'an Jiaotong-Liverpool University, Suzhou, Jiangsu Province, 215123, China, xujun.liu@xjtlu.edu.cn; the research of X. Liu was supported by the National Natural Science Foundation of China under grant No.~12401466 and the Research Development Fund RDF-21-02-066 of Xi'an Jiaotong-Liverpool University.} \and
Yan Wang\thanks{School of Mathematical Sciences, Shanghai Frontier Science Center of Modern Analysis
(CMA-Shanghai), Shanghai Jiao Tong University, Shanghai, 200240, China, yan.w@sjtu.edu.cn; the research of Y. Wang was supported by National Key R\&D Program of China under grant No. 2022YFA1006400, Shanghai Municipal Education Commission (No. 2024AIYB003), National Natural Science Foundation of China under grant No. 12201400 and Explore X project of Shanghai Jiao Tong University.}
 }

\begin{document}
	\maketitle

\begin{abstract}
A packing $(1^{\ell}, 2^k)$-coloring of a graph $G$ is a partition of $V(G)$ into $\ell$ independent sets and $k$ $2$-packings (whose pairwise vertex distance is at least $3$). The square coloring of planar graphs was first studied by Wegner in 1977. Thomassen and independently Hartke et al. proved one can always square color a cubic planar graph with $7$ colors, i.e., every subcubic planar graph is packing $(2^7)$-colorable. We focus on packing $(1^{\ell}, 2^k)$-colorings, which lie between proper coloring and square coloring. Gastineau and Togni proved every subcubic graph is packing $(1,2^6)$-colorable and asked whether every subcubic graph except the Petersen graph is packing $(1,2^5)$-colorable. 


In this paper, we prove an analogue result of Thomassen and Hartke et al. on packing coloring that every subcubic planar graph is packing $(1,2^5)$-colorable. This also answers the question of Gastineau and Togni affirmatively for subcubic planar graphs. Moreover, we prove that there exists an infinite family of subcubic planar graphs that are not packing $(1,2^4)$-colorable, which shows that our result is the best possible. Besides, our result is also sharp in the sense that the disjoint union of Petersen graphs is subcubic and non-planar, but not packing $(1,2^5)$-colorable.

\vspace{3mm} \emph{Keywords}: square coloring, planar graphs, cubic graphs, independent sets, packing coloring.

\end{abstract}




\section{Introduction}

An $i$-packing in a graph $G$ is a set of vertices whose pairwise vertex distance is at least $i+1$. Let $S=(s_1,s_2,...,s_k)$ be a non-decreasing sequence of positive integers. A {\em packing $S$-coloring} of a graph $G$ is a partition of $V(G)$ into sets $V_1,...,V_k$ such that each $V_i$ is an $s_i$-packing. In particular, a packing $(1^{\ell}, 2^k)$-coloring of a graph $G$ is a partition of $V(G)$ into $\ell$ independent sets and $k$ $2$-packings. This concept has drawn much attention in graph coloring (e.g., see~\cite{BKL1, BKL2, BF1, BKRW2, FKL1, GT2, GHHHR1, LZZ1, MT2, TT1}). 
The {\em packing chromatic number} (PCN), $\chi_p(G)$, of a graph $G$ is the minimum $k$ such that $G$ has a packing $(1,2,...,k)$-coloring. The notion of PCN was first studied under the name broadcast chromatic number by Goddard, Hedetniemi, Hedetniemi, Harris, and Rall~\cite{GHHHR1} in 2008, which was motivated by a frequency assignment problem in broadcast networks. They asked whether the PCN of subcubic graphs is bounded by a constant. 
Balogh, Kostochka, and Liu~\cite{BKL1} answered their question in the negative using the probabilistic method. Later, Bre\v{s}ar and Ferme~\cite{BF1} gave an explicit construction which shows it is unbounded. 
The $1$-subdivision of a graph $G$ is obtained by replacing each edge with a path of two edges. Gastineau and Togni~\cite{GT2} asked whether the PCN of the $1$-subdivision of a subcubic graph is bounded by $5$ and subsequently, Bre\v{s}ar, Klav\v zar, Rall, and Wash~\cite{BKRW2} conjectured it is true.
Balogh, Kostochka, and Liu~\cite{BKL2} proved the first upper bound on the PCN of the $1$-subdivision of subcubic graphs. This bound was recently improved to $6$ by Liu, Zhang, and Zhang~\cite{LZZ1}. 



The famous Four Color Theorem, which was proved by Appel and Haken~\cite{AH1}, Appel, Haken, and Koch~\cite{AHK1}, as well as Robertson, Sanders, Seymour, and Thomas~\cite{RSST1} states that every planar graph is $4$-colorable, i.e., every planar graph is packing $(1^4)$-colorable.  The square of a graph $G$, denoted by $G^2$, is the graph obtained from $G$ by adding the edges joining vertices with distance exactly two. Wegner~\cite{W1} conjectured in 1977 that if $G$ is a planar graph with maximum degree $\Delta$ then

\[ 
\chi(G^2) \le
\begin{cases} 
      7 & \Delta = 3, \\
      \Delta + 5 & 4 \le \Delta \le 7, \\
      \lfloor \frac{3 \Delta}{2} \rfloor & \Delta \ge 8.
   \end{cases}
\]

Note that a coloring of $G^2$ using $k$ colors is equivalent to a packing $(2^k)$-coloring of $G$. Thomassen~\cite{T1} and independently Hartke, Jahanbekam, and Thomas~\cite{HJT1} confirmed Wegner's conjecture~\cite{W1} for the case when $\Delta = 3$ by proving that every subcubic planar graph is packing $(2^7)$-colorable. Their result is also sharp due to the existence of subcubic planar graphs that are not packing $(2^6)$-colorable.

\begin{theorem}[Thomassen~\cite{T1} and Hartke et al.~\cite{HJT1}]
Every subcubic planar graph is packing $(2^7)$-colorable. 
\end{theorem}

Bousquet, Deschamps, Meyer, and Pierron~\cite{BDMP1} recently showed every planar graph with maximum degree at most four is packing $(2^{12})$-colorable. The square coloring of planar graphs with girth conditions has also been considered by many researchers. For example, Dvo\v r\'ak, Kr\'al, Nejedl\'y, and \v Skrekovski~\cite{DKNS1} showed a planar graph with girth at least six and sufficiently large maximum degree $\Delta$ 
 is packing $2^{\Delta + 2}$-colorable. The list version of square coloring is also studied (e.g., see~\cite{AEH1, BLP1, CK1, KP1} ). In particular, Cranston and Kim~\cite{CK1} proved the list chromatic number of the square of a subcubic graph, except the Petersen graph, is at most $8$.

By Brooks' theorem~\cite{B1}, every subcubic graph except $K_4$ has a proper $3$-coloring, i.e., a packing $(1^3)$-coloring. Gastineau and Togni~\cite{GT2} showed that if one can prove every subcubic graph except the Petersen graph has a packing $(1^2,2^2)$-coloring then the conjecture of Bre\v sar et al.~\cite{BKRW2} is true. This also motivates the study of packing $(1^{\ell}, 2^k)$-colorings, which lie between proper coloring and square coloring. Gastineau and Togni~\cite{GT2} proved that every subcubic graph is packing $(1^2,2^3)$-colorable. Furthermore, they asked whether every subcubic graph except the Petersen graph is packing $(1,1,2,3)$-colorable. This question still remains open. Very recently, Liu et al.~\cite{LZZ1} showed every subcubic graph is packing $(1,1,2,2,3)$-colorable and conjectured that every subcubic graph except the Petersen graph is packing $(1^2,2^2)$-colorable. 


For packing $(1,2^k)$-coloring, Tarhini and Togni~\cite{TT1} proved that every cubic Halin graph is packing $(1,2^5)$-colorable. Mortada and Togni~\cite{MT2} showed that every subcubic graph with no adjacent heavy vertices (degree three vertices with all neighbours also being $3$-vertices) is packing $(1,2^5)$-colorable. Gastineau and Togni~\cite{GT2} proved every subcubic graph is packing $(1,2^6)$-colorable. Moreover, they asked the following question after performing a computer search on graphs with small order.

\begin{question}[Gastineau and Togni~\cite{GT2}]\label{mainquestion}
Is it true that every subcubic graph except the Petersen graph is packing $(1,2^5)$-colorable?   
\end{question}

All graphs in this paper are simple. For a graph $G$, let $V(G)$ and $E(G)$ denote the vertex set and the edge set of $G$. For a vertex $u$ in a graph $G$, we use $N(u)$ to denote the neighbours of $u$, $N[u]$ to denote $N(u) \cup \{u\}$, $N^2(u)$ to denote the set of vertices of distance one and two from $u$, and $N^2[u]$ to denote $N^2(u) \cup \{u\}$.


\section{Main result and sharpness example}

In this paper, we prove an analogue result of Thomassen~\cite{T1} and Hartke et al.~\cite{HJT1} that all subcubic planar graphs are packing $(1,2^5)$-colorable. This also answers Question~\ref{mainquestion} in the affirmative for subcubic planar graphs. 

\begin{theorem}\label{planar}
Every subcubic planar graph is packing $(1,2^5)$-colorable.
\end{theorem}

Our result is the best possible since the Petersen graph is non-planar and is not packing $(1,2^5)$-colorable. To see this, we know the independence number of the Petersen graph is four and its diameter is two. Therefore, at most four vertices can be colored by the $1$-color and each of the remaining at least six vertices must receive a distinct $2$-color, which is impossible. 

Furthermore, we give an infinite family of non-packing-$(1,2^4)$-colorable subcubic planar graphs in the following example. 
This shows the sharpness of our result in another sense.

\begin{figure}
\begin{center}
  \includegraphics[scale=0.66]{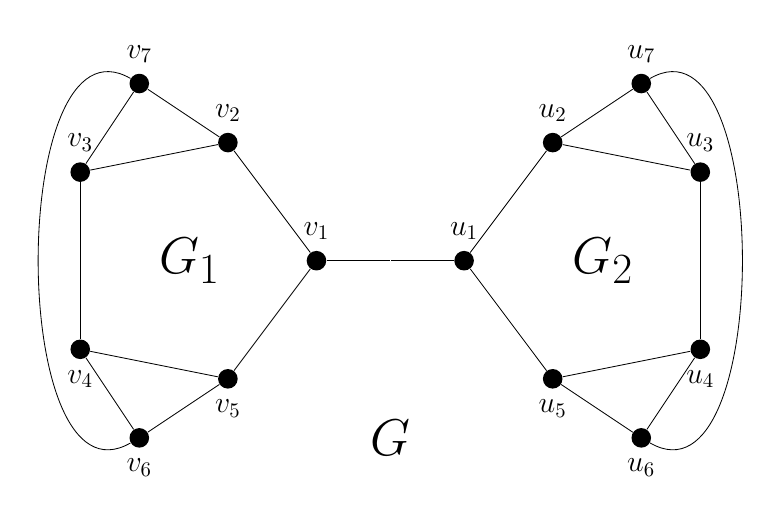}
  \includegraphics[scale=0.66]{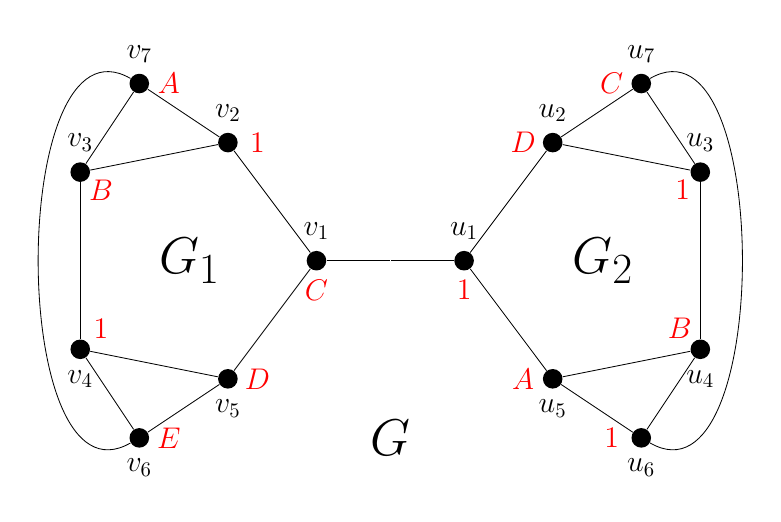}
  \vspace{-8mm}
\caption{A non-packing-$(1,2^4)$-colorable subcubic planar graph.}\label{example}
\end{center}
\vspace{-8mm}
\end{figure}

\begin{example}
Let $G$ be the graph in Figure~\ref{example}, $G_1$ be the subgraph of $G$ induced by the vertices $\{v_1, \ldots, v_7\}$, and $G_2$ be the subgraph of $G$ induced by the vertices $\{u_1, \ldots, u_7\}$. Obviously, the graph $G$ is subcubic and planar. We provide a packing $(1,2^5)$-coloring of $G$ using the colors $1,A,B,C,D,E$, where $1$ is the $1$-color and each of $A,B,C,D,E$ is a $2$-color (see Figure~\ref{example} right picture). We show $G$ is not packing $(1,2^4)$-colorable.
Thus, the disjoint union of $G$'s is an infinite family of non-packing-$(1,2^4)$-colorable subcubic planar graphs.  
\end{example}

\begin{proof}
Suppose to the contrary that $G$ has a packing $(1,2^4)$-coloring using colors $1,A,B,C,D$, where $1$ is the $1$-color and each of $A,B,C,D$ is a $2$-color. We show that $v_1 \in G_1$ must be colored with $1$. Suppose not, say $v_1$ is colored with $A$. Note that $G_1$ has diameter two and thus each of the colors $A,B,C,D$ can be used at most once. Since $v_2v_3v_7$ and $v_4v_5v_6$ are two triangles, the color $1$ can be used at most once in each of the triangles. However, we have at least four uncolored vertices and each of them requires a distinct color from $A,B,C,D$, which is a contradiction to the fact that $G_1$ has diameter two. Similarly, we can show $u_1 \in G_2$ must be colored with $1$. This is a contradiction since $u_1v_1 \in E(G)$.  
\end{proof}

\vspace{-5mm}
\section{Reducible configurations}

It suffices to show Theorem~\ref{planar} for connected graphs since otherwise we can apply the argument to each component. We use {\em good coloring} to denote packing $(1,2^5)$-coloring. In a good coloring, we use colors $1,A,B,C,D,E$, where $1$ is the $1$-color and each of $A,B,C,D,E$ is a distinct $2$-color.  Suppose that Theorem~\ref{planar} is false, i.e., there are subcubic planar graphs that are not packing $(1,2^5)$-colorable. 
Let $G$ be a counterexample with smallest $|V(G)|$. Our plan is first to show that $G$ must be a cubic graph and cannot contain configurations in Figure~\ref{configurations}. Then we use the discharging method to redistribute charges to show that $G$ is too dense to be a planar graph. 

We prove the non-existence of configurations by extending partial good colorings (its existence is guaranteed by the minimality of $G$ after the deletion of each configuration) to $G$. Furthermore, we provide the details of the proof of the non-existence of configurations 
``$3|3$'',``$3$-$3$'',``$3$--$3$'',``$3|4$'',``$3$-$4$'',``$4|4$'',``$3|5$'',``$3|6$'', ``$4|5$'', and``$4|6$'' in the main text. We present the proof of the configuration ``$3$-$5$-$3$'' in the Appendix due to its length.

For other configurations, we group them into two classes. Class One contains the configurations where a $7$-cycle shares an edge with some other faces. It includes configurations ``$3|7|4$'', ``$3|7|5$-I'',``$3|7|5$-II'',``$7|4|4|4$'',``$7|4|5|4$'', ``$7|4|4|5$'', and ``$7|4|5|5|5$''. Class Two consists of the remaining configurations where a $5$-cycle shares an edge with some other faces. It includes configurations ``$5|5|5$-I'', ``$5|5|5$-II'',``$6|5|5$-I'',``$6|5|5$-II'',``$6|5|6$-I'', and ``$6|5|6$-II''. 
We prove the non-existence of configurations in Class One and Class Two with the help of a computer by checking every partial good coloring can be extended.

We include the C++ program in the Appendix. 
For each configuration, we denote the external neighborhood of all the vertices of the configuration \textit{boundary vertices}. 
The boundary vertices, together with their neighborhood not in the configuration, are called \textit{pendant vertices}.
For each configuration, the program iterates over all possible good colorings of pendant vertices, and check whether each of them can be extended to those vertices in the configuration. 
We may add extra edges between boundary vertices while still guaranteeing it is a cubic planar graph in order to limit the number of good colorings to check and thus improve the efficiency of the program. In order to guarantee the boundary vertices and pendant vertices are different from those in configurations in Class One and Class Two, we use the non-existence of $1$-cut, $2$-cut, and configurations proved in Lemmas~\ref{no3|3}-\ref{no4|6}. All cases cannot be covered were listed in Section 6.4 and were checked by computer.
The total running time is less than a week on a Dell desktop with Intel Core i9-12900 Processor CPU and 32GB RAM.
We also include the proof without a computer for one configuration from each class, i.e., configurations ``$3|7|4$'' and ``$5|5|5$-I'' in the Appendix to illustrate the proof idea. 

Often in our arguments, we will color a few uncolored vertices together according to their available colors. We will then use a formulation of Hall's Theorem to obtain a good coloring via a system of distinct representatives. In such a case, we will say statements such as `We are done by SDR'.



\begin{figure}
\vspace{-20mm}
\begin{center}
\hspace{-15mm}
\includegraphics[scale=0.9]{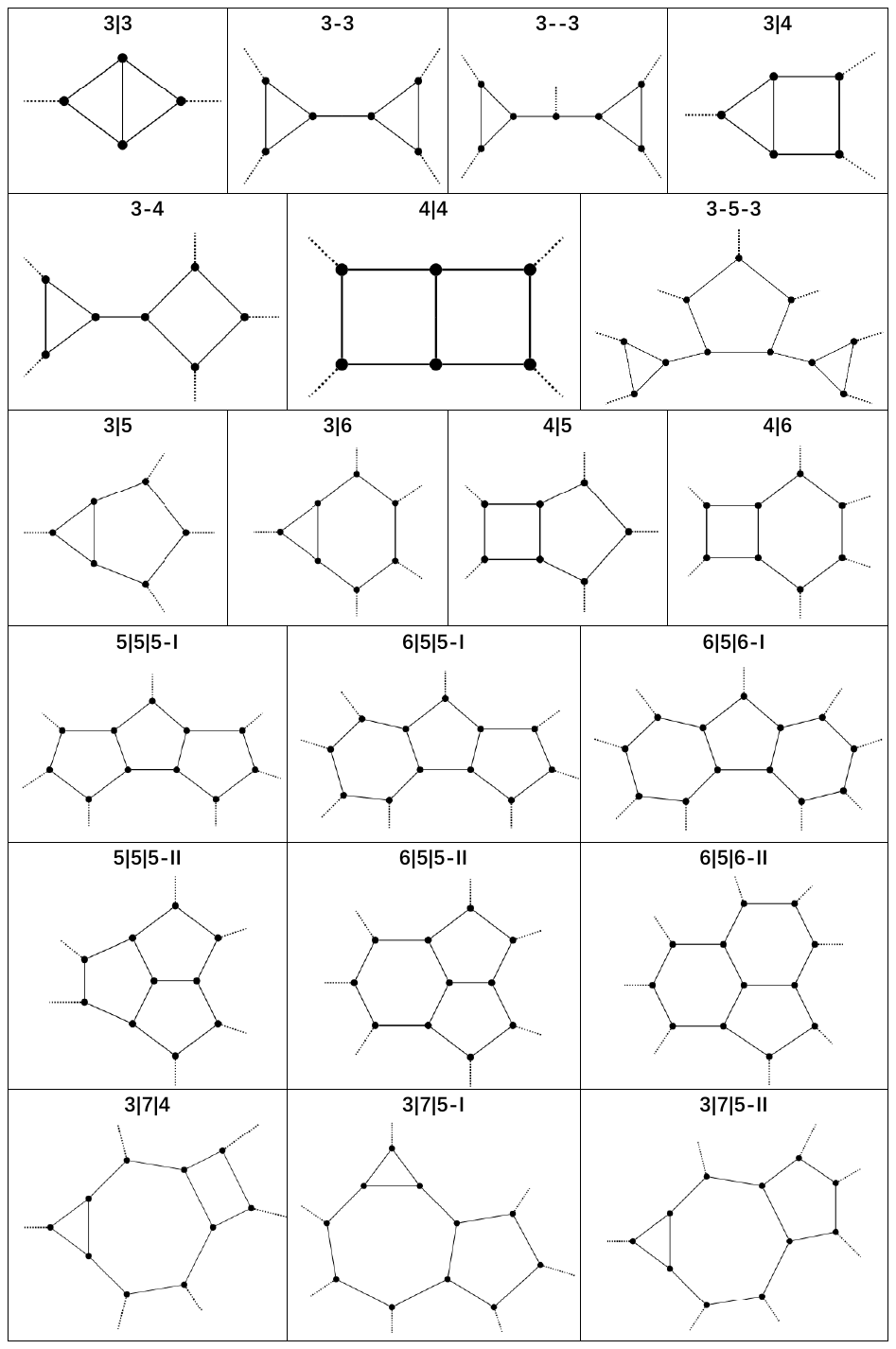}
\end{center}
\end{figure}

\newpage
\begin{figure}
  \vspace{-25mm}
  \hspace{-15mm}
  \includegraphics[scale=0.9]{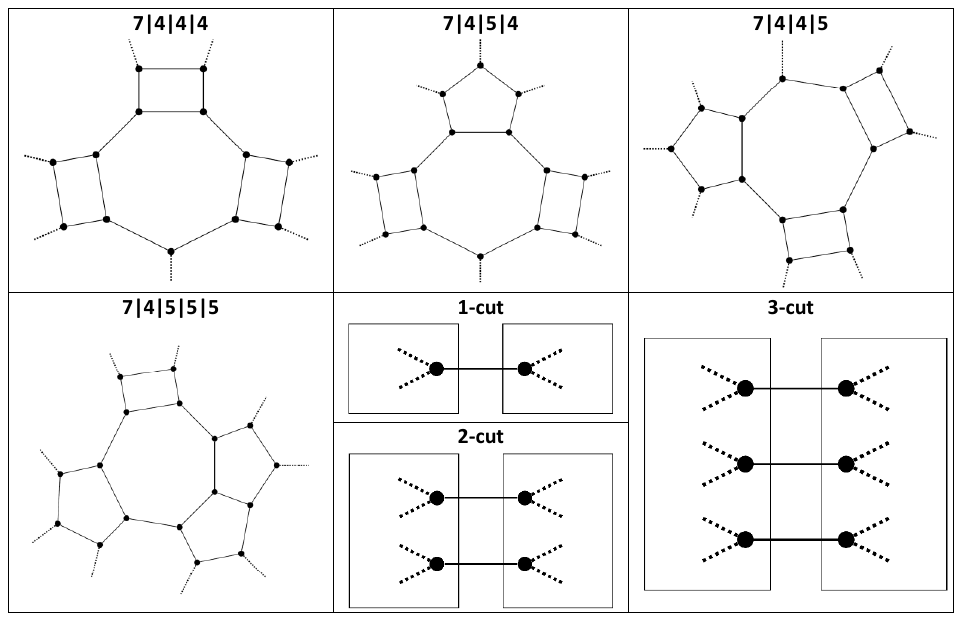}
  \vspace{-140mm}
  \caption{Configurations.}\label{configurations}
  \vspace{-6mm}
\end{figure}

\begin{lemma}\label{cubic}
$G$ is cubic.
\end{lemma}
\begin{proof}
Suppose not, let $u$ be a degree two vertex with neighbours $u_1, u_2$. We delete $u$ and add the edge $u_1u_2$ if it is not in $G$ to obtain a subcubic planar graph $G'$ with one vertex less than $G$ and thus has a good coloring $f$ by our assumption. We extend $f$ to $G$.

 \textbf{Case 1:} $u_1u_2 \in E(G)$.  We may assume both $u_1$ and $u_2$ have degree three. Let $N(u_1) = \{u, u_2, u_3\}$ and $N(u_2) = \{u, u_1, u_4\}$. Then color $u$ with a color $x \in \{1, A, B, C, D, E\} - \{f(u_1), f(u_2), f(u_3), f(u_4)\}$.

 \textbf{Case 2:} $u_1u_2 \notin E(G)$. We may assume both $u_1$ and $u_2$ have degree three. Let $N(u_1) = \{u, u_3, u_4\}$ and $N(u_2) = \{u, u_5, u_6\}$.

 \textbf{Case 2.1:} $1 \in \{f(u_1), f(u_2)\}$. Say $f(u_1) = 1$ and $f(u_2) = A$. We assume $\{f(u_3), f(u_4), f(u_5), f(u_6)\} = \{B, C, D, E\}$ since otherwise we can color $u$ with a color $x \in \{B, C, D, E\} - \{f(u_3), f(u_4), f(u_5), f(u_6)\}$ to obtain a  good coloring of $G$. Then we recolor $u_2$ with $1$ and color $u$ with $A$.

 \textbf{Case 2.2:} $1 \notin \{f(u_1), f(u_2)\}$. Say $f(u_1) = A$ and $f(u_2) = B$. Then we color $u$ with $1$.
\end{proof}

\begin{lemma}\label{degreetwofree}
Let $u$ be a $2$-vertex in $G$ and $f$ be a good coloring of $G$ such that $f(u) = 1$. Let $N(u) = \{u_1, u_2\}$. We show $u$ can always be recolored to a color in $\{A,B,C,D,E\}$, and only the colors of $u_1$ or $u_2$ may be also recolored.  
\end{lemma}

\begin{proof}
We may assume $u_1u_2 \notin E(G)$ since otherwise there are at most $4$ vertices in $N^2(u)$. We can always recolor $u$ with a color in $\{A,B,C,D,E\}$. We may also assume $u_1$ and $u_2$ have no common neighbour except $u$. Otherwise, there are at most $5$ vertices in $N^2(u)$. We can recolor $u$ with a color in $\{A,B,C,D,E\}$ unless they all appear in $N^2(u)$. However, we can switch the colors of $u$ and $u_1$ to obtain the desired coloring. Let $N(u_1) = \{u,u_3, u_4\}$ and $N(u_2) = \{u, u_5, u_6\}$. Then we know each of $A,B,C,D,E$ must appear in $N^2(u)$. Without loss of generality, we assume $f(u_1), f(u_2), f(u_3), f(u_4), f(u_5) = A,B,C,D,E$. If $f(u_6) \neq A$, then we switch the colors of $u$ and $u_1$. Otherwise, we switch the colors of $u$ and $u_2$.
\end{proof}

We first discuss the edge connectivity of $G$. This will be used to prove the non-existence of separating cycles of length up to seven. We first show that there is no cut-edge and therefore no cut-vertex in $G$.

\begin{lemma}\label{nocutedge}
There is no cut-edge in $G$.    
\end{lemma}

\begin{proof}
Suppose there is a cut-edge $uv$. Let $G_1$ ($G_2$) be the component containing $u$ ($v$) in $G-uv$. Let $N(u) = \{u_1, u_2, v\}$ and $N(v) = \{u, v_1, v_2\}$. By the minimality of $G$, each of $G_1$ and $G_2$ has a good coloring, say $f_1$ and $f_2$.

\textbf{Case 1:} Both of $f_1(u)$ and $f_2(v)$ are in $\{A,B,C,D,E\}$. We can permute colors so that $f_1(u) = A$ and $f_1(u_1), f_1(u_2) \in \{1, B, C\}$. Similarly, we can permute colors so that $f_2(v) = D$ and $f_2(u_1), f_2(u_2) \in \{1, B, C\}$.

\textbf{Case 2:} $f_1(u) \in \{A,B,C,D,E\}$, say $f_1(u) = A$, and $f_2(v) = 1$. We can permute the colors so that $f(v_1), f(v_2) \in \{B,C\}$ and $f(u_1), f(u_2) \in \{1, B, C\}$.

\textbf{Case 3:} $f_1(u) = f_2(v) = 1$. By Lemma~\ref{degreetwofree}, we recolor $u$ to one of $A,B,C,D,E$. This is done in Case 2.
\end{proof}

By Lemma~\ref{nocutedge}, if there is a $2$-cut $T = \{u_1v_1, u_2v_2\}$, then all vertices $u_1, u_2, v_1, v_2$ are distinct.

\begin{lemma}\label{no2cut}
There is no $2$-cut in $G$.    
\end{lemma}

\begin{proof}
Suppose there is a $2$-cut $T = \{u_1v_1, u_2v_2\} $ and $u_1,u_2 \in G_1, v_1, v_2 \in G_2$ are two components of $G - T$. By the minimality of $G$, $G_1$ and $G_2$ have good colorings $f_1,f_2$ respectively. We extend $f_1, f_2$ to a good coloring $f$ of $G$.

\textbf{Case 1:} Both $u_1u_2, v_1v_2 \in E(G)$. Let $N(u_1) = \{u_2, u_3, v_1\}, N(u_2) = \{u_1, u_4, v_2\}, N(v_1) = \{u_1, v_2, v_3\}, N(v_2) = \{u_2, v_1, v_4\}$. If $f_1(u_1) = f_2(v_2) = 1$, then we can permute colors so that $f_1(u_3) = f_2(v_4) = C, f_1(u_2) = A, f_2(v_1) = B, f_2(v_3) \in \{1, A,C\}$, and $f_1(u_4) \in \{1, B, C\}$. If $f_1(u_1) = 1$ and $f_1(u_2), f_2(v_1), f_2(v_2) \in \{A,B,C,D,E\}$, then $f_2(v_4) = 1$ (otherwise we can recolor $v_2$ to $1$, which is already solved) and we can permute colors so that $f_1(u_2) = A, f_1(u_3) = C, f_2(v_1) = B, f_2(v_2) = D, f_1(u_4) \in \{1, B, C\}$, and $f_2(v_3) \in \{1, A\}$. If all of $f_1(u_1),f_1(u_2), f_2(v_1), f_2(v_2) \in \{A,B,C,D,E\}$, then $f_1(u_3) = f_1(u_4) = f_2(v_3) = f_2(v_4) = 1$ (otherwise we can recolor $u_1, u_2, v_1, $ or $v_2$ to $1$, which is already solved) and we can permute colors so that $f_1(u_1) = A,f_1(u_2) = B, f_2(v_1) = C, f_2(v_2) = D$. Finally, if $f_1(u_1) = f_1(v_1) = 1$, then by Lemma~\ref{degreetwofree} we can recolor $v_1$ with a color in $\{A,B,C,D,E\}$ (possibly the color of $v_2$ or $v_3$ is also recolored) and it is already solved.

\textbf{Case 2:} $u_1u_2 \notin E(G)$ and $v_1v_2 \in E(G)$. Let $N(u_1) = \{u_3, u_4, v_1\},$ $ N(u_2) = \{u_5, u_6, v_2\}, N(v_1) = \{u_1, v_2, v_3\},$ and $ N(v_2) = \{u_2, v_1, v_4\}$. Note that $u_1$ and $u_2$ may have common neighbours. We extend a good coloring $f_1$ of $G_1 + u_1u_2$ and a good coloring $f_2$ of $G_2$ to a good coloring $f$ of $G$. There are two cases up to symmetry.

\textbf{Case 2.1:} $f_1(u_1) = 1$. Then we may assume $f_1(u_2), f_1(u_3), f_1(u_4) = A, B, C$. By Lemma~\ref{degreetwofree}, we may assume $f_2(v_1) \neq 1$. If $f_2(v_2) = 1$, we permute colors so that $f_2(v_1) = D, f_2(v_4) = B$, and $f_2(v_3) \in \{1,A,B\}$. Therefore, we may assume $f_2(v_2) \neq 1$ and $f_2(v_4) = 1$. However, we permute colors so that $f_2(v_1) = D, f_2(v_3) \in \{1, A\}$, and $f_2(v_2) \in \{B,C,E\} - \{f_1(u_5), f_1(u_6)\}$.

\textbf{Case 2.2:} $f_1(u_1), f_1(u_2) \in \{A,B,C,D,E\}$. Say $f_1(u_1) = A, f_1(u_2) = B$. Then we may assume $f_1(u_3) = f_1(u_5) = 1$, since otherwise we can recolor $u_1$ or $u_2$ with $1$, which is solved in Case 2.1. If one of $v_1$ and $v_2$ has color $1$, say $f_2(v_1) = 1$, then permute colors so that $f_1(u_6) = C$, $f_2(v_2) = D, f_2(v_4) \in \{1,E\}$, and $f_2(v_3) \in \{B,E\}$.

\textbf{Case 3:} Both $u_1u_2, v_1v_2 \notin E(G)$. Let $N(u_1) = \{u_3, u_4, v_1\},$ $ N(u_2) = \{u_5, u_6, v_2\}, N(v_1) = \{u_1, v_3, v_4\},$ and $ N(v_2) = \{u_2, v_5, v_6\}$. Note that $u_1, u_2$ may have common neighbours. We extend a good coloring $f_1$ of $G_1 + u_1u_2$ and a good coloring $f_2$ of $G_2+v_1v_2$ to a good coloring $f$ of $G$. There are three cases up to symmetry.

\textbf{Case 3.1:} $f_1(u_1) = f_2(v_2) = 1$. Then we permute colors so that $f_1(u_2) = f_2(v_1) = A$. Note that $f_1(u_3) \neq A$ and $f_1(u_4) \neq A$ since $u_1u_2$ was added to $G_1$ to obtain $f_1$. Similarly, $f_2(v_5) \neq A$ and $f_2(v_6) \neq A$. This case is done.

\textbf{Case 3.2:} $f_1(u_1) = 1$ and $f_1(u_2), f_2(v_1), f_2(v_2) \in \{A,B,C,D,E\}$. We permute colors so that $f_1(u_2) = f_2(v_1) = A$ and $f_2(v_2) = B$. Since $u_1u_2$ was added to $G_1$ to obtain $f_1$, $f_1(u_3) \neq A$ and $f_1(u_4) \neq A$. Similarly, $f_2(v_5) \neq A$ and $f_2(v_6) \neq A$. Obviously, $f_1(u_5) \neq A$ and $f_1(u_6) \neq A$. We permute colors so that $f_1(u_5), f_1(u_6) \in \{1, C, D\}$.

\textbf{Case 3.3:} $f_1(u_1), f_1(u_2), f_2(v_1), f_2(v_2) \in \{A,B,C,D,E\}$. We permute colors so that $f_1(u_2) = f_2(v_1) = A$ and $f_1(u_1) = f_2(v_2) = B$. Since $u_1u_2$ was added to $G_1$ to obtain $f_1$, $f_1(u_3) \neq A$, $f_1(u_4) \neq A$, $f_1(u_5) \neq B$, $f_1(u_6) \neq B$. Similarly, $f_2(v_3) \neq B$, $f_2(v_4) \neq B$, $f_2(v_5) \neq A$, $f_2(v_6) \neq A$. 
\end{proof}

\begin{lemma}\label{no3|3}
Suppose there is a triangle $u_1, u_2, u_3$. Then each $u_i$, $i \in [3]$ has distinct neighbours apart from the triangle. Say $N(u_1) = \{u_2, u_3, u_4\}$, $N(u_2) = \{u_1, u_3, u_5\}$, $N(u_3) = \{u_1, u_2, u_6\}$, then $u_4 \neq u_5$, $u_5 \neq u_6$, and $u_4 \neq u_6$.  Thus, there are no two triangles sharing an edge.
\end{lemma}

\begin{proof}
Suppose not, i.e., there are two triangles sharing an edge. Let $u_1, u_2, u_3$ be a triangle and assume that $u_5 = u_6$. We know $u_4 \neq u_5$ since otherwise we color $u_1, u_2, u_3, u_5$ with $1, A, B, C$. Furthermore, we know $u_4u_5 \notin E(G)$ since otherwise, say $N(u_4) = \{u_1, u_5, u_7\}$, $u_4u_7$ is a $1$-cut, which contradicts Lemma~\ref{nocutedge}. Let $N(u_5) = \{u_2, u_3, u_7\}$ and $N(u_7) = \{u_5, u_8, u_9\}$. We delete $u_1, u_2, u_3, u_5$ and add $u_4u_7$ to obtain a subcubic planar graph $G'$. By the minimality of $G$, $G'$ has a good coloring $f$.


\textbf{Case 1:} $1 \in \{f(u_4), f(u_7) \}$. Say $f(u_7) = 1$ and $f(u_4) = A$. We may assume $B \notin \{f(u_8), f(u_9)\}$. Then we color $u_1, u_2, u_3, u_5$ with $1, C, D, B$.

\textbf{Case 2:} $1 \notin \{f(u_4), f(u_7)\}$. Say $f(u_4) = A$ and $f(u_7) = B$. Color $u_1, u_2, u_3, u_5$ with $1, C, D, 1$.
\end{proof}

\begin{lemma}\label{no3|4}
There is no triangle sharing an edge with a four-cycle.
\end{lemma}

\begin{proof}
Let $u_1, u_2, u_3$ be a triangle and $u_2, u_3, u_4, u_5$ be a four-cycle such that they share the edge $u_2u_3$. By Lemma~\ref{no3|3}, $u_1u_4, u_1u_5 \notin E(G)$. Let $N(u_1) = \{u_2, u_3, u_6\}$, $N(u_4) = \{u_2, u_5, u_7\}$, and $N(u_5) = \{u_3, u_4, u_8\}$. By Lemma~\ref{no2cut}, $|\{u_6, u_7, u_8\}| = 3$.
We delete $u_1, u_2, u_3$ and add a vertex $u$ with the edges $uu_4$, $uu_5$ and $uu_6$ to obtain a subcubic planar graph $G'$. By the minimality of $G$, $G'$ has a good coloring $f$. We extend $f$ to $G$. 

 \textbf{Case 1:} $1 \in \{f(u_4), f(u_5), f(u_6)\}$.

 \textbf{Case 1.1:} $1 \in \{f(u_4), f(u_5)\}$ and $f(u_6) \neq 1$. Say $f(u_4) = 1$. We color $u_1$ with $1$. Then each of $u_2$ and $u_3$ has at least two available colors from $\{A, B, C, D, E\}$ and thus $f$ can be extended to $G$.

 \textbf{Case 1.2:} $1 \notin \{f(u_4), f(u_5)\}$ and $f(u_6) = 1$. Color $u_3$ with $1$. Then $u_1,u_2$ each has at least $1$ and $2$ available color from $\{A, B, C, D, E\}$ respectively. Thus, we can extend $f$ to $G$.

 \textbf{Case 1.3:} $1 \in \{f(u_4), f(u_5)\}$ and $f(u_6) = 1$. Say $f(u_4) = 1$. We color $u_2$ with $1$. Then each of $u_1$ and $u_3$ has at least two available colors from $\{A, B, C, D, E\}$ and thus $f$ can be extended to $G$.

 \textbf{Case 2:} $1 \notin \{f(u_4), f(u_5), f(u_6)\}$. We may assume $f(u_8) = 1$ since otherwise we recolor $u_5$ with $1$ and it is solved in Case 1. Similarly, we may assume $f(u_7) = 1$. We color $u_1$ with $1$. Then each of $u_4$ and $u_5$ has at least two available colors from $\{A, B, C, D, E\}$, and thus we can extend $f$ to $G$.
\end{proof}

Let $u_1u_2u_3$ be a triangle in $G$ if it exists. Say each $u_i$ has its neighbor $v_i$ outside of the triangle and $N(v_i) = \{u_i, v_i', v_i''\}$, where $i \in [3]$. By Lemmas~\ref{no3|3} and~\ref{no3|4}, we know that $|\{v_1, v_2, v_3\}| = 3$ and there is no edge joining two vertices from $\{v_1, v_2, v_3\}$. We obtain $G'$ from $G$ by contracting $u_1,u_2,u_3$ to one vertex $u$. Note that $G'$ is a subcubic planar graph with $|V(G')| < |V(G)|$. By the inductive hypothesis, $G'$ has a good coloring $f$.

\begin{lemma}\label{only-case}
For every good coloring $f$ of $G'$, we must have $f(v_1), f(v_2), f(v_3) \in \{A,B,C,D,E\}$ and $|\{f(v_1), f(v_2),$ $f(v_3)\}| = 3$. Furthermore, $\{f(v_i'), f(v_i'')\} = \{1, x\}$ for each $i \in [3]$ and $x \in \{A,B,C,D,E\} - \{f(v_1), f(v_2), f(v_3)\}$.
\end{lemma}

\begin{proof}
We show that if the statement does not hold, then we can extend $f$ to a good coloring of $G$.

 \textbf{Case 1:} $f(v_1) = f(v_2) = f(v_3) = 1$. Then each of $u_1, u_2, u_3$ has three available colors from $\{A, B, C, D, E\}$ and thus we can extend $f$ to $u_1, u_2, u_3$.

 \textbf{Case 2:} Two of $v_1, v_2, v_3$ are colored by $1$. Without loss of generality, we assume $v_3$ is colored by a color in $\{A, B, C, D, E\}$. We may assume $1 \in \{f(v_3'), f(v_3'')\}$ since otherwise we recolor $v_3$ with $1$ and it was solved in Case 1. Then we color $u_3$ with $1$ and each of $u_1, u_2$ have at least two available colors from $\{A, B, C, D, E\}$. Thus, we can color $u_1$ and $u_2$. 

 \textbf{Case 3:} One of $v_1, v_2, v_3$ is colored by $1$. Without loss of generality, we assume $f(v_1) = 1$, $f(v_2) = A$, and $f(v_3) = B$. We know $1 \in \{f(v_3'), f(v_3'')\}$ since otherwise we can recolor $v_3$ by $1$ and it was solved in Case 2. We color $u_2$ by $1$. Then $u_1$ has at least one color available from $\{A,B,C,D,E\}$ and $u_3$ has at least two colors available from $\{A,B,C,D,E\}$. Thus, we can extend the coloring to $G$.

 \textbf{Case 4:} None of $v_1, v_2, v_3$ is colored by $1$. By construction of $G'$, we can assume $f(v_1) = A$, $f(v_2) = B$, and $f(v_3) = C$. We also know $1 \in \{f(v_1'), f(v_1'')\}$, $1 \in \{f(v_2'), f(v_2'')\}$, $1 \in \{f(v_3'), f(v_3'')\}$, since otherwise we can recolor one of $v_1, v_2, v_3$ to $1$ and it was solved in Case 3. We know each of $u_1, u_2, u_3$ has at least one color available from $\{A,B,C,D,E\}$. If any two of $u_1,u_2, u_3$ together has two available colors from $\{A,B,C,D,E\}$, then we can color $u_3$ by $1$ and finish the extension greedily. Therefore, their available color must be the same. 
\end{proof}

\begin{remark}\label{neighbor}
In the later proof, although we might use different names to denote the vertices $v_1, v_2, v_3$ due to the structure difference, we may assume the corresponding vertex of $v_1$ is colored by $A$, that of $v_2$ is colored by $B$, and that of $v_3$ is colored by $C$ respectively. Moreover, we may assume the corresponding neighborhood $f(N(v_1)-u_1) = f(N(v_2)-u_2) = f(N(v_3)-u_3) = \{1,D\}$.
\end{remark}

\begin{lemma}\label{no3-3}
There is no two triangles connected by an edge.
\end{lemma}

\begin{proof}
Let $u_1, u_2, u_3$ be a triangle and $u_4, u_5, u_6$ be another triangle such that $u_3u_4$ is an edge. By Lemma~\ref{no3|3} and~\ref{no3|4}, $u_1u_5, u_1u_6, u_2u_5, u_2u_6 \notin E(G)$. Let $N(u_1) = \{u_2, u_3, u_7\}$, $N(u_2) = \{u_1, u_3, u_8\}$, $N(u_5) = \{u_4, u_6, u_9\}$, and $N(u_6) = \{u_4, u_5, u_{10}\}$. By Lemma~\ref{no3|3} and~\ref{no3|4}, $u_7 \neq u_8$, $u_9 \neq u_{10}$, $u_7u_8 \notin E(G)$, and $u_9u_{10} \notin E(G)$. We contract $u_1, u_2, u_3$ to a single vertex $u$ to obtain a subcubic planar graph $G'$. By the minimality of $G$, $G'$ has a good coloring $f$. We extend $f$ to $G$.

By Lemma~\ref{only-case} and Remark~\ref{neighbor}, we know $1 \notin \{f(u_4), f(u_7), f(u_8)\} $, $f(u_4) = C$, $f(u_7) = A$, $f(u_8) = B$, $f(N(u_7)-u_1) = f(N(u_8) - u_2) = f(N(u_4) - u_3) = \{1,D\}$. We recolor $u_4$ with a color $x \in \{A, B, D, E\} - \{f(u_6), f(u_9), f(u_{10})\}$, and color $u_1, u_2, u_3$ with $1,E,C$, which is a contradiction.
\end{proof}

\begin{lemma}\label{no3|5}
There is no adjacent triangle with a five-cycle.
\end{lemma}

\begin{proof}
Let $u_1, u_2, u_3$ be a triangle and $u_2, u_3, u_4, u_5, u_6$ be a five cycle such that they share the edge $u_2u_3$. By Lemma~\ref{no3|3} and~\ref{no3|4}, we know $u_1u_6, u_1u_5, u_1u_4 \notin E(G)$. Let $N(u_1) = \{u_2, u_3, u_7\}$, $N(u_4) = \{u_3, u_5, u_8\}$, $N(u_5) = \{u_4, u_6, u_9\}$, and $N(u_6) = \{u_2, u_5, u_{10}\}$. By Lemma~\ref{no3|4}, $u_7 \notin \{u_8, u_{10}\}$. We contract $u_1, u_2, u_3$ to a single vertex $u$ to obtain a subcubic planar graph $G'$. By the minimality of $G$, $G'$ has a good coloring $f$. We extend $f$ to $G$. 

By Lemma~\ref{only-case}, we only need to consider the case when $f(u_7) = A$, $f(u_6) = B$, and $f(u_4) = C$. By Remark~\ref{neighbor}, we can also assume $f(N(u_7)-u_1) = f(N(u_6)-u_2) = f(N(u_4)-u_3) = \{1,D\}$. If $f(u_5) \neq 1$, then $f(u_5) = D$ and $f(u_8) = f(u_{10}) = 1$. We must have $f(u_9) = 1$ since otherwise we can recolor $u_5$ with $1$ and it contradicts our assumption. Then we claim that $f(N(u_9)) = \{A, D, E\}$ since otherwise we can recolor $u_5$ with $A$ or $E$, color $u_1, u_2, u_3$ with $E,1,D$. We next claim $f(N(u_{10})- u_6) = \{A,D\}$. Otherwise, say $A \notin f(N(u_{10})- u_6)$, then we can recolor $u_6$ to $A$ and color $u_1, u_2, u_3$ with $B,1,E$. Thus,  $D \notin f(N(u_{10})- u_6)$ and we can switch the colors of $u_5$ and $u_6$. Then we can color $u_1, u_2, u_3$ with $B,1,E$, which is a contradiction. Similarly, $f(N(u_8)- u_4) = \{A,D\}$.  We recolor $u_4, u_5, u_6$ to $E,B,C$, and color $u_1, u_2, u_3$ with $B,D,1$. This is a contradiction.

Therefore, $f(u_5) = 1$, $f(u_8) = f(u_{10}) = D$. If we can switch the colors of $u_5$ and $u_6$, then we color $u_1, u_2, u_3$ with $B,1,E$; if we can switch the colors of $u_4$ and $u_5$, then we color $u_1, u_2, u_3$ with $C,1,E$. Thus, $f(N(u_9) - u_5) = \{B,C\}$. Moreover, if $f(u_9) \neq A$ then we recolor $u_5$ with $A$, $u_4$ with $1$, and color $u_1, u_2, u_3$ with $C,1,E$.  Thus, $f(u_9) = A$. We switch the colors of $u_5$ and $u_9$, recolor $u_4$ with $1$, and color $u_1, u_2, u_3$ with $C,1,E$ ($u_7 = u_9$ is possible).
\end{proof}

\begin{lemma}\label{no3|6}
There is no adjacent triangle with a six-cycle.
\end{lemma}

\begin{proof}

Let $u_1, u_2, u_3$ be a triangle and $u_2, u_3, u_4, u_5, u_6, u_7$ be a six cycle sharing the edge $u_2u_3$ with the triangle. By Lemma~\ref{no3|3} and~\ref{no3|4}, $u_1u_4, u_1u_5, u_1u_6, u_1u_7, u_4u_7 \notin E(G)$. Let $N(u_1) = \{u_2, u_3, u_8\}$, $N(u_4) = \{u_3, u_5, u_{9}\}$, $N(u_5) = \{u_4, u_6, u_{10}\}$, $N(u_6) = \{u_5, u_7, u_{11}\}$, $N(u_7) = \{u_2, u_6, u_{12}\}$. We delete $u_1, u_2, u_3$ and add a vertex $u$ with the edges $uu_4, uu_7, uu_8$ to obtain a subcubic planar graph $G'$. By the minimality of $G$, $G'$ has a good coloring $f$.

By Lemma~\ref{only-case}, we only need to consider the case when $f(u_8) = A$, $f(u_7) = B$, and $f(u_4) = C$. We may assume $1 \in f(N(u_4))$, $1 \in f(N(u_7))$ and $1 \in f(N(u_8))$, as in Lemma~\ref{no3|5}. Then $u_1$ has at least one available color from $\{A, B, C, D, E\}$ and thus we can color $u_1$ with a color $x \in \{A, B, C, D, E\}$, say $x = E$ and $f(N(u_8)) = \{1, D, E\}$. Since each of $u_2, u_3$ can be colored by $1$, all of the colors of $\{A, B, C, D, E\}$ must be present in $N^2(u_2)$ and $N^2(u_3)$ respectively. Thus, we may assume that $\{f(u_5), f(u_9)\} = \{f(u_6), f(u_{12})\} = \{1, D\}$. By symmetry, we may assume that $f(u_5) = f(u_{12}) = 1$ and $f(u_6) = f(u_9) = D$. If $1 \notin f(N(u_9))$, then we recolor $u_9$ with $1$, a contradiction with our assumption that $\{f(u_5), f(u_9)\} = \{1, D\}$. Moreover, if $B \notin f(N^2(u_4))$, then we recolor $u_4$ with $B$, color $u_3$ with $C$, and $u_2$ with $1$. Thus, we may assume $B \in f(N^2(u_4))$. Similarly, we may assume $A \in f(N^2(u_4))$. Then we know $E \notin f(N^2(u_4))$ and thus can recolor $u_4$ with $E$, color $u_3$ with $C$, and $u_2$ with $1$.
\end{proof}

\begin{lemma}\label{no3--3}
There are no two triangles connected by a path of length two.
\end{lemma}

\begin{proof}

Let $u_1, u_2, u_3$ be a triangle and $u_5, u_6, u_7$ be another triangle such that $u_3u_4, u_4u_5$ are edges in $G$. By Lemma~\ref{no3|4}, $u_1u_{10}, u_2u_{10} \notin E(G)$. By Lemma~\ref{no3|5}, $u_1u_6, u_1u_7, u_2u_6, u_2u_7 \notin E(G)$. Let $N(u_1) = \{u_2, u_3, u_8\}$, $N(u_2) = \{u_1, u_3, u_9\}$, $N(u_3) = \{u_1, u_2, u_4\}$, $N(u_4) = \{u_3, u_5, u_{10}\}$, $N(u_5) = \{u_4, u_6, u_7\}$, $N(u_6) = \{u_5, u_7, u_{11}\}$, and $N(u_7) = \{u_5, u_6, u_{12}\}$. By Lemma~\ref{no3|6}, $u_8u_{11}, u_8u_{12}, u_9u_{11}, u_9u_{12} \notin E(G)$. We delete $u_1, u_2, u_3$ and add a vertex $u$ with the edges $uu_4, uu_8, uu_9$ to obtain a subcubic planar graph $G'$. By the minimality of $G$, $G'$ has a good coloring $f$. We extend $f$ to $G$.

By Lemma~\ref{no3|5}, we only need to consider the case when $1 \notin \{f(u_8), f(u_9), f(u_{4})\} $.  Say $f(u_4) = A$, $f(u_8) = B$, and $f(u_9) = C$. We may also assume $\{1, D\} \subseteq f(N(u_4))$, $\{1, D\} \subseteq f(N(u_8))$, and $\{1, D\} \subseteq f(N(u_9))$. If $B \notin f(N^2(u_4))$ then we can recolor $u_4$ with $A$, color $u_2$ with $E$, and $u_1$ with $1$. Thus, we may assume $B \in f(N^2(u_4))$. Similarly, we can assume $C \in f(N^2(u_4))$.

 \textbf{Case 1:} $f(u_5) = 1$ and $f(u_{10}) = D$. Then $1 \in \{f(u_6), f(u_7)\}$ since otherwise we can recolor $u_5$ with $1$ and it contradicts our assumption that $\{f(u_5), f(u_{10})\} = \{1, D\}$. We assume $f(u_6) = 1$. If $B \notin f(N^2(u_5))$ then we can recolor $u_5$ with $B$ and it contradicts our assumption. Thus, we may assume $B \in f(N^2(u_5))$. Similarly, we can assume $C, E \in f(N^2(u_5))$. Then $A \notin f(N^2(u_5))$ and we can switch the color of $u_4$ and $u_5$, color $u_2$ with $A$, $u_1$ with $1$, and $u_3$ with $E$.

 \textbf{Case 2:} $f(u_5) = D$ and $f(u_{10}) = 1$. Similarly to Case 1, we may assume $\{A, B, C, E\} \subseteq f(N(u_5))$. Say $f(u_6) = B$, $f(u_7) = C$, $f(u_{11}) = A$, and $f(u_{12}) = E$. Then we switch the colors of $u_5$ and $u_7$, recolor $u_4$ with $1$, color $u_3$ with $A$, $u_2$ with $E$, and $u_1$ with $1$.
\end{proof}

\begin{lemma}\label{no3-4}
There is no triangle and a four-cycle connected by an edge.
\end{lemma}

\begin{proof}
Let $u_1, u_2, u_3$ be a triangle and $u_4, u_5, u_6, u_7$ be a four cycle such that $u_3u_4$ is an edge. By Lemma~\ref{no3|4} and~\ref{no3|5}, $u_1u_5, u_1u_6, u_1u_7, u_2u_5, u_2u_6, u_2u_7 \notin E(G)$. Let $N(u_1) = \{u_2, u_3, u_8\}$, $N(u_2) = \{u_1, u_3, u_9\}$, $N(u_5) = \{u_4, u_6, u_{10}\}$, $N(u_6) = \{u_5, u_7, u_{11}\}$, and $N(u_7) = \{u_4, u_6, u_{12}\}$. We delete $u_1, u_2, u_3$ and add a vertex $u$ with the edges $uu_8, uu_9, uu_4$ to obtain a subcubic planar graph $G'$. By the minimality of $G$, $G'$ has a good coloring $f$. We extend $f$ to $G$.

By Lemma~\ref{no3-3}, we only need to consider the case when $1 \notin \{f(u_4), f(u_8), f(u_9)\} $. Say $f(u_4) = A$, $f(u_8) = B$, and $f(u_9) = C$. We may assume $1, D \in \{f(u_5), f(u_7)\}$, $1, D \in f(N(u_8))$ and $1, D \in f(N(u_9))$. By symmetry, say $f(u_5) = 1$ and $f(u_7) = D$. We may assume $\{B, C\} \subseteq f(N^2(u_4))$ since otherwise say $x \in \{B, C\}$ is missing from $f(N^2(u_4))$ then we can recolor $u_4$ with $x$ and color $u_2$ with $A$, $u_3$ with $1$, and $u_1$ with $E$. Note that $1 \in f(N(u_7))$, since otherwise we recolor $u_7$ with $1$ and it contradicts our assumption that $\{f(u_5), f(u_7)\} = \{1, D\}$. By symmetry and $f(u_5) = 1$, we assume that $f(u_{12}) = 1$, $f(u_6) = B$, and $f(u_{10}) = C$. If $E \notin f(N^2(u_7))$, then we recolor $u_7$ with $E$ and it contradicts our assumption that $\{f(u_5), f(u_7)\} = \{1, D\}$. Thus, $E \in f(N^2(u_7))$. Similarly, $C \in f(N^2(u_7))$. Moreover, if $A \notin f(N^2(u_7))$, then we recolor $u_7$ with $A$ and $u_4$ with $D$, color $u_1, u_2, u_3$ with $A, 1, E$. Thus, $A \in f(N^2(u_7))$.

 \textbf{Case 1:} $f(u_{11}) = A$. Then $1 \in f(N(u_{10}) - u_5)$ since otherwise we switch the colors of $u_6$ and $u_{10}$, which contradicts our assumption that $f(u_{5}) = 1$. We may also assume $E \in f(N(u_{10}) - u_5)$ since otherwise we recolor $u_5$ with $E$ and it contradicts our assumption that $f(u_{5}) = 1$. Thus, $f(N(u_{10}) - u_5) = \{1, E\}$ and we switch the colors of $u_5$ and $u_6$, a contradiction.

 \textbf{Case 2:} $f(u_{11}) = E$. Similarly to Case 1, we assume $1 \in f(N(u_{10}) - u_5)$. We may also assume $A \in f(N(u_{10}) - u_5)$ since otherwise we switch the colors of $u_4$ and $u_5$, and it contradicts our assumption that $f(u_{5}) = 1$. Thus, $f(N(u_{10}) - u_5) = \{1, A\}$ and we switch the colors of $u_5$ and $u_6$, a contradiction.

 \textbf{Case 3:} $f(u_{11}) = C$. Similarly to Case 1 and 2, we assume $f(N(u_{10}) - u_5) = \{A, E\}$ and we switch the colors of $u_5$ and $u_{10}$ to reach a contradiction.
\end{proof}

\begin{lemma}\label{no4|4}
There is no adjacent four-cycle with four-cycle.
\end{lemma}

\begin{proof}
Let $u_1, u_2, u_3, u_4$ be a four cycle and $u_3, u_4, u_5, u_6$ be another four cycle such that they share the edge $u_3u_4$. By Lemma~\ref{no3|4}, $u_1u_5, u_2u_6 \notin E(G)$. By Lemma~\ref{no2cut}, $u_1u_6, u_2u_5 \notin E(G)$. Let $N(u_1) = \{u_2, u_4, u_7\}$, $N(u_2) = \{u_1, u_3, u_8\}$, $N(u_5) = \{u_4, u_6, u_9\}$, $N(u_6) = \{u_3, u_5, u_{10}\}$. By Lemma~\ref{no3|4}, $u_7 \neq u_8$ and $u_9 \neq u_{10}$. We delete $u_1, u_2, u_3, u_4$ and add the edge $u_7u_8$ if it is not an edge to obtain a subcubic planar graph $G'$. By the minimality of $G$, $G'$ has a good coloring $f$. We extend $f$ to $G$.

 \textbf{Case 1:} $1 \in \{f(u_7), f(u_8)\}$ and $1 \in \{f(u_5), f(u_6)\}$. By symmetry, we consider two cases $f(u_7) = f(u_6) = 1$ and $f(u_7) = f(u_5) = 1$.

 \textbf{Case 1.1:} $f(u_7) = f(u_6) = 1$. Color $u_2,u_4$ with $1$. There are at least $1,2$ colors available from $\{A, B, C, D, E\}$ at $u_1$ and at $u_3$ respectively, and thus we can color them in this order. This is a contradiction. 
 

 \textbf{Case 1.2:} $f(u_7) = f(u_5) = 1$. Color $u_3$ with $1$. There is at least $1,2,3$ colors from $\{A, B, C, D, E\}$ at $u_2,u_1$, and $u_4$ respectively. Thus, we can extend $f$ to $G$ by coloring $u_2, u_1, u_4$ in this order. This is a contradiction.

 \textbf{Case 2:} $1 \in \{f(u_7), f(u_8)\}$ and $1 \notin \{f(u_5), f(u_6)\}$. Say $f(u_7) = 1$. We may assume that $f(u_9) = f(u_{10}) = 1$ since otherwise we can recolor $u_5$ or $u_6$ with $1$ and it is solved in Case 1. We color $u_2$ and $u_4$ with $1$.  There is at least one color available from $\{A, B, C, D, E\}$ at $u_1$ and at least two colors available from $\{A, B, C, D, E\}$ at $u_3$. Thus, we can extend $f$ to $G$ by coloring $u_1, u_3$ in this order. This is a contradiction. 

 \textbf{Case 3:} $1 \notin \{f(u_7), f(u_8)\}$ and $1 \in \{f(u_5), f(u_6)\}$. Say $f(u_5) = 1$. Then we may assume $1 \in f(N(u_8))$ since otherwise we can recolor $u_8$ with $1$ and it is solved in Case 1. Similarly, we may assume $1 \in f(N(u_7))$. We color $u_1$ and $u_3$ with $1$. There are at least $1,2$ colors available from $\{A, B, C, D, E\}$ at $u_2$ and $u_4$, and we are done by SDR.

 \textbf{Case 4:} $1 \notin \{f(u_7), f(u_8)\}$ and $1 \notin \{f(u_5), f(u_6)\}$. Similarly to Case 2 and 3, we may assume that $f(u_9) = f(u_{10}) = 1$, $1 \in f(N(u_7))$, and $1 \in f(N(u_8))$. We color $u_1$ and $u_3$ with $1$. There are at least $1,2$ available colors from $\{A, B, C, D, E\}$ at $u_2$ and $u_4$. Thus, we can extend $f$ to $G$. This is a contradiction.
\end{proof}

\begin{lemma}\label{no4|5}
There is no adjacent four-cycle with five-cycle.
\end{lemma}

\begin{proof}
Let $u_1, u_2, u_3, u_4$ be a four cycle and $u_3, u_4, u_5, u_6, u_7$ be a five cycle such that they share the edge $u_3u_4$. By Lemma~\ref{no3|4} and~\ref{no4|4}, $u_1u_5, u_1u_6, u_2u_6, u_2u_7 \notin E(G)$. We further claim that $u_1u_7 \notin E(G)$. Suppose not, $u_1u_7 \in E(G)$, then by Lemma~\ref{nocutedge}, $u_2u_5 \notin E(G)$ (otherwise, $u_6$ is an end of a cut edge). This is a contradiction with Lemma~\ref{nocutedge}, since $u_2$ is an end of a cut edge. By symmetry, $u_2u_5 \notin E(G)$. Let $N(u_1) = \{u_2, u_4, u_8\}$, $N(u_2) = \{u_1, u_3, u_9\}$, $N(u_5) = \{u_4, u_6, u_{10}\}$, $N(u_6) = \{u_5, u_7, u_{11}\}$, and $N(u_7) = \{u_3, u_6, u_{12}\}$. By Lemma~\ref{no4|4}, $u_1, u_5$ have no common neighbour except $u_4$. We also claim $u_1,u_6$ have no common neighbour. Suppose not, i.e., $u_1, u_6$ have a common neighbour $u'$. By planarity, Lemma~\ref{no3|4}, and~\ref{no3|5}, no two of $u_2,u_5, u_7, u'$ are adjacent. We have either the $4$-cycle $u_1u_2u_3u_4u_1$ or the $5$-cycle $u_3u_4u_5u_6u_7u_3$ or the $5$-cycle $u_1u_4u_5u_6u'u_1$ is a separating cycle, which leads to either a cut-edge or a $2$-cut. This is a contradiction to either Lemma~\ref{nocutedge} or Lemma~\ref{no2cut}. Furthermore, $u_1,u_7$ have no common neighbour. Suppose not, i.e. $u_1, u_7$ have a common neighbour $u'$. By planarity, Lemmas~\ref{no3|4}, and~\ref{no3|5}, no two of $u_2,u_5, u_7, u'$ are adjacent. We have either the $4$-cycle $u_1u_2u_3u_4u_1$ or the $5$-cycle $u_1u_2u_3u_7u'u_1$ is a separating cycle, which leads to either a cut-edge or a $2$-cut. This is a contradiction to either Lemma~\ref{nocutedge} or Lemma~\ref{no2cut}. By symmetry, $u_2$ and each of $u_5,u_6, u_7$ have no common neighbour. 
We delete $u_1, u_2, u_3, u_4$ to obtain a subcubic planar graph $G'$. By the minimality of $G$, $G'$ has a good coloring $f$. We extend $f$ to $G$.

 \textbf{Case 1:} $f(u_8) = f(u_9) = f(u_5) = f(u_7) = 1$. Then each of $u_1, u_2, u_3, u_4$ has at least three available colors from from $\{A, B, C, D, E\}$. The only case we cannot extend $f$ to $G$ is when the number of available colors for each of $u_1, u_2, u_3, u_4$ from $\{A, B, C, D, E\}$ are the same and there are exactly three of them, say $\{A, B, C\}$ and $f(N(u_8)) = f(N(u_9)) = f(N(u_5)) = f(N(u_7)) = \{C, D\}$. We may assume $f(u_6) = C$ and $f(u_{10}) = f(u_{12}) = D$. If $A \notin f(N(u_{11}))$, then we can recolor $u_6$ with $A$, which contradict our assumption that $f(u_6) = C$. Thus, $A \in f(N[u_{11}])$. Similarly, we may assume that $\{B, E\} \subseteq f(N[u_{11}])$. But then we can recolor $u_{11}$ with $1$ and at least one color $x \in \{A, B, E\}$ is missing from $f(N[u_{11}])$ thus we can recolor $u_6$ with $x$, which contradicts our assumption that $f(u_6) = C$.

 \textbf{Case 2:} One of $f(u_8), f(u_9), f(u_5), f(u_7)$ is in $\{A, B, C, D, E\}$, say $A$. By symmetry, there are two subcases, i.e., $f(u_9) = A$ and $f(u_5) = A$. We only show the proof for the case when $f(u_9) = A$ since the proof for $f(u_5) = A$ is exactly the same. We may assume $1 \in f(N(u_9))$ since otherwise we can recolor $u_9$ with $1$ and it is solved in Case 1. We color $u_2$ with $1$. Then $u_1,u_3,u_4$ has at least $2,2,3$ available colors from $\{A, B, C, D, E\}$ respectively. Thus, we can extend $f$ to $G$.

 \textbf{Case 3:} Two of $f(u_8), f(u_9), f(u_5), f(u_7)$ are in $\{A, B, C, D, E\}$. By symmetry, we have the following $4$ cases.

 \textbf{Case 3.1:} $f(u_5) = f(u_9) = 1$. Then we color $u_1$ and $u_3$ with $1$. There is at least one color available from $\{A, B, C, D, E\}$  at $u_2$ and at least one color available from $\{A, B, C, D, E\}$ at $u_3$. Thus, we may assume the only available color at $u_2$ and $u_4$ is the same, say $E$. By symmetry, we may assume that $f(u_8) = A, f(u_7) = B, f(u_6) = C$, and $f(u_{10}) = D$. We may assume that $f(u_{12}) = 1$ since otherwise we can recolor $u_7$ with $1$ and it is solved in Case 2. We claim that $A \in f(N[u_{11}])$ since otherwise we can recolor $u_6$ with $A$ and it contradicts $f(u_6) = C$. Similarly, we have $E \in f(N[u_{11}])$. We may also assume $\{A, D\} \subseteq f(N^2(u_7))$ since otherwise we can recolor $u_7$ with the missing color and it contradicts $f(u_7) = B$. Let $N(u_{12}) = \{u_7, u_{13}, u_{14}\}$.

If $f(u_{11}) = A$, then we may assume $1 \in f(N(u_{11}))$ since otherwise we can recolor $u_{11}$ with $1$. Since $E \in f(N(u_{11}))$, we know $f(N(u_{11})) = \{1, C, E\}$. Then since $\{A, D\} \subseteq f(N^2(u_7))$, we assume $f(u_{13}) = D$. Moreover, we may assume $f(u_{14}) = C$ since otherwise we can switch the colors of $u_6$ and $u_7$, color $u_2$ with $B$ and $u_4$ with $E$. We may assume $1 \in f(N(u_{13}))$ since otherwise we can switch the colors of $u_{12}$ and $u_{13}$, color $u_7$ with $1$, which is solved in Case 2. We can also assume that $B \in f(N(u_{13})) \cup f(N(u_{14}))$ since otherwise we can recolor $u_{12}$ with $B$ and $u_7$ with $1$ , which is solved in Case 2. Similarly, we can assume that $\{A, E\} \subseteq f(N(u_{13})) \cup f(N(u_{14}))$. Then we recolor $u_{14}$ with $1$, $u_{12}$ with $C$, $u_6$ with $B$, and $u_7$ with $1$, which is solved in Case 2.

If $f(u_{11}) = E$, then we may assume $1 \in f(N(u_{11}))$ since otherwise we can recolor $u_{11}$ with $1$. Since $A \in f(N(u_{11}))$, we know $f(N(u_{11})) = \{1, A, C\}$. Moreover, we know $f(N(u_{12}) = \{A, B, D\})$ and thus can switch the colors of $u_6$ and $u_7$, color $u_2$ with $B$ and $u_4$ with $E$.

If $f(u_{11}) \notin \{A, E\}$, then since $\{A, E\} \subseteq f(N[u_{11}])$ we have $f(N(u_{11})) = \{A, C, E\}$. We can recolor $u_{11}$ with color $1$. We claim $A \in f(N(u_{12}))$ since otherwise we can recolor $u_7$ with $A$ and color $u_2$ with $B$ and $u_4$ with $E$. Similarly, we assume $D \in f(N(u_{12}))$. Then we can switch the colors of $u_7$ and $u_6$,  color $u_2$ with $B$ and $u_4$ with $E$.

 \textbf{Case 3.2:} $f(u_8) = f(u_5) = 1$. We may assume $1 \in f(N(u_9))$ and $1 \in f(N(u_7))$. Then there are at least two colors available from $\{A, B, C, D, E\}$ at each of $u_1, u_2, u_3, u_4$ and $1$ is also available at $u_2$ and $u_3$. Therefore, the number of available colors from $\{A, B, C, D, E\}$ for each of $u_1, u_2, u_3, u_4$ must be two, and the two colors must be the same for them, say $D, E$. We may thus assume $f(u_9) = A$, $f(u_7) = B$, $f(u_{10}) = A$, $f(u_{6}) = C$, and $\{B, C\} \subseteq f(N(u_8))$. We may also assume that $f(u_{12}) = 1$ since otherwise we can recolor $u_7$ with $1$ and it is solved in Case 2. We assume $A \in f(N^2(u_7))$ since otherwise we recolor $u_7$ with $A$ and color $u_1, u_2, u_3, u_4$ with $D, 1, B, E$. Similarly, we assume $\{D, E\} \subseteq f(N^2(u_7))$. We claim $\{D, E\} \subseteq f(N^2(u_6))$ since otherwise, say $D$ is missing, then we can recolor $u_6$ with $D$ and color $u_1, u_2, u_3, u_4$ with $D, 1, C, E$. 

If $f(u_{11}) = A,$ then $f(u_{12}) = \{B, D, E\}$, $f(N(u_{11})) = \{C, D, E\}$, and we can recolor $u_{11}$ with $1$, which is a contradiction with $A \in f(N^2(u_7))$. Thus, we may assume $f(u_{11}) \in \{D, E\}$, say $f(u_{11}) = D$. We may also assume $1 \in f(N(u_{11}))$ since otherwise we can recolor $u_{11}$ with $1$. Thus,  $f(u_{12}) = \{A, B, E\}$ and $f(N(u_{11})) = \{1, C, E\}$. Then we switch the colors of $u_6$ and $u_7$, and color $u_1, u_2, u_3, u_4$ with $D, 1, E, C$.

 \textbf{Case 3.3:} $f(u_8) = f(u_9) = 1$. We may assume $1 \in f(N(u_5))$ and $1 \in f(N(u_7))$. Similarly to Case 3.2, we may thus assume $f(u_5) = C$, $f(u_7) = A$, $\{A, B\} \subseteq f(N(u_8))$, $\{B, C\} \subseteq f(N(u_9))$.

 \textbf{Case 3.3.1:} $f(u_6) = 1$. Then $f(u_{10}) = f(u_{12}) = B$. We may assume $A \in f(N[u_{11}])$ since otherwise we can switch the colors of $u_6, u_7$, color $u_1, u_2, u_3, u_4$ with $D, A, E, 1$. Similarly, we can assume that $C \in f(N[u_{11}])$. Moreover, $\{D, E\} \subseteq f(N[u_{11}])$ since otherwise, say $D$ is missing, we can recolor $u_6$ with $D$, a contradiction with the case. Then $\{1, A, C, D, E\} \subseteq f(N[u_{11}])$, a contradiction.

 \textbf{Case 3.3.2:} $f(u_6) \neq 1$. Then $f(u_{10}) = f(u_{12}) = 1$. We also have $f(u_{11}) = 1$ since otherwise we recolor $u_6$ with $1$ and it contradicts the case. Similarly to Case 3.3.1, we assume $f(N(u_{11})) = \{B, D, E\}$. We also assume $f(N(u_{12})) = \{A, D, E\}$, since otherwise, say $D$ is missing, we recolor $u_7$ with $D$, and color $u_1, u_2, u_3, u_4$ with $D, A, 1, E$. We switch the colors of $u_6$ and $u_7$, and color $u_1, u_2, u_3, u_4$ with $D, A, 1, E$.

 \textbf{Case 3.4:} $f(u_5) = f(u_7) = 1$. We may assume that $1 \in f(N(u_8))$ and $1 \in f(N(u_9))$. Similarly to Case 3.2, we may assume $f(u_8) = A$, $f(u_9) = B$, $\{1, C\} \subseteq f(N(u_8))$, $\{1, C\} \subseteq f(N(u_9))$, $f(u_6) = C$, $f(u_{10}) = B$, and $f(u_{12}) = A$. We assume $\{D, E\} \subseteq f(N[u_{11}])$ since otherwise, say $D$ is missing, we recolor $u_6$ with $D$ and color $u_1, u_2, u_3, u_4$ with $D, 1, C, E$.

 \textbf{Case 3.4.1:} $f(u_{11}) = 1$. Then $1 \in f(N(u_{12}) - u_7)$ since otherwise we can recolor $u_{12}$ with $1$, $u_7$ with $A$, and color $u_1, u_2, u_3, u_4$ with $1, D, 1, E$. We also have $B \in f(N(u_{12}))$ since otherwise we can recolor $u_7$ with $A$, and color $u_1, u_2, u_3, u_4$ with $1, D, 1, E$. Similarly, we assume $f(N(u_{10}) - u_5) = \{1, A\}$. Let $N(u_{11}) = \{u_6, u_{13}, u_{14}\}$. Since $\{D, E\} \subseteq f(N[u_{11}])$, we assume $f(u_{13}) = D$ and $f(u_{14}) = E$. Moreover, we may assume $\{A, B\} \subseteq f(N(u_{13})) \cup f(N(u_{14}))$ since otherwise, say $A$ is missing, then we can recolor $u_{11}$ with $A$ and it contradicts the case. 

If $A, B$ together in one of $f(N(u_{13}))$ and $f(N(u_{14}))$, say $f(N(u_{13})) = \{A, B\}$, then $f(N(u_{14})) = \{1, D\}$ since if $1$ is missing then we switch the colors of $u_{11}$ and $u_{14}$ and if $D$ is missing then we switch the colors of $u_{11}$ and $u_{13}$, each of which contradicts the case. We recolor $u_5, u_6, u_7, u_{11}$ with $E, 1, D, C$, and color $u_1, u_2, u_3, u_4$ with $D, 1, C, 1$.

If $A, B$ is distributed in $f(N(u_{13}))$ and $f(N(u_{14}))$, say $A \in f(N(u_{13}))$ and $B \in f(N(u_{14}))$. Then either $E \in f(N(u_{13}))$ or $1 \in f(N(u_{14}))$ since otherwise we can switch the colors of $u_{11}$ and $u_{14}$, which contradicts the case. Similarly, either $D \in f(N(u_{14}))$ or $1 \in f(N(u_{13}))$. Then we recolor $u_5, u_6, u_7, u_{11}$ with $E, 1, D, C$, and color $u_1, u_2, u_3, u_4$ with $D, 1, C, 1$.

 \textbf{Case 3.4.2:} $f(u_{11}) \neq 1$. we assume $1 \in f(N(u_{11}))$ since otherwise we recolor $u_{11}$ with $1$ and it is solved in Case 3.4.1. Since $\{D, E\} \subseteq f(N[u_{11}])$, we may assume that $f(u_{11}) = D$ and $f(N(u_{11})) = \{1, C, E\}$. Then similarly to Case 3.4.2, we assume $f(N(u_{10}) - u_5) = \{1, A\}$ and $f(N(u_{12}) - u_7) = \{1, B\}$. Then we recolor $u_5, u_6, u_7$ with $E, 1, C$, and color $u_1, u_2, u_3, u_4$ with $C, 1, D, 1$.

 \textbf{Case 4:} Three of $f(u_8), f(u_9), f(u_5), f(u_7)$ are in $\{A, B, C, D, E\}$.

 \textbf{Case 4.1:} $f(u_5) = 1$. Then we may assume $1 \in f(N(u_8))$, $1 \in f(N(u_9))$, $1 \in f(N(u_7))$, since otherwise, say $1 \notin f(N(u_8))$ we recolor $u_8$ with $1$ and it is solved in Case 3. Then we may color $u_1$ and $u_3$ with $1$. Since each of $u_2, u_4$ has at least one color available from $\{A, B, C, D, E\}$, both of $u_2$ and $u_4$ must have exactly one color available in $\{A, B, C, D, E\}$ and they must be the same color, say $E$. Thus, we may assume $f(u_8) = A$, $f(u_9) = B$, $f(u_7) = C$, $f(N(u_9)) = \{1, D\}$, $f(u_{12}) = 1$, $f(u_6) = D$, $f(u_{10}) = B$. Then $u_1$ has $\{C, D, E\}$ available and only one of them can be in $f(N(u_8))$. Thus, a color $x \in \{C, D\}$ can be used at $u_1$, and we color $u_2, u_3, u_4$ with $1, A, E$.

 \textbf{Case 4.2:} $f(u_8) = 1$. Similarly to Case 4.1, we assume $f(u_9) = A$, $f(u_7) = B$, $f(u_5) = C$, $f(u_{12}) = D$, $f(u_6) = 1$, $f(u_{10}) = A$, $f(N(u_8)) = \{B, D\}$, and $f(N(u_9)) = \{1, C\}$. Similarly to Case 3.3.1, we reach a contradiction.

 \textbf{Case 5:} All of $f(u_8), f(u_9), f(u_5), f(u_7)$ are in $\{A, B, C, D, E\}$. Similarly to Case 4, we may assume $f(u_8) = A$, $f(u_9) = B$, $f(u_7) = C$, $f(u_5) = D$, $f(u_6) = 1$, $f(u_{10}) = B$, $f(u_{12}) = A$, $f(N(u_8)) = \{1, C\}$, and $f(N(u_9)) = \{1, D\}$. Similarly to Case 3.3.1, we reach a contradiction.
\end{proof}

\begin{lemma}\label{no4|6}
There is no adjacent four-cycle with six-cycle.
\end{lemma}

\begin{proof}
Let $u_1, u_2, u_3, u_4$ be a four cycle and $u_3, u_4, u_5, u_6, u_7, u_8$ be a six cycle such that they share the edge $u_3u_4$. By Lemmas~\ref{no3|4},~\ref{no4|4}, and~\ref{no4|5}, $u_1u_5, u_1u_6, u_1u_7, u_2u_6, u_2u_7, u_2u_8 \notin E(G)$. We also know $u_1u_8 \notin E(G)$, since otherwise, say $u_1u_8 \in E(G)$, then either the $4$-cycle $u_1u_2u_3u_4u_1$ or the $4$-cycle $u_1u_2u_3u_8u_1$ is a separating $4$-cycle, which leads to a cut-edge or a $2$-cut. This is a contradiction to either Lemma~\ref{nocutedge} or Lemma~\ref{no2cut}. By symmetry, $u_2u_5 \notin E(G)$. Let $N(u_1) = \{u_2, u_4, u_9\}$, $N(u_2) = \{u_1, u_3, u_{10}\}$, $N(u_5) = \{u_4, u_6, u_{11}\}$, $N(u_6) = \{u_5, u_7, u_{12}\}$, $N(u_7) = \{u_6, u_8, u_{13}\}$, $N(u_8) = \{u_3, u_7, u_{14}\}$. By Lemmas~\ref{no4|4} and~\ref{no4|5}, $u_1,u_5$ have no common neighbour except $u_4$; $u_2, u_8$ have no common neighbour except $u_3$; $u_1, u_6$; $u_2, u_7$ have no common neighbour. We also know $u_1,u_8$ have no common neighbour, since otherwise, say $u_1, u_8$ have a common neighbour $u'$, either the $4$-cycle $u_1u_2u_3u_4u_1$ or the $5$-cycle $u_1u_2u_3u_8u'u_1$ is a separating cycle, which leads to a cut-edge or a $2$-cut. This is a contradiction to either Lemma~\ref{nocutedge} or Lemma~\ref{no2cut}. Furthermore, $u_1, u_7$ have no common neighbour. To see this, suppose $u_1, u_7$ have a common neighbour $u'$ with $N(u') = \{u_1, u_7, u''\}$. We notice that there is always a cut-edge or a $2$-cut among the five edges $u_2u_{10}, u_5u_{11}, u_6u_{12}, u_8u_{14},u'u''$. This is a contradiction to either Lemma~\ref{nocutedge} or Lemma~\ref{no2cut}. By symmetry, $u_2$ and each of $u_5, u_6$ have no common neighbour. We delete $u_1, u_2, u_3, u_4$ to obtain a subcubic planar graph $G'$. By the minimality of $G$, $G'$ has a good coloring $f$. We extend $f$ to $G$.

 \textbf{Case 1:} $f(u_9) = f(u_{10}) = f(u_5) = f(u_8) = 1$. Similarly to Case 1 in lemma~\ref{no4|5}, we assume $f(N(u_9)) = f(N(u_{10})) = f(N(u_5)) = f(N(u_8)) = \{A, B\}$. Specifically, we assume $f(u_7) = f(u_{11}) = A$ and $f(u_6) = f(u_{14}) = B$. Then we have $\{C, D, E\} \subseteq \{f(u_{12}), f(u_{13})\}$ since otherwise, say $C$ is missing, we recolor $u_7$ with $C$, color $u_1, u_2, u_3, u_4$ with $C, E, A, D$.

 \textbf{Case 2:} One of $f(u_9), f(u_{10}), f(u_5), f(u_8)$ is in $\{A, B, C, D, E\}$, say $A$. Similarly to Case 2 in lemma~\ref{no3|5}, $f$ can be extended to $G$.

 \textbf{Case 3:} Two of $f(u_9), f(u_{10}), f(u_5), f(u_8)$ is in $\{A, B, C, D, E\}$. By symmetry, we have the following $4$ cases.

 \textbf{Case 3.1:} $f(u_5) = f(u_{10}) = 1$. Similarly to Case 3.1 in lemma~\ref{no3|5}, we may assume $f(u_9) = A$, $f(N(u_{10})) = \{B, C\}$, $f(u_8) = D$, $f(u_{11}) = B$, $f(u_6) = C$, $1 \in f(N(u_9))$, and $1 \in \{f(u_7), f(u_{14})\}$. If $E \notin f(N^2(u_5))$, then we recolor $u_5$ with $E$, color $u_1, u_2, u_3, u_4$ with $x, E, y, 1$, where $x \in \{B, C, D\} - f(N(u_9))$ and $y \in \{A, B, C\} - \{f(N(u_5)), x\}$. Thus, $E \in f(N^2(u_5))$. Similarly, $\{A, D\} \subseteq f(N^2(u_5))$ and $\{A, D, E\} \subseteq f(N^2(u_{10}))$. Thus, we have either $1 \notin f(N(u_{11}) - u_5)$ and $B \notin f(N(u_6))$ or $1 \notin f(N(u_{6}) - u_5)$ and $C \notin f(N(u_{11}))$, and we can either color $u_5$ with $B$ and $u_{11}$ with $1$ or color $u_5$ with $C$ and $u_6$ with $1$. Similarly, $u_{10}$ can be colored with $B$ or $C$, and $1 \in f(N(u_{10}))$. We must have the new colors of $u_{5}$ and $u_{10}$ are $\{B, C\}$ since if not, say both have the new color $B$, then we color $u_1, u_2, u_3, u_4$ with $x, 1, y, 1$, where $x \in \{C, D, E\} - f(N(u_9))$ and $y \in \{A, C, E\} - \{f(N(u_5)), x\}$. Thus, we must have the new colors of $u_5$ and $u_{10}$ are different, and $f(N(u_9)) = \{1, D\}$ and $f(N(u_8)) = \{1, A\}$. If $f(N(u_{11}) - u_5) \cup f(N(u_6) - u_5) \cap \{1\} = \emptyset$, then we recolor both $u_6, u_{11}$ with $1$, $u_5$ with $B$, $u_1, u_2, u_3, u_4$ with $1, E, 1, C$. Thus, $1 \in f(N(u_{11}) - u_5)$ or $1 \in f(N(u_6) - u_5)$.

 \textbf{Case 3.1.1:} $1 \notin f(N(u_{6}) - u_5)$ and $C \notin f(N(u_{11}))$. Then $1 \in f(N(u_{11}) - u_5)$ and  $f(N(u_6) - u_5) \subseteq \{A, D, E\}$. Then we must have $f(u_7) = A$ and $f(u_{14}) = 1$. Then $f(u_{13}) = 1$ since otherwise we recolor $u_7$ with $1$ and it contradicts our assumption. Moreover, $\{B, E\} \subseteq f(N(u_{13}))$ since otherwise we can recolor $u_7$ with the missing color and it contradicts the case.

 \textbf{Case 3.1.1.1:} $E \in f(N(u_{11}))$ and $f(u_{12}) = D$. Then we know $1 \in f(N(u_{12}))$  since otherwise we can recolor $u_{12}$ with $1$ and $E \in f(N(u_{12}))$ since otherwise we can recolor $u_6$ with $E$, a contradiction with $f(u_6) = C$. Then we recolor $u_6$ with $A$ and $u_7$ with $C$, a contradiction.

 \textbf{Case 3.1.1.2:} $D \in f(N(u_{11}))$ and $f(u_{12}) = E$. Similarly to Case 3.1.1.1, we can assume $f(N(u_{12})) = \{1, C, D\}$ and $f(N(u_{13})) = \{A, B, C\}$. Moreover, $B \in f(N(u_{14}))$ since otherwise we can recolor $u_8$ with $B$, and $A \in f(N(u_{14}))$ since otherwise we can switch the colors of $u_7$ and $u_8$. Then we switch the colors of $u_5$ and $u_6$, and recolor $u_8$ with $C$ to obtain a contradiction.

 \textbf{Case 3.1.2:} $1 \notin f(N(u_{11}))$ and $B \notin f(N(u_6))$. Then $1 \in f(N(u_{6}) - u_5)$. 

 \textbf{Case 3.1.2.1:} $f(u_{12}) = 1$. Then $f(u_7) = A$ and $f(u_{14}) = 1$. Then $f(u_{13}) = 1$ since otherwise we can recolor $u_7$ with $1$ and it contradicts $f(u_7) = A$. Moreover, $\{f(N(u_{13}) - u_7) = \{B, E\}\}$ since otherwise we can recolor $u_7$ with the missing color and it contradicts $f(u_7) = A$. Then $E \in f(N(u_{12}) - u_6)$ since otherwise we can recolor $u_6$ with $E$ and $A \in f(N(u_{12}) - u_6)$ since otherwise we can recolor $u_6$ with $A$ and $u_7$ with $C$, which contradicts $f(u_7) = A$. We claim $B \in f(N(u_{14}) - u_8)$ since otherwise we can recolor $u_8$ with $B$ and $A \in f(N(u_{14}) - u_8)$ since otherwise we can recolor $u_8$ with $A$ and $u_7$ with $D$. Thus, we can recolor $u_{11}$ with $1$, $u_5$ with $B$, $u_6$ with $D$, and $u_8$ with $C$, recolor $u_{10}$ to $B$ or $C$, color $u_1, u_2, u_3, u_4$ with $E, 1, D, 1$.

 \textbf{Case 3.1.2.2:} $f(u_7) = 1$. Then $f(u_{14}) = A$, $1 \in f(N(u_{14}))$, $f(u_{12}) \in \{A, D, E\}$. We also know that $B \in f(N(u_{14}) - u_8)$ or $f(u_{13}) = B$ since otherwise we can recolor $u_8$ with $B$. 

If $f(u_{13}) = B$ then $f(N(u_{13}) - u_7) = \{1, E\}$ since if $1$ is missing then we can recolor $u_7$ with $B$, $u_8$ with $1$, and it is solved in Case 2; if $E$ is missing then we can recolor $u_7$ with $E$ and $u_8$ with $1$, which is again solved in Case 2. Moreover, $f(u_{12}) = D$ since otherwise we can switch the colors of $u_7$ and $u_8$. We can also assume $\{1, E\} \subseteq f(N(u_{12}) - u_6)$ since if $1$ is missing then we recolor $u_{12}$ with $1$ and if $E$ is missing then we recolor $u_6$ with $E$. Then we recolor $u_{11}$ with $1$, $u_5$ with $C$, $u_6$ with $A$, and color $u_1, u_2, u_3, u_4$ with $1, E, 1, B$.

Thus, $f(N(u_{14}) - u_8) = \{1, B\}$ and $f(u_{13}) \in \{A, E\}$. If $f(u_{13}) = E$, then similarly to previous cases we can assume $f(N(u_{13}) - u_7) = \{B, C\}$; we also assume $f(u_{12}) = E$ since otherwise we switch the colors of $u_7$ and $u_{13}$; then we switch the colors of $u_7$ and $u_8$, which is solved in Case 2. If $f(u_{13}) = A$, then we can assume $B \in f(N(u_{13}) - u_7)$; since $E \in N^2(u_7)$, we also have either $f(N(u_{13}) - u_7) = \{1, B\}$ and $f(u_{12}) = E$ or $f(N(u_{13}) - u_7) = \{E, B\}$ and $f(u_{12}) = A$, as otherwise we can switch the color of $u_7$ and $u_{13}$; then we switch the colors of $u_7$ and $u_8$, which is solved in Case 2.



 \textbf{Case 3.2:} $f(u_5) = f(u_{9}) = 1$. Similarly to Case 3.1, we may assume that $f(u_{10}) = A$, $f(N(u_{10})) = \{1, C\}$, $f(N(u_9)) = \{B, C\}$, $f(u_8) = B$, $\{f(u_7), f(u_{14})\} = \{1, C\}$, $\{f(u_6), f(u_{11})\} = \{A, C\}$.

 \textbf{Case 3.2.1:} $f(u_{14}) = 1$ and $f(u_7) = C$. Then $f(u_6) = A$ and $f(u_{11}) = C$. Then $f(u_{13}) = 1$ since otherwise we recolor $u_7$ with $1$ and it contradicts our case. We claim $f(N(u_{14}) - u_8) = \{D, E\}$ since if say $D$ is missing then we recolor $u_8$ with $D$, and color $u_1, u_2, u_3, u_4$ with $D, 1, B, E$. Similarly, $f(N(u_{13}) - u_7) = \{D, E\}$ and $\{D, E\} \subseteq f(N[u_{12}] - u_6)$. Then $f(u_{12}) = B$ since otherwise we switch the colors of $u_7$ and $u_8$, and color $u_1, u_2, u_3, u_4$ with $B, 1, D, E$. Thus, $f(N(u_{12}) - u_{6}) = \{D, E\}$ and we can recolor $u_{12}$ with $1$, which contradicts our assumption that $f(u_{12}) = B$.

 \textbf{Case 3.2.2:} $f(u_{14}) = C$ and $f(u_7) = 1$. Then $1 \in f(N(u_{14}))$ since otherwise we can recolor $u_{14}$ with $1$ and it contradicts the case. Moreover, similarly to Case 3.2.1, $\{D, E\} \subseteq f(N^2(u_8))$ and thus we may assume by symmetry that $f(u_{13}) = E$ and $f(N(u_{14})) = \{1, B, D\}$. 

 \textbf{Case 3.2.2.1:} $f(u_{11}) = A$ and $f(u_6) = C$. We claim $\{A, B, D\} \subseteq f(N^2(u_7)) - u_8$ since say $B$ is missing then we recolor $u_7$ with $B$ and $u_8$ with $1$, which is solved in Case 2. Then we switch the colors of $u_7$ and $u_{13}$, recolor $u_8$ with $1$, which is again solved in Case 2.

 \textbf{Case 3.2.2.2:} $f(u_{11}) = C$ and $f(u_6) = A$. We have $B \in f(N[u_{12}] - u_6)$ since otherwise we switch the colors of $u_6$ and $u_8$, color $u_1, u_2, u_3, u_4$ with $D, 1, B, E$. We also have $D \in f(N[u_{12}] - u_6)$ since otherwise $u_6$ can be recolored to $D$, which contradicts the case. Similarly to Case 3.2.2.1, we assume $\{B, D\} \subseteq f(N^2(u_7)) - u_8$. If $f(u_{12}) \neq B$, then we switch the colors of $u_7$ and $u_{13}$, and recolor $u_8$ with $1$, which is solved in Case 2. Thus, we assume $f(u_{12}) = B$, $f(N(u_{12}) - u_6) = \{1, D\}$, and $f(N(u_{13}) - u_7) = \{1, D\}$. Moreover, we have $\{D, E\} \subseteq f(N^2(u_5))$ since if say $D$ is missing then we recolor $u_5$ with $D$, switch the colors of $u_6$ and $u_7$, and recolor $u_8$ with $1$, which is solved in Case 3.1. Then we switch the colors of $u_5$ and $u_{11}$, and color $u_1, u_2, u_3, u_4$ with $D, 1, E, 1$. 

 \textbf{Case 3.3:} $f(u_9) = f(u_{10}) = 1$. Similarly to Case 3.1, we may assume that $f(N(u_9)) = \{B, C\}$, $f(N(u_{10})) = \{A, C\}$, $f(u_8) = B$, $f(u_5) = A$, $f(N(u_8)) = f(N(u_5)) = \{1, C\}$. By symmetry, we may assume that $f(u_6) = f(u_{14}) = C$ and $f(u_7) = f(u_{11}) = 1$. Moreover, $\{A, D, E\} \subseteq f(N^2(u_8))$ since if $D (E)$ or $A$ is missing, then we recolor $u_8$ with the missing color and color $u_1, u_2, u_3, u_4$ with $D (E), E(D), B, 1$. Thus, $f(N(u_{14}) - u_8) \subseteq \{A, D, E\}$ and we can recolor $u_{14}$ with $1$, and color $u_1, u_2, u_3, u_4$ with $D, E, C, 1$.



 \textbf{Case 3.4:} $f(u_5) = f(u_{8}) = 1$. Similarly to Case 3.1, we may assume that $f(u_9) = A$, $f(u_{10}) = B$, $f(N(u_9)) = f(N(u_{10})) = \{1, C\}$, and by symmetry $f(u_{11}) = f(u_7) = A$ and $f(u_6) = f(u_{14}) = C$. We also assume $\{B, D, E\} \subseteq f(N^2(u_7))$ since otherwise we can recolor $u_7$ with $B (D, E)$ and color $u_1, u_2, u_3, u_4$ with $D, 1, A, E$. Similarly, $\{B, D, E\} \subseteq f(N^2(u_6))$. If one of $u_{13}$ or $u_{12}$ is not colored by $1$, say $u_{13}$, then we have $f(N[u_{12}]) = \{B, C, D, E\}$ and thus we can recolor $u_{12}$ with $1$, which contradicts $\{B, D, E\} \subseteq f(N^2(u_6))$. Thus, we may by symmetry assume $f(u_{12}) = D$, $f(u_{13}) = E$, and $f(N(u_{12}) - u_6) = f(N(u_{13}) - u_7) = \{1, B\}$. Moreover, we have $f(N(u_{14}) - u_8) = \{A, B\}$ since if $A (B)$ is missing then we recolor $u_8$ with $A(B)$, $u_7$ with $1$, and color $u_1, u_2, u_3, u_4$ with $1, D, 1, E$. Similarly, we can assume $f(N(u_{11}) - u_5) = \{B, C\}$. Then we recolor $u_{11}$ with $1$, $u_5$ with $A$, $u_7$ with $1$, $u_8$ with $D$, and color $u_1, u_2, u_3, u_4$ with $D, 1, E, 1$.

 \textbf{Case 4:} Three of $f(u_8), f(u_9), f(u_5), f(u_7)$ are in $\{A, B, C, D, E\}$. By symmetry, there are two cases, i.e., $f(u_9) = 1$ or $f(u_8) = 1$.

 \textbf{Case 4.1:} $f(u_9) = 1$. Similarly to Case 3, we may assume $f(u_5) = A$, $f(u_8) = B$, $f(u_{10}) = C$, $f(N(u_9)) = \{B, D\}$, $f(N(u_8)) = \{1, D\}$, and $1 \in f(N(u_5))$. Then $\{A, C, E\} \subseteq f(N^2(u_8))$, since if $A(C, E)$ is missing then we can recolor $u_8$ with $A(C, E)$ and color $u_1, u_2, u_3, u_4$ with $E, 1, B, 1$. Similarly, $B \in f(N^2(u_5))$.

 \textbf{Case 4.1.1:} $f(u_7) = 1$ and $f(u_{14}) = D$.

 \textbf{Case 4.1.1.1:} $f(u_{13}) = A$. If $f(u_6) = C$ and $f(N(u_{14}) - u_8) = \{1, E\}$, then $f(u_{11}) = 1$. Similarly to Case 3.4, we may assume $\{B, D, E\} \subset f(N^2[u_6] - u_8)$, but then we can recolor $u_{12}$ with $1$, a contradiction. Thus, we may assume $f(u_6) = E$ and $f(N(u_{14}) - u_8) = \{1, C\}$. Similarly, $\{B, D, C\} \subset f(N^2[u_6] - u_8)$, but then we can recolor $u_{12}$ with $1$, a contradiction.

 \textbf{Case 4.1.1.2:} $f(N(u_{14})) = \{1, A, B\}$. Then $\{f(u_6), f(u_{13})\} = \{C, E\}$. 

If $f(u_6) = C$ and $f(u_{13}) = E$, then $f(u_{11}) = 1$. Similarly to Case 4.1.1, we may assume $\{B, D, E\} \subseteq f(N^2(u_6))$. We claim $B \in f(N^2(u_7) - u_8)$ since otherwise we can switch the colors of $u_7$ and $u_8$, and either $f(u_{12}) = E$ or $1 \in f(N(u_{13}) - u_7)$ since otherwise we can switch the colors of $u_7$ and $u_{13}$ then recolor $u_8$ with $1$, both are solved in Case 3. Thus if $f(u_{12}) = D$, then $f(N(u_{13}) - u_7) = \{1, B\}$ and we switch the colors of $u_6$ and $u_7$, and recolor $u_8$ with $1$, which is solved in Case 3. If $f(u_{12}) = B$, then $f(N(u_{13}) - u_7) = \{1, C\} $, $f(N(u_{12}) - u_6) = \{1, D\}$, and since one of the colors $x \in \{C, D, E\}$ is missing in $f(N(u_{11}) - u_5)$ we can recolor $u_5$ with $x$, $u_6$ with $1$, $u_7$ with $A$, and $u_8$ with $1$, which is solved in Case 3. If $f(u_{12}) \notin \{B, D\}$, then $f(N(u_{12}) - u_6) = \{B, D\}$, we recolor $u_{12}$ with $1$; again we assume that $f(N(u_{13}) - u_7) = \{1, B\}$ and since one of the colors $x \in \{B, D, E\}$ is missing in $f(N(u_{11}) - u_5)$ we can recolor $u_5$ with $x$, $u_7$ with $A$, and $u_8$ with $1$, which is solved in Case 3.

If $f(u_6) = E$, then $f(u_{11}) = 1$. We have either $f(u_{12}) = 1$ or $E \in f(N(u_{13}) - u_7)$ since otherwise we can switch the colors of $u_6$ and $u_7$, and recolor $u_8$ with $1$. If $f(u_{12}) = 1$, then we have $f(N(u_{13}) - u_7) = \{1, B\}$ since otherwise we can either switch the colors of $u_7$ and $u_{13}$ (then recolor $u_8$ with $1$) or $u_7$ and $u_8$. Then since one of the colors $x \in \{B, C, D\}$ is missing in $f(N(u_{11}) - u_5)$ we can recolor $u_5$ with $x$, $u_7$ with $A$, and $u_8$ with $1$, which is solved in Case 3. Thus, $E \in f(N(u_{13}) - u_7)$ and $f(u_{12}) \neq 1$. Similarly, we claim then $f(u_{12}) = C$ or $f(N(u_{13}) - u_7) = \{1, E\}$ and $f(u_{12}) = B$ or $f(N(u_{13}) - u_7) = \{B, E\}$. Then, one of the colors $x \in \{B (C), D, E\}$ is missing in $f(N(u_{11}) - u_5)$ we can recolor $u_5$ with $x$, $u_6$ with $1$, $u_7$ with $A$, and $u_8$ with $1$, which is solved in Case 3.

 \textbf{Case 4.1.2:} $f(u_7) = D$ and $f(u_{14}) = 1$. Then $1 \in \{f(u_6), f(u_{13})\}$. Similarly to Case 4.1.1, $\{A, C, E\} \subseteq f(N^2(u_8))$.

 \textbf{Case 4.1.2.1:} $f(u_{13}) = 1$. Then $f(u_6) \in \{C, E\}$ and $f(u_{12}) = 1$. Then $f(N(u_{11})) = \{B, C(E)\}$ and $f(N(u_{13}) - u_7) = \{B, C (E)\}$. If $f(u_6) = C (E)$ and $f(N(u_{14}) - u_8) = \{A, E (C)\}$, then $f(N(u_{13}) - u_7) = \{B, E (C)\}$. Similarly to Case 4.1.1, $f(N(u_{11}) - u_5) = \{B, E(C)\}$. We recolor $u_5,u_7,u_8$ with $D,A,1$, which is solved in Case 3.

 \textbf{Case 4.1.2.2:} $f(u_6) = 1$. Similarly to Case 4.1.2.1, $\{A, C, E\} \subseteq f(N^2(u_8))$. 

If $f(u_{13}) = A$, then $1 \in f(N(u_{13}) - u_7)$ and since $\{C, E\} \subseteq f(N^2(u_7))$ we can switch the colors of $u_7$ and $u_8$, we assume $f(u_{12}) = C(E)$ and $f(N(u_{13}) - u_7) = 1, E(C)$. We may assume $D \in f(u_{11}) \cup (f(N(u_{12}) - u_6))$, $B \in  f(u_{11}) \cup (f(N(u_{12}) - u_6))$, $1 \in (f(N(u_{12}) - u_6))$ or $f(u_{11}) = C(E)$. Thus, if $f(u_{11}) = C(E)$, then $(f(N(u_{12}) - u_6)) = D, B$. Then we change the color of $u_5$ to $1$ and recolor $u_6$ with $E(C)$, which is solved in Case 3. Therefore, $f(N(u_{12}) - u_6) = \{1, D\}$ and $f(u_{11}) = B$, and again we can change the color of $u_5$ to $1$ and recolor $u_6$ with $E (C)$, which is solved in Case 3.

Thus, $f(u_{13}) \in \{C, E\}$, say $C (E)$ (the proof for $E$ is symmetric). Similarly, we have $f(N(u_{14}) - u_8)=\{A, E(C)\}$, $1 \in f(N(u_{13}) - u_7)$, $E(C) \in f(u_{12}) \cup f(N(u_{13}) - u_7)$, $f(u_{11}) = D$ or $D \in f(N[u_{12}])$, $B \in f(N[u_{12}]) \cup f(u_{11})$. 

$\bullet$ $f(u_{12}) = E (C)$. Then either $f(u_{11}) = 1$ and $f(N(u_{12}) - u_6) = \{B, D\}$ or $f(u_{11}) \neq 1$ and $f(N(u_{12}) - u_6) \in \{A, B, D\}$; but then we can switch the colors of $u_6$ and $u_{12}$, color $u_7$ with $1$, and color $u_1, u_2, u_3, u_4$ with $E, 1, D, 1$.

$\bullet$ $f(u_{12}) = B$. Then either $f(u_{11}) = 1$ or $A \in f(N(u_{12})- u_6)$. We also know either $f(u_{11}) = D$ or $D \in f(N(u_{12})- u_6)$ and either $f(u_{11}) = B$ or $1 \in f(N(u_{12})- u_6)$. Thus, the colors of $u_{11}$ and $f(N(u_{12}) - u_6)$ are either $B, A, D$, $D, A, 1$, or $1, D, 1$, and we can thus recolor $u_6$ with $E (C)$ and $u_7$ with $1$, a contradiction.

$\bullet$ $f(u_{12}) = C (E)$. Then again we have either $f(u_{11}) = D$ or $D \in f(N(u_{12}) - u_6)$, either $f(u_{11}) = 1$ or $A \in f(N(u_{12}) - u_6)$, and either $f(u_{11}) = B$ or $B \in f(N(u_{12}) - u_6)$. Thus, the colors of $u_{11}$ and $f(N(u_{12}) - u_6)$ are either $B, A, D$, $D, A, B$, or $1, D, B$, and we can thus recolor $u_6$ with $E (C)$ and $u_7$ with $1$, a contradiction.

 \textbf{Case 4.2:} $f(u_8) = 1$. Similarly to Case 4.1, we can assume that $f(u_5) = A$, $f(u_9) = B$, $f(u_{10}) = C$, $1 \in f(u_{10})$, $1 \in f(N(u_5))$, $\{f(u_7), f(u_{14})\} = \{B, D\}$, and $f(N(u_9)) = \{1, D\}$. We may assume $f(u_7) = B$ and $f(u_{14}) = D$ (the proof for $f(u_7) = D$ and $f(u_{14}) = B$ is similar). We claim $\{A, C\} \subseteq f(N^2(u_8))$ since if $A(C)$ is missing then we recolor $u_8$ with $A(C)$ and color $u_1, u_3 $ with $1$. Since $\{C, D, E\}$ are available at $u_4$ and $D, E$ are available at $u_2$, and we also know $1 \in f(N(u_5)) \cap f(N(u_{10}))$ thus at most one colors can be deleted from the available colors of each of $u_2$ and $u_4$, there are at least two colors available at $u_4$ and at least one color available at $u_2$, and thus we can extend the coloring to $G$. For the same reason, we have either $1 \in \{f(u_6), f(u_{13})\}$ or $B \in f(N(u_{14}) - u_8)$.

 \textbf{Case 4.2.1:} $f(u_{11}) = 1$ and $f(u_{6}) \neq 1$. Then $f(u_6) \in \{C, D, E\}$ and we assume $f(u_{12}) = f(u_{13}) = 1$. 

 \textbf{Case 4.2.1.1:} $f(u_{6}) = D$. We have $A \in f(N^2(u_8))$ since otherwise we recolor $u_8$ with $A$ and color $u_1, u_2, u_3, u_4$ with $1, D \text{ or } E, 1, C$. We have $C \in f(N^2(u_8))$ since otherwise we recolor $u_8$ with $C$ and color $u_1, u_2, u_3, u_4$ with $1, A \text{ or } D, 1, E$. Then $f(N(u_{14}) - u_8) = \{A, C\}$ and we recolor $u_8$ with $E$, color $u_1, u_2, u_3, u_4$ with $1, D \text{ or } A, 1, C$.


 \textbf{Case 4.2.1.2:} $f(u_{6}) = E$. Similarly to Case 4.2.2.1, $f(N(u_{14}) - u_8) = \{A, C\}$. We switch the colors of $u_8$ and $u_{14}$, and color $u_1, u_2, u_3, u_4$ with $1, E \text{ or }A, 1, C$.

 \textbf{Case 4.2.1.3:} $f(u_{6}) = C$. Then $f(N(u_{12}) - u_6) = \{D, E\}$ since otherwise we can recolor $u_6$ with $D$ or $E$, which is solved in Case 4.2.2.1 and Case 4.2.2.2. Similarly, $A \in f(N(u_{14}) - u_8)$ and thus we can either recolor $u_8$ with $D$ or $E$. We also claim $D \in f(N(u_{11}) - u_5) $ since if $D$ is missing then we recolor $u_5$ with $D$, and color $u_1, u_2, u_3, u_4$ with $E, 1, A, 1$. We know $E \in f(N(u_{13}) - u_7)$ and claim $f(N(u_{13}) - u_7) = \{E, C\}$ since otherwise we can switch the colors of $u_6$ and $u_7$, and color $u_1, u_2, u_3, u_4$ with $E, 1, B, 1$. Moreover, we have $B \in f(N(u_{11}) - u_5)$ since otherwise we can recolor $u_5$ with $B$, $u_7$ with $A$, and color $u_1, u_2, u_3, u_4$ with $E, 1, B, 1$. Then we recolor $u_5$ with $E$, $u_7$ with $A$, and color $u_1, u_2, u_3, u_4$ with $E, 1, B, 1$.

 \textbf{Case 4.2.2:} $f(u_6) = 1$. 

 \textbf{Case 4.2.2.1:} $f(N(u_{14}) - u_8) = \{A, C\}$. Then we must have $f(u_{13}) = E$ and $1 \in f(N(u_{13}) - u_7)$ since otherwise we can recolor $u_8$ with $E$ and $u_{14}$ with $1$, and color $u_1, u_2, u_3, u_4$ with $E, 1, D, 1$. We also assume $f(u_{11}) = C$ and $f(N(u_{10})) = \{1, A\}$ since otherwise we can switch the colors of $u_8$ and $u_{14}$, color $u_1, u_2, u_3, u_4$ with $1,A,1,E$ or $1, E, 1, C$. We know $A \in f(N(u_{12}) - u_6)$ since otherwise we can color $u_6$ with the missing color and recolor $u_5$ with $1$, which is solved in Case 3. Moreover, we assume $B \in f(N(u_{12}) - u_6)$ since otherwise we switch the colors of $u_6$ and $u_7$, switch the colors of $u_8$ and $u_{14}$, and color $u_1, u_2, u_3, u_4$ with $E, 1, B, 1$. Thus, $f(N(u_{12}) - u_6) = \{A, B\}$. Then we recolor $u_{12}$ with $1$, $u_6$ with $D$, and $u_5$ with $1$, which is solved in Case 3.

 \textbf{Case 4.2.2.2:} $f(u_{13}) = A$. Then $f(N(u_{14}) - u_8) = \{1, C\}$ or $\{C, E\}$ since otherwise we recolor $u_8$ with $E$, $u_{14}$ with $1$, and color $u_1, u_2, u_3, u_4$ with $E, 1, D, 1$. We claim $f(N(u_{13}) - u_7) = \{1, C(E)\}$ and $f(u_{12}) = E(C)$. Then since $u_8$ can be replaced by $D$ (and $u_{14}$ with $1$) or $E$, we can switch the colors of $u_6$ and $u_7$, and color $u_1, u_2, u_3, u_4$ with $E, 1, B, 1$, unless $B \in f(u_{11}) \cup f(N(u_{12}) - u_6)$. We also have $C \in f(u_{11}) \cup f(N(u_{12}) - u_6)$. Moreover, $1 \in f(N(u_{12}) - u_6) $ or $f(u_{11}) = E$. Then we have the colors of $u_{11}$ and $f(N(u_{12}) - u_6)$ is $E, B, C$ or $C, B, 1$ or $B, C, 1$, and thus we can recolor $u_6$ with $D$ and $u_5$ with $1$, which is solved in Case 3.

 \textbf{Case 4.2.2.3:} $f(u_{13}) = C$. Then $f(N(u_{14}) - u_8) = \{1, A\}$ or $\{A, E\}$. Similarly, we have $f(u_{12}) = E$ or $E \in f(N(u_{13}) - u_7)$, $f(u_{11}) = 1$ or $A \in f(N[u_{12}])$, $f(u_{11}) = B$ or $B \in f(N(u_{12}) - u_6)$. If $f(u_{12}) = E$ and $f(u_{11}) = 1$, then $f(N(u_{12}) - u_6) = \{1, B\}$ and thus we can recolor $u_6$ with $D$ and $u_5$ with $1$, which is solved in Case 3. If $f(u_{12}) = E$ but $f(u_{11}) \neq 1$, then $f(u_{11}) \in \{B, E\}$ and $f(N(u_{12}) - u_6) = \{1(B), A\}$ and again we can recolor $u_6$ with $D$ and $u_5$ with $1$, which is solved in Case 3. Thus, we may assume $f(u_{12}) \neq E$ and $f(N(u_{13}) - u_7) = \{1, E\}$. Then $f(u_{12}) \in \{C, D\}$. Say $f(u_{12}) = D (C)$. If $f(u_{11}) = 1$ then $f(N(u_{12}) - u_6) = \{B, E\}$ and we can recolor $u_6$ with $D$, which is solved in Case 4.2.1. If $f(u_{11}) \neq 1$, then $f(N(u_{12}) - u_6) = \{A, B (E)\}$ and $f(u_{11}) = E(B)$, and thus we can recolor $u_6$ with $D$ and $u_{12}$ with $1$, and recolor $u_5$ with $1$, which is solved in Case 3.

 \textbf{Case 5:} All of $f(u_8), f(u_9), f(u_5), f(u_7)$ are in $\{A, B, C, D, E\}$. Similarly to Case 3, we may assume $f(u_9) = A$, $f(u_{10}) = B$, $f(u_5) = C$, $f(u_8) = D$, $f(N(u_9)) = \{1, D\}$, $f(N(u_{10})) = \{1, C\}$, $f(N(u_5)) = \{1, B\}$, and $f(N(u_8)) = \{1, A\}$. By symmetry, we have the following two cases.

 \textbf{Case 5.1:} $f(u_{11}) = f(u_{14}) = 1$, $f(u_{6}) = B$, and $f(u_{7}) = A$. Then $f(u_{12}) = f(u_{13}) = 1$. We can also assume $f(N(u_{14}) - u_8) = \{C, E\}$ since if $C (E)$ is missing then we recolor $u_8$ with $C (E)$ and color $u_1, u_2, u_3, u_4$ with $E, 1, D, 1$. Similarly, we can assume that $f(N(u_{11}) - u_5) = \{D, E\}$ and $E \in f(N(u_{13}) - u_7) \cap f(N(u_{12}) - u_6)$. Moreover, $D \in f(N(u_{13}) - u_7)$ since otherwise we can switch the colors of $u_7$ and $u_8$. Similarly, we can assume $C \in f(N(u_{12}) - u_6)$. Then we switch the colors of $u_6$ and $u_7$, and color $u_1, u_2, u_3, u_4$ with $1, E, 1, B$.

 \textbf{Case 5.2:} $f(u_{11}) = f(u_{7}) = 1$, $f(u_{6}) = B$, and $f(u_{14}) = A$. Similarly to Case 5.1, we may assume $1 \in f(N(u_{14}))$ and $\{C, E\} \subseteq \{f(N^2(u_8))\}$.

 \textbf{Case 5.2.1:} $f(u_{13}) = C$. Similarly to Case 5.1 and cases before, we assume $\{A, D, E\} \subseteq f(N^2(u_5))$ and $E \in f(N(u_{13})-u_7)$. If $f(u_{12}) \in \{A, D\}$ then $B \in f(N(u_{13})-u_7)$ since otherwise we switch the colors of $u_6$ and $u_7$, and recolor $u_8$ with $1$, which is solved in Case 4. We switch the colors of $u_{13}$ and $u_7$ and recolor $u_8$ with $1$, which is solved in Case 4, a contradiction. Thus, we may assume $f(u_{12}) = E$. Similarly, we assume $f(N(u_{13}) - u_7) = \{D, B\}$ and switch the colors of $u_{13}$ and $u_7$ and recolor $u_8$ with $1$, which is solved in Case 4, a contradiction.

 \textbf{Case 5.2.2:} $f(u_{13}) = E$. Similarly to Case 5.2.1, if $f(u_{12}) \in \{A, D\}$, then $f(N(u_{13}) - u_7) = \{1, B\}$ and we recolor $u_5, u_6, u_7, u_8$ to $B, 1, C, 1$, which is solved in Case 4. If $f(u_{12}) = E$, then $f(N(u_{13}) - u_7) = \{B, D\}$ and we recolor $u_5, u_6, u_7, u_8$ to $B, 1, C, 1$, which is solved in Case 4.
\end{proof}

\begin{lemma}\label{no374}
There are no configurations in Class One.
\end{lemma}

\begin{proof}
The proof for configurations in Class One, i.e., ``$3|7|4$'', ``$3|7|5$-I'', ``$3|7|5$-II'', ``$7|4|4|4$'', ``$7|4|5|4$'', ``$7|4|4|5$'', and ``$7|4|5|5|5$'', is done by computer (see Appendix). 
To illustrate the proof idea, we give the proof without a computer for the non-existence of the configuration ``$3|7|4$'' in the Appendix.
\end{proof}

\begin{lemma}\label{no555-I}
There are no configurations in Class Two. 
\end{lemma}

\begin{proof}
The proof for configurations in Class Two, i.e., ``$5|5|5$-I'', ```$5|5|5$-II'', ``$6|5|5$-I'', ``$6|5|5$-II'', ``$6|5|6$-I'', and ``$6|5|6$-II'', is done by computer (see Appendix). 
To illustrate the proof idea, we give the proof without a computer for the non-existence of the configuration  ``$5|5|5$-I'' in the Appendix.
\end{proof}

\begin{lemma}\label{no353}
There is no configuration ``$3$-$5$-$3$''. 
\end{lemma}

\begin{proof}
The proof is shown in the Appendix. We also checked it using a computer.   
\end{proof}

\section{Proof of Theorem \ref{planar}}

Suppose that Theorem~\ref{planar} is false, i.e., there are subcubic planar graphs that are not packing $(1,2,2,2,2,2)$-colorable. 
Let $G$ be a counterexample with smallest $|V(G)|$. 
Note that $G$ is a connected planar graph.
We proceed by the discharging method. 
According to Euler's formula for connected planar graphs, we use face charging:
\begin{equation}\label{facecharging}
\sum\limits_{v \in V(G)} (2d(v)-6) + \sum\limits_{f \in F(G)} (\ell(f)-6) = -12.
\end{equation}

Let the initial charge of each vertex in $G$ be $Ch(v) = 2d(v) - 6$ and each face in $G$ be $Ch(f) = \ell(f)-6$. Our goal is to redistribute the charges so that the final charge $Ch^*(v)$ for each vertex $v$ and $Ch^*(f)$ for each face $f$ satisfy $Ch^*(v) \ge 0$ and $Ch^*(f) \ge 0$. This will give us $$ \sum\limits_{v \in V(G)} Ch^*(v) + \sum\limits_{f \in F(G)} Ch^*(f) \ge 0,$$ which is a contradiction with Equation~\eqref{facecharging}.

We call a vertex $v$ or a face $f$ \textit{happy} if $Ch^*(v) \ge 0$ or $Ch^*(f) \ge 0$ respectively. Now we define the rules for redistributing the charges.  

\textbf{(i):} Each $3$-face receives charge $1$ from each of its adjacent faces of length at least seven.

\textbf{(ii):} Each $4$-face receives charge $\frac{1}{2}$ from each of its adjacent faces of length at least seven.

\textbf{(iii):} Each $5$-face receives charge $\frac{1}{4}$ from each of its adjacent faces of length at least seven.

We show all the vertices and faces are happy by the following claims.

\begin{claim}\label{6-face}
All the vertices and faces of length at most $6$ are happy.     
\end{claim}

\begin{proof}
After applying rules \textbf{(i),(ii)}, and \textbf{(iii)}, all the vertices are happy since $G$ is a cubic graph by Lemma \ref{cubic} and the charges on vertices are unchanged by the redistributing rules.
By Lemmas~\ref{no3|3},~\ref{no3|4},~\ref{no3|5}, and~\ref{no3|6}, a $3$-face must be adjacent to three faces of length at least $7$ and therefore receives charge $1$ from each adjacent face. The final charge of a $3$-face is $3-6+3 \cdot 1 = 0$ and thus every $3$-face is happy. By Lemmas~\ref{no3|4},~\ref{no4|4},~\ref{no4|5}, and~\ref{no4|6}, a $4$-face must be adjacent to four faces of length at least $7$ and therefore receives charge $\frac{1}{2}$ from each adjacent face. The final charge is $4-6+4 \cdot \frac{1}{2} = 0$ and thus every $4$-face is happy. For a $5$-face, by Lemmas~\ref{no3|5} and~\ref{no4|5}, it is not adjacent to $3$-faces and $4$-faces. Furthermore, by Lemma~\ref{no555-I}, a $5$-face is adjacent to at most one face of length at most $6$, and thus it receives charge $\frac{1}{4}$ from each of at least four adjacent faces of length $7$ or more. The final charge is at least $5-6+4 \cdot \frac{1}{4} = 0$ and thus every $5$-face is happy. All the $6$-faces are happy since a $6$-face does not give any charge to adjacent faces.  
\end{proof}

\begin{claim}\label{7face}
All the $7$-faces are happy.    
\end{claim}

\begin{proof}
For a $7$-face $f$, it can be adjacent to at most one $3$-face by Lemma~\ref{no3|3},~\ref{no3-3}, and~\ref{no3--3}. 

\textbf{Case 1:} $f$ is adjacent to exactly one $3$-face. 
Furthermore, by Lemmas~\ref{no3|4},~\ref{no3|5},~\ref{no3-4}, and~\ref{no374}, it cannot be adjacent to any other faces of length at most $5$. Therefore, $f$ is happy since the final charge for $f$ is $7-6-1 = 0$. 

\textbf{Case 2:} $f$ is not adjacent to any $3$-faces. 

\textbf{Case 2.1:} $f$ is adjacent to at least one $4$-face. By Lemmas~\ref{no4|4} and \ref{no374}, $f$ can be adjacent to at most two $4$-faces. If $f$ is adjacent to two $4$-faces, then by Lemmas~\ref{no4|4},~\ref{no4|5} and~\ref{no374}, the other faces that $f$ is adjacent to are of length at least $6$. Thus, the final charge for $f$ is $7-6-2 \cdot \frac{1}{2} = 0$ and $f$ is happy. So $f$ is adjacent to one $4$-face. By Lemmas~\ref{no4|5},~\ref{no555-I}, and~\ref{no374}, $f$ can be adjacent to at most two faces of length $5$. Thus, the final charge of $f$ is at least $7-6-\frac{1}{2} - 2 \cdot \frac{1}{4} = 0$ and $f$ is happy. 

\textbf{Case 2.2:} $f$ is not adjacent any $3$-faces or $4$-faces. By Lemma~\ref{no555-I}, $f$ can be adjacent to at most four $5$-faces. Therefore, the final charge of $f$ is at least $7-6- 4 \cdot \frac{1}{4} = 0$ and $f$ is happy.  
\end{proof}

\begin{claim}\label{8face}
All the $8$-faces are happy.    
\end{claim}

\begin{proof}
For a $8$-face $f$, by Lemmas~\ref{no3|3},~\ref{no3-3} and~\ref{no3--3}, $f$ can be adjacent to at most two $3$-faces. 

\textbf{Case 1:} $f$ is adjacent to exactly two $3$-faces. Then they must be in opposite positions. Furthermore, by 
Lemmas~\ref{no3|4},~\ref{no3|5},~\ref{no3-4} and~\ref{no353}, the other faces that $f$ is adjacent to are of length at least $6$. Thus, the final charge of $f$ is $8-6-2 \cdot 1 = 0$.

\textbf{Case 2:} $f$ is adjacent to exactly one $3$-face. By Lemmas~\ref{no3|4},~\ref{no3-4} and~\ref{no4|4}, $f$ can be adjacent to at most two $4$-faces. If $f$ is adjacent to two $4$-faces, then it cannot be adjacent to any $5$-faces by Lemma~\ref{no4|5}. Thus, the final charge of $f$ is $8-6-1-2 \cdot \frac{1}{2} = 0$ and $f$ is happy. If $f$ is adjacent to one additional $4$-face,  then by Lemma~\ref{no4|5} again, it can be adjacent to at most two $5$-faces. Thus, the final charge of $f$ is at least $8-6-1-\frac{1}{2}-2 \cdot \frac{1}{4} = 0$ and $f$ is happy. Thus, $f$ is not adjacent to any $4$-face. By Lemma~\ref{no555-I}, it can be adjacent to at most four $5$-faces. Thus, the final charge of $f$ is at least $8-6-1-4 \cdot \frac{1}{4} = 0$ and $f$ is happy.

\textbf{Case 3:} $f$ is not adjacent to any $3$-face. Then $f$ can be adjacent to at most four $4$-faces by Lemma~\ref{no4|4}. Let $y$ be the number of $4$-faces that $f$ is adjacent to and $z$ be the number of $5$-faces that $f$ is adjacent to. By Lemma~\ref{no4|4},
we have $0 \le y \le 4$. By Lemma~\ref{no555-I}, we have $0 \le z \le 5$. By Lemma~\ref{no4|5}, we have $2y + \frac{3}{2} z \le 8$. Thus, the final charge is at least $$8-6-y \cdot \frac{1}{2} - z \cdot \frac{1}{4} = 2 - \frac{1}{4}(2y+z) \ge 2 - \frac{1}{4}(2y+\frac{3}{2}z) = 0,$$
and $f$ is happy.
\end{proof}

\begin{claim}\label{9face}
All the $9$-faces are happy.    
\end{claim}

\begin{proof}
For a $9$-face $f$, by Lemmas~\ref{no3|3},~\ref{no3-3} and~\ref{no3--3}, $f$ can be adjacent to at most two $3$-faces. 

\textbf{Case 1:} $f$ is adjacent to exactly two $3$-faces. The two $3$-faces must have distance exactly $3$. By Lemmas~\ref{no3|4} and~\ref{no3-4}, $f$ cannot be adjacent to any $4$-face. By Lemmas~\ref{no3|5} and~\ref{no353}, $f$ can be adjacent to at most two more $5$-faces. Therefore, the final charge is at least $9-6-2 \cdot 1 - \frac{1}{4} \cdot 2 = 0$. 

\textbf{Case 2:} $f$ is adjacent to exactly one $3$-face. Let $y$ be the number of $4$-faces that $f$ is adjacent to and $z$ be the number of $5$-faces that $f$ is adjacent to. By Lemmas~\ref{no3|4},~\ref{no3|5},~\ref{no3-4},~\ref{no4|4},~\ref{no4|5} and~\ref{no555-I}, we know $0 \le y \le 2$, $0 \le z \le 4$, and $2y + \frac{3}{2} z \le 7$. Therefore, the final charge is at least $$9-6-1- y\cdot \frac{1}{2} - z \cdot \frac{1}{4}  = 2 - \frac{1}{4}(2y+z) \ge 2 - \frac{7}{4} = \frac{1}{4} > 0,$$
and $f$ is happy.

\textbf{Case 3:} $f$ is not adjacent to any $3$-face. Then $f$ can be adjacent to at most four $4$-faces by Lemma~\ref{no4|4}. Let $y$ be the number of $4$-faces that $f$ is adjacent to and $z$ be the number of $5$-faces that $f$ is adjacent to. By Lemma~\ref{no4|4},
we have $0 \le y \le 4$. By Lemma~\ref{no555-I}, we have $0 \le z \le 6$. By Lemma~\ref{no4|5}, we have $2y + \frac{3}{2} z \le 9$. Thus, the final charge is at least $$9-6-y \cdot \frac{1}{2} - z \cdot \frac{1}{4} = 3 - \frac{1}{4}(2y+z) \ge 3 - \frac{9}{4} = \frac{3}{4} > 0,$$
and $f$ is happy.
\end{proof}

\begin{claim}\label{10+face}
All the faces of length at least $10$ are happy.    
\end{claim}

\begin{proof}
For a face of length $\ell \ge 10$. Let $x,y,z$ be the number of $3,4,5$-faces that $f$ is adjacent to respectively. By Lemmas~\ref{no3|3},~\ref{no3-3},~\ref{no3--3},~\ref{no3|4},~\ref{no3|5},~\ref{no3-4},~\ref{no4|4},~\ref{no4|5} and~\ref{no555-I}, we have $0 \le 4x \le \ell$, $0 \le 2y \le \ell$, $0 \le \frac{3}{2}z \le \ell$, $3x+2y \le \ell$, $2x + \frac{3}{2}z \le \ell$, and $2y+ \frac{3}{2}z \le \ell$. Thus, the final charge is at least $$\ell - 6 - x - y \cdot \frac{1}{2} - z \cdot \frac{1}{4} = \ell - 6 - \frac{7}{30} (3x+2y) - \frac{3}{20} (2x+\frac{3}{2}z) - \frac{1}{60} (2y+\frac{3}{2}z)$$
$$\ge \ell - 6 - (\frac{7}{30}+\frac{3}{20}+\frac{1}{60}) \cdot \ell = \frac{3}{5} \ell - 6 \ge 0,$$
since $\ell \ge 10$. Therefore, all faces of length at least $10$ are happy.
\end{proof}

By Claims~\ref{6-face},~\ref{7face},~\ref{8face},~\ref{9face} and~\ref{10+face}, it is a contradiction to~\eqref{facecharging} since all the vertices and faces are happy. 
This completes the proof of Theorem \ref{planar}.
\hfill \qed


\section{Concluding Remarks}

We believe the answer to Question 2 is affirmative and therefore posed the following conjecture.

\begin{conj}
Every subcubic graph except the Petersen graph is packing $(1,2^5)$-colorable.
\end{conj}

\section*{Acknowledgement}

The authors would like to thank Bernard Lidick\'{y} for the helpful discussion. The C++ program in the Appendix is adapted from an early version of him.

\section{Appendix}


\subsection{Class One configurations}

\textbf{Proof of Lemma~\ref{no374} for configuration ``$3|7|4$'':}
Let $u_1, u_2, u_3$ be a triangle and $u_2, u_3, u_4, u_5, u_6, u_7, u_8$ be a seven cycle such that $u_2u_3$ is shared by the triangle and the seven cycle. Let $N(u_1) = \{u_2, u_3, u_9\}$, $N(u_4) = \{u_3, u_5, u_{10}\}$, $N(u_5) = \{u_4, u_6, u_{11}\}$, $N(u_6) = \{u_5, u_7, u_{12}\}$, $N(u_7) = \{u_6, u_8, u_{13}\}$, and $N(u_8) = \{u_4, u_7, u_{14}\}$. By Lemma~\ref{no3-4}, we may assume that $u_{11}u_{12} \in E(G)$ and let $N(u_{11}) = \{u_5, u_{12}, u_{15}\}$. We delete $u_1, u_2, u_3$ and add a vertex $u$ with the edges $uu_9, uu_8, uu_4$ to obtain a subcubic planar graph $G'$. By the minimality of $G$, $G'$ has a good coloring $f$. We extend $f$ to $G$.

By Lemma~\ref{no3-4}, we only need to consider the case when $1 \notin \{f(u_4), f(u_8), f(u_{9})\} $. Say $f(u_4) = A$, $f(u_9) = B$, and $f(u_8) = C$. We may assume $\{1, D\} = \{f(u_5), f(u_{10})\}$, $\{1, D\} = \{f(u_7), f(u_{14})\}$, $1, D \in f(N(u_9))$, $\{B, C\} \subseteq f(N^2(u_4))$. Moreover, if $B \notin f(N^2(u_5))$, then we can recolor $u_5$ with $B$ and $u_4$ with $1$, which contradicts out assumption that $f(u_4) \neq 1$. Thus, we may assume that $B \in f(N^2(u_5))$. Similarly, we assume $C, E \in f(N^2(u_5))$. 


 \textbf{Case 1:} $f(u_5) = 1$ and $f(u_{10}) = D$. If $A \notin f(N^2(u_5))$, then we can switch the colors of $u_4$ and $u_5$, which contradicts our assumption that $f(u_4) \neq 1$. Thus, we may assume that $A \in \{f(u_{12}), f(u_{15})\}$. Note that $1 \in f(N(u_{10}))$, since otherwise we recolor $u_{10}$ with $1$ and it contradicts our assumption that $\{f(u_5), f(u_{10})\} = \{1, D\}$. 

 \textbf{Case 1.1:} $f(u_6) \in \{B, D\}$. We know $f(u_{11}) \neq 1$.  Since $A, C, E \in f(N^2(u_5))$ and $f(u_7) \in \{1, D\}$, we switch the colors of $u_5$ and $u_{11}$, recolor $u_4$ with $1$ to obtain a contradiction.

 \textbf{Case 1.2:} $f(u_6) = E$. We obtain a contradiction by switching $B$ with $E$ in the argument of Case 1.1.

%

 \textbf{Case 2:} $f(u_5) = D$ and $f(u_{10}) = 1$. Then $f(u_7) = 1$ and $f(u_6) \neq 1$. We have $f(u_{11}) = 1$ since otherwise we recolor $u_5$ with $1$ and it contradicts our case. We know $f(u_6) \in \{B, E\}$. 

If $f(u_6) = E$ then since $\{B, C\} \subseteq f(N^2(u_5))$ we have $\{f(u_{12}), f(u_{15})\} = \{C, E\}$. Since $\{B, C\} \subseteq f(N^2(u_4))$, we know $f(N(u_{10})) = \{A, B, C\}$ and thus can switch the colors of $u_4$ and $u_5$, color $u_3$ with $E$, $u_2$ with $A$, and $u_1$ with $1$. Thus, we have $f(u_6) = B$. Let $N(u_{13}) = \{u_7, u_{16}, u_{17}\}$. Then we know $B \in \{f(u_{16}), f(u_{17})\}$ since other wise we switch the colors of $u_6$ and $u_7$, which contradicts our assumption that $f(u_7) = 1$. Similarly, $C \in \{f(u_{16}), f(u_{17})\}$. Moreover, $f(u_{13}) = A$ since otherwise we recolor $u_7$ with $A$ and it contradicts our assumption that $f(u_7) = 1$. Then we switch the colors of $u_7$ and $u_{13}$, a contradiction again. \hfill \qed


\subsection{Class Two configurations}

We use $A(u)$ to denote the number of available colors that $u$ has from $A,B,C,D,E$.

\noindent\textbf{Proof of Lemma~\ref{no555-I} for configuration ``$5|5|5$-I'':}
Suppose not, i.e., there is such a structure in our graph $G$. Let $u_1u_2u_3u_4u_5$ be the five cycle $C$ in the middle and $v_i$ be the neighbours of $u_i$ that do not belong to the cycle $C$, where $i \in [5]$. Let $N(v_i) = \{u_i, v_i', v_i''\}$, where $i \in [5]$, $v_3'' = v_2''$, and $v_4'' = v_5''$. Let $N(v_5') = \{v_5, w_5', w_5''\}$, $N(v_4'') = \{v_4, v_5, w_4\}$, $N(w_4) = \{w_4', w_4'', v_4''\}$, and $N(v_4') = \{v_4, z_4', z_4''\}$. We delete $u_1, u_2, u_3, u_4, u_5$ from $G$ and add the edges $v_2v_5, v_3v_4, v_2''v_4''$ to form $G'$. By the minimality of $G$, $G'$ has a good coloring $f$. We extend $f$ to $G$.

 \textbf{Case 1:} $v_1 = 1$.  We discuss in cases on the number of vertices from $v_2, v_3, v_4, v_5$ that is colored by $1$.

 \textbf{Case 1.1:} Two of $v_2, v_3, v_4, v_5$ are colored by $1$. By symmetry, we have the following cases.

 \textbf{Case 1.1.1:} $f(v_2) = f(v_3) = 1$. Say $f(v_4) = A$ and $f(v_5) = B$. Then $|A(u_1)| \ge 2, |A(u_2)| \ge 3, |A(u_3)| \ge 2, |A(u_4)| \ge 1, |A(u_5)| \ge 1$. We discuss cases that $u_1, u_2, u_3, u_4, u_5$ cannot be colored.

 \textbf{Case 1.1.1.1:} $|A(u_4) \cup A(u_5)| = 1$. Say  $A(u_4) = A(u_5) = \{E\}$. Then either $|A(u_1) \cup A(u_3) \cup A(u_4) \cup A(u_5)| = 2$ or $|A(u_1) \cup A(u_3) \cup A(u_4) \cup A(u_5)| = |A(u_1) \cup A(u_2) \cup A(u_3) \cup A(u_4) \cup A(u_5)| = 3$, since otherwise we are done by SDR (and recolor one of $u_4, u_5$ by $1$). In the former case, $A(u_1) = A(u_3) = \{C,E\} \text{ or  } \{D,E\}$. Since $v_2v_5$ is an edge in $G'$, we know $B \in A(v_2)$. We recolor $v_4$ with $1$ and color $u_1, u_2, u_3, u_4, u_5$ with $C (D),B,E,A,1$, which is a contradiction. In the latter case,  we again know $B \in A(u_2)$ and thus $A(u_1) \cup A(u_3) \cup A(u_4) \cup A(u_5) = \{B,E,A(C,D)\}$. Since $B \notin A(u_1)$, $A(u_1) = \{E,A(C,D)\}$. By $|A(u_1) \cup A(u_3) \cup A(u_4) \cup A(u_5)| = 3$, $A(u_3)=\{B,A(C,D)\}$ or $\{B,E\}$. 

If $A(u_1) = \{E,C(D)\}$, then $A(u_3) = \{B,C(D)\} \text{ or } \{B,E\}$ and we can recolor $v_4$ with $1$, color $u_4, u_5$ with $A, 1$, and color $u_1, u_2, u_3$ with SDR. If $A(u_1) = \{E,A\}$, then $A(u_3) = \{B,E\}$ and $A(u_2) = \{A,B,E\}$. Since $v_2''v_4''$ is an edge in $G'$, we know $f(v_2'') = D$ and $f(v_2') = f(v_3') = C$. We can recolor $v_2''$ to one of $\{A,B,E\}$, and color $u_1,u_2,u_3,u_4,u_5$ with $A,D,B,E,1$, which is again a contradiction.

 \textbf{Case 1.1.1.2:} $|A(u_4) \cup A(u_5)| = 2$. Say $A(u_4) \cup A(u_5) = \{D,E\}$.

If $A(u_4) = \{E\}$ and $A(u_5) = \{D\}$, then we know $f(v_4'') = C, f(v_4') = D, f(v_5') = E$. We first recolor $v_4$ and $v_5$ to $1$, then $A(u_4) = \{A,B,E\}$, $A(u_5) = \{A,B,D\}$, and each of $|A(u_i)| \ge 3$, where $i \in [3]$. Therefore, the only cases we are not done by SDR is when the union of available colors of four vertices has size three or $|A(u_1) \cup A(u_2) \cup A(u_3) \cup A(u_4) \cup A(u_5)| = 4$. In the former case, the four vertices must be $\{u_1, u_2, u_3, u_4\}$ or $\{u_1, u_2, u_3, u_5\}$; if the four vertices are $\{u_1, u_2, u_3, u_4\}$, then since $A \in A(u_3)$, $A(u_1) = A(u_2) = A(u_3) = \{A,B,D\}$ we can recolor $v_5$ with $B$ and color $u_1, u_2, u_3, u_4, u_5$ with $A,B,D,E,1$, which is a contradiction; if the four vertices are $\{u_1, u_2, u_3, u_5\}$, then since $B \in A(u_2)$, $A(u_1) = A(u_2) = A(u_3) = \{A,B,E\}$ we can recolor $v_4$ with $A$ and color $u_1, u_2, u_3, u_4, u_5$ with $A,B,E,1,D$, which is again a contradiction. In the latter case, $A(u_1) \cup A(u_2) \cup A(u_3) \cup A(u_4) \cup A(u_5) = \{A,B,D,E\}$ and we may assume $A(u_1) \cup A(u_2) \cup A(u_3) \cup A(u_4) = \{A,B,D,E\}$ as well. We can recolor $v_5$ to $B$ and color $u_2, u_5$ with $B,1$. Then $|A(u_1)| \ge 2, |A(u_3)| \ge 2, |A(u_4)| \ge 2$ and $|A(u_1) \cup A(u_3) \cup A(u_4)| \ge 3$ thus we are done by SDR.

If $A(u_4) = E$ and $A(u_5) = \{D,E\}$, then we know $f(v_5) = 1$ and $ f(v_4'') = C$ and $f(v_4') = D$.  We first recolor $v_4$ with $1$. Then $A(u_5) = \{A,D,E\}$ and $A(u_4) = \{A,E\}$. Since $B \in A(u_2)$, $|A(u_1)| \ge 2$, $|A(u_2)| \ge 3$, $|A(u_3)| \ge 3$, the only cases we are not done by SDR is when the union of available colors of four vertices has size three or $|A(u_1) \cup A(u_2) \cup A(u_3) \cup A(u_4) \cup A(u_5)| = 4$. In the former case, the four vertices must be $\{u_1, u_2, u_3, u_4\}$ or $\{u_1, u_3, u_4, u_5\}$; if the four vertices are $\{u_1, u_2, u_3, u_4\}$, then $A(u_1) = A(u_4) = \{A,E\}$, $A(u_2) = A(u_3) = \{A,B,E\}$, and we can recolor $v_4$ to $A$, color $u_1, u_2, u_3, u_4, u_5$ with $A,B,E,1,D$, which is a contradiction; if the four vertices are $\{u_1, u_3, u_4, u_5\}$, then $A(u_4) = \{A,E\}$, $A(u_3) = A(u_5) = \{A,D,E\}$, and we can color $u_2, u_5$ with $B,1$, and color $u_1, u_3, u_4$ by SDR.  In the latter case, $A(u_1) \cup A(u_2) \cup A(u_3) \cup A(u_4) \cup A(u_5) = \{A,B,D,E\}$ and we may assume $A(u_1) \cup A(u_2) \cup A(u_3) \cup A(u_4) = \{A,B,D,E\}$ as well. We can color $u_5$ with $1$ and color $u_1, u_2, u_3, u_4$ by SDR.

.
If $A(u_5) = D$ and $A(u_4) = \{D,E\}$, then we know $f(v_4') = 1$, $ f(v_4'') = C$ and $f(v_5') = E$. We first recolor $v_5$ with $1$. Then $A(u_5) = \{B,D\}$ and $A(u_4) = \{B,D,E\}$. By $v_2''v_4'' \in G'$, $f(v_2'') \in \{D,E\}$ and thus $\{D,E\} \nsubseteq A(u_2)$. Since $B \in A(u_2)$, $|A(u_1)| \ge 3$, $|A(u_2)| \ge 3$, $|A(u_3)| \ge 2$, the only cases we are not done by SDR is when the union of available colors of four vertices has size three or $|A(u_1) \cup A(u_2) \cup A(u_3) \cup A(u_4) \cup A(u_5)| = 4$. In the former case, the four vertices must be $\{u_1, u_2, u_3, u_5\}$ or $\{u_1, u_3, u_4, u_5\}$; if the four vertices are $\{u_1, u_2, u_3, u_5\}$, then $A(u_5) = \{B,D\}$, $A(u_1) = A(u_2) = \{B,D,A(C)\}$, and we can recolor $v_5$ to $B$, color $u_2, u_4, u_5$ with $B,E,1$ and $u_1,u_3$ by SDR; if the four vertices are $\{u_1, u_3, u_4, u_5\}$, then $A(u_4) = \{A,E\}$, $A(u_1) = A(u_4) = \{B,D,E\}$, and we can color $u_2, u_4$ with $C \text{ or }A,1$, and color $u_1, u_3, u_5$ by SDR.  In the latter case, $A(u_1) \cup A(u_2) \cup A(u_3) \cup A(u_4) \cup A(u_5) = \{A (C),B,D,E\}$ and we may assume $A(u_1) \cup A(u_2) \cup A(u_3) \cup A(u_5) = \{A (C),B,D,E\}$ as well. We can color $u_4$ with $1$ and $u_1, u_2, u_3, u_5$ by SDR.

 \textbf{Case 1.1.1.3:} $|A(u_4)| \ge 2$ and $|A(u_5)| \ge 2$. Then the cases we are not done by SDR are $|A(u_1) \cup A(u_3) \cup A(u_4) \cup A(u_5)| = 2$ or $|A(u_1) \cup A(u_3) \cup A(u_4) \cup A(u_5)| = |A(u_1) \cup A(u_2) \cup A(u_3) \cup A(u_4) \cup A(u_5)| = 3$. 

$\bullet$ $|A(u_1) \cup A(u_3) \cup A(u_4) \cup A(u_5)| = 2$. Say $|A(u_1) \cup A(u_3) \cup A(u_4) \cup A(u_5)| = \{D, E\}$. It implies $A(u_1) = A(u_3) = A(u_4) = A(u_5) = \{D,E\}$. Since $v_2v_5 \in G'$, $B \in A(u_2)$ and $f(v_2'') \neq B$. Thus, $f(v_3') = B$ and $f(v_2'') = C$. We can also assume $\{f(v_1'), f(v_1'')\} = \{A,C\}$. We have either $f(v_4'') = 1$ or $f(v_4') = f(v_5') = 1$. 

In the former case, $f(v_4') = f(v_5') = C$, $1 \in \{f(z_4'), f(z_4'')\}$, and $1 \in \{f(w_5'), f(w_5'')\}$. If we can recolor $v_4$ to $D$ or $E$, then $A \in A(u_4)$ and it implies $|A(u_1) \cup A(u_3) \cup A(u_4)| \ge 3$. Since $B \in A(u_2)$ and $B \notin A(u_4)$, it implies $|A(u_1) \cup A(u_2) \cup A(u_3) \cup A(u_4)| \ge 4$. Thus, we can color $u_5$ with $1$ and finish the coloring on $u_1, u_2, u_3, u_4$ by SDR. Therefore, we can assume $f(w_4) = E$ and $\{f(z_4'), f(z_4'')\} = \{1,D\}$. Similarly, we can assume $\{f(w_5'), f(w_5'')\} = \{1,D\}$. Then we can recolor $v_5$ with $A$, $v_4$ with $B$, and color $u_1, u_2, u_3, u_4, u_5$ with $D,B,A,1,E$, which is a contradiction. 

In the latter case, $f(v_4'') = C$. We may assume $f(v_4') = 1$ since otherwise we can recolor $v_4$ to $1$  then $A \in A(u_4)$ and we are done by the same reason which was just covered in the previous paragraph (line 2-4). Similarly, $f(v_5') = 1$. We may assume $f(w_4) = 1$ since otherwise we can recolor $v_4''$ with $1$, which contradicts our assumption. If we can recolor $v_4$ to $D$ or $E$, then $A \in A(u_4)$ and we are done by the same reason which was just covered in the previous paragraph (lines 2-4). Then we have $\{f(z_4'), f(z_4'')\} = \{f(w_5), f(w_5'')\} = \{D,E\}$, and we are done by recoloring $v_5$ with $A$, $v_4$ with $B$, and color $u_1, u_2, u_3, u_4, u_5$ with $D,B,A,1,E$.

$\bullet$ $|A(u_1) \cup A(u_3) \cup A(u_4) \cup A(u_5)| = |A(u_1) \cup A(u_2) \cup A(u_3) \cup A(u_4) \cup A(u_5)| = 3$. Let $U = A(u_1) \cup A(u_2) \cup A(u_3) \cup A(u_4) \cup A(u_5)$. Since $v_2v_5 \in G'$, $B \in A(u_2)$ and thus $B \in U$. 

If $A \in U$, say $U = \{A,B,C\}$, then we may assume $f(v_4'') = f(v_3') = f(v_2') = f(v_1') = D$ and $f(v_4') = f(v_5') = f(v_1'') = f(v_2'') = E$. Then $|A(u_4)| = |A(u_5)| = 1$, which was solved in Case 1.1.1.1. Thus, we know $A \notin U$ and we may assume by symmetry $U = \{B,D,E\}$. Since $v_3v_4, v_2v_5 \in G'$,  we must have $f(v_2') = A$ and $f(v_2'') = C$. It implies $v_3' = B$ and thus $B \notin A(u_1) \cup A(u_3) \cup A(u_4) \cup A(u_5)$ and thus $|A(u_1) \cup A(u_3) \cup A(u_4) \cup A(u_5)| = 2$. By the previous case of Case 1.1.1.3, we are done.

 \textbf{Case 1.1.2:} $f(v_3) = f(v_5) = 1$. By SDR, the only case we are not done is $|A(u_1) \cup A(u_3) \cup A(u_5)| = 2$. We have either $f(v_2) = f(v_4)$ or $f(v_2) \neq f(v_4)$.

 \textbf{Case 1.1.2.1:} $f(v_2) = f(v_4)$. Say $f(v_2) = f(v_4) = A$ and assume $A(u_1) \cup A(u_3) \cup A(u_5) = \{D,E\}$. Then we can assume $f(v_4'') = B$, $f(v_5') = C$, $\{f(v_1'), f(v_1'')\} = \{B,C\}$, $f(v_2'') = C$, and $f(v_3') = B$. We may also assume $f(w_4) = 1$ since otherwise we can recolor $f(v_4'')$ with $1$ and $|A(u_1) \cup A(u_3) \cup A(u_5)| = 3$ (we are done by SDR). If we can recolor $v_4$ to $C,D$ or $E$, then we are again done by SDR since then $|A(u_1) \cup A(u_3) \cup A(u_5)| = 3$. Therefore, $\{f(v_4'), f(z_4'), f(z_4'')\} = \{C,D,E\}$. However, we can recolor $v_4'$ to $1$ and reach a contradiction. 


 \textbf{Case 1.1.2.2:} $f(v_2) \neq f(v_4)$. Say $f(v_2) = B$ and $f(v_4) = A$. Since $v_2v_5 \in G'$, $B \in A(u_5)$. Since $|A(u_1)| \ge 2$ and $B \notin A(u_1)$, we know $|A(u_1) \cup A(u_3) \cup A(u_5)| \ge 3$. Therefore, by $|A(u_1)| \ge 2$, $|A(u_5)| \ge 2$, and $|A(u_3)| \ge 1$, we can color $u_1, u_3, u_5$ by SDR and color $u_2$ and $u_4$ with $1$.

 \textbf{Case 1.2:} One of $v_2, v_3, v_4, v_5$ is colored by $1$. By symmetry, we have two cases.

 \textbf{Case 1.2.1:} $f(v_2) = 1$. Since $v_3v_4 \in G'$ and $v_4,v_5$ are at distance two, we know $f(v_4) \neq f(v_3)$ and $f(v_4) \neq f(v_5)$. We thus have the following two subcases.

 \textbf{Case 1.2.1.1:} $f(v_4) = B$ and $f(v_3) = f(v_5) = A$. In this case, $A$ is not available for each of $u_1, u_2, u_3, u_4, u_5$. Since $|A(u_4)| \ge 1$, $|A(u_2)| \ge 2$, and $|A(u_1)| \ge 2$, the only case  we are not done by SDR is when $|A(u_1) \cup A(u_2) \cup A(u_4)| = 2$. Let $U = A(u_1) \cup A(u_2) \cup A(u_4)$. If $B \in U$, say $U = \{B,C\}$, then we must have $\{f(v_4'), f(v_4'')\} =  \{f(v_2'), f(v_2'')\} = \{f(v_1'), f(v_1'')\} = \{D,E\}$ and $f(w_4) = 1$. We must also have $f(v_5') = 1$ since otherwise we can recolor $v_5$ with $1$ and then we are done by SDR ($A$ becomes available in $U$). Then we can recolor $v_5$ to one of $C,D,E$, then $A \in A(u_5)$ and $A(u_2) = \{B,C\}$ imply that we are done by SDR. Therefore, we must have $B \notin U$, say $U = \{D,E\}$. We must have $\{f(v_1'), f(v_1'')\} = \{f(v_2'), f(v_2'')\} = \{B,C\}$. If we can recolor $v_4$ with $1$, then $B$ is available in $U$ and we are done by SDR. Thus, at least one of $v_4'$ and $v_4''$ must be colored with $1$. For the same reason, we cannot recolor $v_4$ with $D,E$, and thus $D,E$ are with distance two from $v_4$. If $f(v_4'') = 1$, then $f(v_4'') = C$ and $1 \in \{f(z_4'), f(z_4'')\}$. We may assume $f(w_4) = E$ and $\{f(z_4'), f(z_4'')\} = \{1,D\}$. If we can recolor $v_5$ with one of $B,C,D$ (for color $B$ we switch the colors of $v_4$ and $v_5$), then $A$ becomes available in $U$ and we are done by SDR. Thus, $\{f(w_5'), f(w_5''), f(v_5')\} = \{B,C,D\}$. However, we can recolor $v_5'$ with $1$, which is a contradiction. Therefore, we must have $f(v_4') = 1$ and $f(v_4'') = C$. Similarly, we know $f(w_4) = 1$ and $\{f(z_4'), f(z_4'')\} = \{D,E\}$. By a similar argument, we reach a contradiction.

 \textbf{Case 1.2.1.2:} $f(v_4) = B$, $f(v_5) = A$, and $f(v_3) = C$.  Similarly to Case 1.2.1.1, we know $|A(u_1) \cup A(u_2) \cup A(u_4)| = 2$. Since $v_2v_5 \in G'$, $A \in A(u_2)$. However, $|A(u_1)| \ge 2$ and $A \notin A(u_1)$, which is a contradiction.

 \textbf{Case 1.2.2:} $f(v_3) = 1$. Since $v_2v_5 \in G'$, $f(v_2) \neq f(v_5)$ and $f(v_4) \neq f(v_5)$. We have the following cases.

 \textbf{Case 1.2.2.1:} $f(v_5) = B$ and $f(v_2) = f(v_4)$, say $f(v_2) = f(v_4) = A$. Then none of $u_1, u_2, u_3, u_4, u_5$ has color $A$ available. We know one of $v_4''$ and $v_5'$ must be $1$, since otherwise we can recolor $v_5$ with $1$ and it was already done in Case 1.1.2. Let $U = A(u_1) \cup A(u_3) \cup A(u_4)$. In case $|A(u_4)| = 1$, $f(v_4')$ and $f(v_4'')$ are both $2$-colors, then we can recolor $v_4$ with $1$, color $u_4$ with $A$, $u_2, u_5$ with $1$, and color $u_1, u_3$ by SDR. Thus, we may assume $|A(u_4)| \ge 2$. Since $|A(u_1)| \ge 1$, $|A(u_3)| \ge 2$, and $|A(u_4)| \ge 2$, the only case we are not done by SDR is when $|U| = 2$. Since $|A(u_4)| \ge 2$ and $B \notin A(u_4)$, $B \notin U$. Thus, we may assume $U = \{D,E\}$. 

Recall that one of $v_4''$ and $v_5'$ must be colored with $1$. If $f(v_4'') = 1$, then $f(v_4') = C$ and say $\{f(v_1'), f(v_1'')\} = \{C,D\}$. We also know $f(v_5') = C$, since otherwise we can color $u_1, u_2, u_3, u_4, u_5$ with $E, 1, D, 1, C$, which is a contradiction. We claim $\{f(w_4), f(w_4'), f(w_4'')\}$ must be $\{B,D,E\}$ since otherwise we can recolor $v_4''$ with $B$ ($D$ or $E$) and recolor $v_5$ with $1$, which was already covered in Case 1.1.2. However, we can switch the colors of $v_4''$ and $w_4$, recolor $v_4$ with $1$, and color $u_1, u_2, u_3, u_4, u_5$ with $E, 1, D, A, 1$. If $f(v_4'') \neq 1$, then $f(v_5') = f(w_4) = 1$ and $f(v_4') = C$. Since $|A(u_4)| \ge 2$, $f(v_4') = 1$. We claim $\{f(w_5'), f(w_5'')\} = \{D,E\}$ since otherwise we recolor $v_5$ with $D$ or $E$, and color $u_1, u_2, u_3, u_4, u_5$ with $B,1,E,D,1$, which is a contradiction. We also know $\{f(w_4'), f(w_4'')\} = \{D,E\}$ since otherwise we can recolor $v_4''$ with $D$ or $E$, and color $u_1, u_2, u_3, u_4, u_5$ with $E(D),1,D(E),C,1$, which is a contradiction. However, we can switch the colors of $v_4''$ and $v_5$, and color $u_1, u_2, u_3, u_4, u_5$ with $B,1,E,D,1$, which is again a contradiction.


 \textbf{Case 1.2.2.2:} $f(v_5) = B$ and $f(v_2) \neq f(v_4)$, say $f(v_2) = A$ and $f(v_4) = C$. 

The first case that we are not done is when $u_1$ and $u_3$ both have only one color available and the available colors are the same, say $D$. Then we may assume that $\{f(v_1'), f(v_1'')\} = \{C,E\}$, $f(v_2') = 1$, $f(v_2'') = E$, and $f(v_3') = B$. Since $v_2''v_4'' \in G'$ and $v_2v_5 \in G'$, $f(v_4'')$ can be $1$ or $D$. In case $f(v_4'') = D$. We may assume $f(w_4) = 1$. We know $f(v_5') = 1$ since otherwise we can recolor $v_5$ with $1$ and it was solved in Case 1.1.2. We also claim $f(v_4') = 1$ since otherwise we can recolor $v_4$ with $1$ and color $u_1, u_2, u_3, u_4, u_5$ with $D,1,C,E(A),1$, which is a contradiction. If we can recolor $v_4$ to $A$ or $E$, then we can color $u_1, u_2, u_3, u_4, u_5$ with $D,1,C,1,E(A)$, which is a contradiction. Thus, we know $\{f(z_4'), f(z_4'')\} = \{A,E\}$. Similarly, we know $\{f(w_5'), f(w_5'')\} = \{A,E\}$. Then we switch the colors of $v_5$ and $v_4$, and color $u_1, u_2, u_3, u_4, u_5$ with $D,1,C,1,A$, which is a contradiction. In case $f(v_4'') = 1$. Let $N(v_2'') = \{v_2, v_3, y_2\}$ and $N(y_2) = \{v_2'', y_2', y_2''\}$. We know $f(y_2)$ can be $1,B,C,D$. If $f(y_2) = 1$, then we know $f(N(v_3') - v_3) = \{1,C\}$ since otherwise we can recolor $v_3$ with $B$ and $v_3'$ with $1$ (recolor $v_3$ with $C$), and color $u_1, u_2, u_3, u_4, u_5$ with $D,C(B),1,E \text{ or }A,1$, which is a contradiction; we recolor $v_3$ with $D$, and color $u_1, u_2, u_3, u_4, u_5$ with $D,B,1,E \text{ or }A, 1$, which is again a contradiction. If $f(y_2) = C$, then we know $1 \in \{f(y_2'), f(y_2'')\}$ since otherwise we can recolor $y_2$ with $1$, which is already solved. We also know $D \in \{f(y_2'), f(y_2'')\}$ since otherwise we can recolor $v_2''$ with $D$, and color $u_1, u_2, u_3, u_4, u_5$ with $D,1,E,1,A$, which is a contradiction. We have $f(N(v_3') - v_3) = \{1,D\}$ since otherwise we can recolor $v_3$ with $B$ or $D$, and color $u_1, u_2, u_3, u_4, u_5$ with $D,1,E,1,A$, which is a contradiction. We claim $f(v_4')$ must be $A$ since otherwise we can switch the colors of $v_2''$ and $v_3$, and color $u_1, u_2, u_3, u_4, u_5$ with $D,B,1,A,1$, which is a contradiction. We claim $f(N(v_2')-v_2) = \{B,D\}$ since otherwise we can recolor $v_2$ with $B (D)$, and color $u_1, u_2, u_3, u_4, u_5$ with $D (A),1 (B),A (1),E,1$, which is a contradiction. Thus, we can recolor $v_2, v_2'', v_3$ with $E,1,A$, and color $u_1, u_2, u_3, u_4, u_5$ with $D,B,1,E,1$, which is again a contradiction. If $f(y_2) = D$, then similarly to previous cases we can assume that $\{f(y_2'), f(y_2'')\} = f(N(v_3') - v_3) = \{1,C\}$, $f(v_4') = A$, and $f(N(v_2') - v_2) = \{B,C\}$; thus we can recolor $v_2, v_2'', v_3$ with $E,1,A$, and color $u_1, u_2, u_3, u_4, u_5$ with $D,B,1,E,1$, which is a contradiction. If $f(y_2) = B$, then similarly to previous cases we can assume that $\{f(y_2'), f(y_2'')\}  = \{1,C\}$, $f(N(v_3') - v_3) = \{C,D\}$, $f(v_4') = A$, and $f(N(v_2') - v_2) = \{C,E\}$; thus we can recolor $v_2, v_2'', v_3$ with $D,1,A$, and color $u_1, u_2, u_3, u_4, u_5$ with $A,B,1,E,1$, which is a contradiction. 

Therefore, we know $|A(u_1) \cup A(u_3)| \ge 2$. We know $1 \in \{f(v_4''), f(v_5')\}$ since otherwise we can recolor $v_5$ with $1$ and it was already solved in Case 1.1.2. We know $|A(u_5)| \ge 2$ and $|A(u_1) \cup A(u_3)| \ge 2$ thus the only case we are not done by SDR ($u_2$ and $u_4$ can be colored with $1$) is when $|A(u_1) \cup A(u_3) \cup A(u_5)| = 2$. Since $A \in A(u_5)$ and $B,C \notin A(u_5)$, we may assume by symmetry that $A(u_5) = A(u_1) \cup A(u_3) = A(u_1) \cup A(u_3) \cup A(u_5) = \{A,D\}$. However, $A \notin A(u_1)$ and $A \notin A(u_3)$ imply that $A(u_1) \cup A(u_3) = \{D\}$, which is a contradiction. 

 \textbf{Case 1.3:} None of $v_2, v_3, v_4, v_5$ is colored by $1$. We know at least one of $v_4''$ and $v_5'$ are colored with $1$, since otherwise we can recolor $v_5$ with $1$ and it was already solved in Case 1.2. Similarly, we know at least one of $v_4'$ and $v_4''$, $v_3'$ and $v_2''$, $v_2'$ and $v_2''$ are colored with $1$. By symmetry, we have three cases for the colors on $v_2, v_3, v_4, v_5$.

 \textbf{Case 1.3.1:} $v_2, v_3, v_4, v_5$ has colors $B,A,B,A$. We know either $f(v_4'') = 1$ and $f(v_2'') \neq 1$, $f(v_2'') = 1$ and $f(v_4'') \neq 1$, and both $f(v_2'')$ and $f(v_4'')$ are not colored with $1$. By symmetry, we have two cases.

 \textbf{Case 1.3.1.1:} $f(v_2'') = 1$ and $f(v_4'') = C$. Then $f(v_4') = f(v_5') = 1$. Since $|A(u_1)| \ge 1$, $|A(u_3)| \ge 2$, $|A(u_4)| \ge 2$, we know $|A(u_1) \cup A(u_3) \cup A(u_4)| = 2$ since otherwise we are done by SDR (color $u_2$ and $u_5$ with $1$). Since $A(u_4) = \{D,E\}$, we know $A(u_3) = \{D,E\}$ and $A(u_1) = \{E\}$ (by symmetry). Therefore, $f(v_3') = C$. Similarly, we know $f(v_2') = C$. We can assume $f(w_4) = 1$ since otherwise we recolor $v_4''$ with $1$ and we are done by SDR. If we can recolor $v_4$ with $D$ or $E$, then we can color $u_1, u_2, u_3, u_4, u_5$ with $E,D,1,B,1$, which is a contradiction. Therefore, $f(N(v_4')-v_4) = \{D,E\}$. Similarly, $f(N(v_5')-v_5) = \{D,E\}$ since otherwise we can recolor $v_5$ with $D (E)$, and color $u_1, u_2, u_3, u_4, u_5$ with $A,D (E),1,E (D),1$, which is a contradiction. We switch the colors of $v_4$ and $v_5$, and color $u_1, u_2, u_3, u_4, u_5$ with $A,1,E,D,1$, which is a contradiction.

 \textbf{Case 1.3.1.2:} $f(v_2'') = D$ and $f(v_4'') = C$. Then $f(v_2') = f(v_3') = f(v_4') = f(v_5') = 1$, $A(u_2) = A(u_3) = \{C,E\}$, $A(u_4) = A(u_5) = \{D,E\}$, and $|A(u_1)| \ge 1$. Thus, we can color $u_2$ and $u_5$ with $1$, and color $u_1, u_3, u_4$ with SDR.

 \textbf{Case 1.3.2:} $v_2, v_3, v_4, v_5$ has colors $B,C,B,A$. Since $v_2v_5, v_2''v_4'', v_3v_4 \in G'$, we have the following two cases.

 \textbf{Case 1.3.2.1:} $f(v_2'') = 1$ and $f(v_4'') = D$. We know $f(v_4') = f(v_5') = 1$. Since $A(u_5) = \{C,E\}$ and $C \notin A(u_3)$, we know $|A(u_1) \cup A(u_3) \cup A(u_5)| \ge 3$. By $|A(u_1)| \ge 1$, $|A(u_3)| \ge 2$, $|A(u_5)| \ge 2$, and $|A(u_1) \cup A(u_3) \cup A(u_5)| \ge 3$, we know $u_1,u_3, u_5$ can be colored with SDR. We finish by coloring $u_2$ and $u_4$ with $1$.

 \textbf{Case 1.3.2.2:} $f(v_2'') = E$ and $f(v_4'') = D$. We know $f(v_2') = f(v_3') = f(v_4') = f(v_5') = 1$. We color $u_2$ and $u_4$ with $1$. Since $|A(u_1)| \ge 1$, $A(u_5) = \{C,E\}$, and $A(u_3) = \{A,D\}$, we can color $u_1, u_3, u_5$ with SDR.

 \textbf{Case 1.3.3:} $v_2, v_3, v_4, v_5$ has colors $B,D,C,A$. Since $v_2v_5, v_2''v_4'', v_3v_4 \in G'$,  $f(v_2'') = 1$ and $f(v_4'') = E$ up to symmetry. Then $f(v_5') = f(v_4') = 1$. We know $A(u_4) = \{B\}$, $A \in A(u_2)$, $|A(u_2)| = 2$, $|A(u_1)| \ge 1$, and $A,B \notin A(u_1)$. Therefore, we can color $u_3$ and $u_5$ with $1$, and color $u_1,u_2,u_4$ with SDR.

 \textbf{Case 2:} $f(v_1) \neq 1$. Say $f(v_1) = A$ and $f(v_1') = 1$. We have three cases depending on the colors of $v_2,v_3,v_4,v_5$.

 \textbf{Case 2.1:} Two of $v_2,v_3,v_4,v_5$ are colored with $1$. Since $v_2v_5, v_3v_4 \in G'$, we have two subcases.

 \textbf{Case 2.1.1:} $f(v_2) = f(v_3) = 1$. By symmetry, we have three cases depending on the colors of $v_4$ and $v_5$.

 \textbf{Case 2.1.1.1:} $f(v_4) = B$ and $f(v_5) = A$. Let $U  = A(u_2) \cup A(u_3) \cup A(u_5)$ and we intend to color $u_1,u_4$ with $1$. We know $|A(u_2)| \ge 2$, $|A(u_3)| \ge 2$, $|A(u_5)| \ge 1$, $A \notin A(u_2)$, and $B \notin A(u_3)$. We must have $A,B \notin U$ since otherwise we are done by SDR. We also know $|U| = |A(u_2)| = |A(u_5)| = 2$ and therefore can assume $U = A(u_2) = A(u_5) = \{D,E\}$. By $v_2v_5, v_3v_4 \in G'$, $f(v_2') = B$, $f(v_2'') = f(v_5')= C$, $f(w_5') = 1$, and $f(v_3') = A$. By symmetry and $v_2''v_4'' \in G'$, we know $f(v_4'')$ is either $1$ or $D$. In the former case, we know $f(w_4) = D$ or $E$, and we assume $f(w_4) = D$. If we can recolor $v_4$ with $E$ or $C$, then we can color $u_1,u_2,u_3,u_4,u_5$ with $1,E,D,1,B$, which is a contradiction. Thus, $v_4', z_4', z_4''$ are colored with $E,C,1$ or $C,E,1$ or $1,E,C$. If we can switch the colors of $v_4$ and $v_5$, then we can color $u_1,u_2,u_3,u_4,u_5$ with $1,E,B,1,D$, which is a contradiction. Therefore, $f(w_5'') = B$ and we can recolor $v_5,v_4$ with $E,A$ and color $u_1,u_2,u_3,u_4,u_5$ with $1,E,D,1,B$, which is again a contradiction. In the latter case, we know $f(w_4) = 1$ and can repeat the argument in the former case.

 \textbf{Case 2.1.1.2:} $f(v_4) = A$ and $f(v_5) = B$. Similarly to Case 2.1.1.1, we know $|U| = 2$. By $v_2v_5 \in G'$, $B \in A(u_2)$ and $B \notin A(u_5)$. We also know $A \notin U$ and thus $U = \{B,C\}$ by symmetry. Then we may assume $f(v_2') = f(v_3') = D$ and $f(v_2'') = E$. Since $v_2''v_4'' \in G'$, we must have $f(v_4'') = D$ and $f(v_5') = E$. Similarly to Case 2.1.1.1, we know $v_4', z_4', z_4''$ are colored with $E,C,1$ or $C,E,1$ or $1,E,C$. Therefore, we recolor $v_5$ with $1$, $v_4$ with $B$, and color $u_1,u_2,u_3,u_4,u_5$ with $1,B,A,1,C$, which is a contradiction.

 \textbf{Case 2.1.1.3:} $f(v_4) = C$ and $f(v_5) = B$. Similarly to Case 2.1.1.1, we know $|U|=2$, $B \in U$, and $A,C \notin U$. Thus, we may assume $U = \{B,D\}$, $f(v_2') = C$, $f(v_2'') = E$, $f(v_3') = A$. By $v_2''v_4'' \in G'$, we can assume $f(v_5') = E$ and conclude $f(v_4'') = 1,A, \text{ or }D$. In case $f(v_4'') = 1$, then we know $f(w_4) = A \text{ or }D$. If $v_4$ can be recolored with $D(A)$ or $E$, then we can color $u_1,u_2,u_3,u_4,u_5$ with $1,D,B,1,C$, which is a contradiction. Thus, $v_4',z_4', z_4''$ are colored with $1,D(A),E$ or $D(A),1,E$ or $E,1,D(A)$. If we can switch the colors of $v_4$ and $v_5$, then we can color $u_1,u_2,u_3,u_4,u_5$ with $1,B,C,1,D$, which is a contradiction. Therefore, $f(w_5'') = C$ and we can recolor $v_5$ with $D$ then color $u_1,u_2,u_3,u_4,u_5$ with $1,B,D,1,C$, which is again a contradiction. In case $f(v_4'') = A(D)$, we know $f(w_4) = 1$. Similarly to former case, we know $v_4',z_4', z_4''$ are colored with $1,D(A),E$ or $D(A),1,E$ or $E,1,D(A)$. Therefore, we recolor $v_4, v_5$ with $B,1$ and color $u_1,u_2,u_3,u_4,u_5$ with $1,B,D,1,C$, which is a contradiction.

 \textbf{Case 2.1.2:} $f(v_2) = f(v_4) = 1$. By symmetry, we have three cases depending on the colors of $v_4$ and $v_5$. Let $U = A(u_1) \cup A(u_2) \cup A(u_4)$. We color $u_3,u_5$ with $1$.

 \textbf{Case 2.1.2.1:} $f(v_3) = A$ and $f(v_5) = B$. Since $|A(u_4)| \ge 1$, $|A(u_1)| \ge 2$, $|A(u_2)| \ge 2$, $|U| = 2$ since otherwise we are done by SDR. Thus, $A(u_1) = A(u_2)$. Since $v_2v_5 \in G$, $B \in A(u_2)$. However, $B \notin A(u_1)$, which is a contradiction.

 \textbf{Case 2.1.2.2:} $f(v_3) = C$ and $f(v_5) = B$. Since $|A(u_4)| \ge 1$, $|A(u_2)| \ge 1$, $|A(u_1)| \ge 2$, $B \in A(u_2)$, and $B \notin A(u_1)$, we have $|U| \ge 3$. The only case we are not done by SDR is when $A(u_2) = A(u_4)$, which contradicts with the fact that $B \in A(u_2)$ and $B \notin A(u_4)$.

 \textbf{Case 2.1.2.3:} $f(v_3) = B$ and $f(v_5) = A$. Since $|A(u_4)| \ge 1$, $|A(u_2)| \ge 1$, and $|A(u_1)| \ge 3$, we must have $A(u_2) = A(u_4)$ and both have size $1$. Thus, we may assume $A(u_2) = A(u_4) = \{C\}$, $f(v_2') = D$, and $f(v_2'') = E$. Since $v_2''v_4'' \in G'$, $f(v_4'') = D$ and $f(v_4') = E$. We claim $f(v_5') = 1$ since otherwise we can recolor $v_5$ with $1$, and color $u_1, u_2, u_3, u_4, u_5$ with $1,C,1,A,B(E)$, which is a contradiction. If we can recolor $v_4''$ with $B$ or $C$, then we can color $u_1, u_2, u_3, u_4, u_5$ with $B(E),C,1,D,1$, which is a contradiction. Therefore, $w_4,w_4', w_4''$ are colored with $1,B,C$ or $B,1,C$ or $C,1,B$. Let $f(w_4) = B (C)$. If we can recolor $v_5$ with $E$ or $C(B)$, then we can color $u_1, u_2, u_3, u_4, u_5$ with $D(B),C,1,D,1$ or $D(E),C,1,D,1$, which is a contradiction. Thus, $\{f(w_5'), f(w_5'')\} = \{E,C(B)\}$. If we can recolor $v_4$ with $C(B)$, then we recolor $v_4''$ with $1$ and color $u_1, u_2, u_3, u_4, u_5$ with $E(B),C,1,D,1$, which is a contradiction.  If we can switch the colors of $v_4'$ and $v_4$, then we recolor $v_4''$ with $1$, and color $u_1, u_2, u_3, u_4, u_5$ with $B(E),C,1,D,1$, which is a contradiction. Therefore, $\{f(z_4'), f(z_4'')\} = \{1,C\}$. Then we recolor $v_4,v_4'', v_5$ with $A,1,D$, and color $u_1, u_2, u_3, u_4, u_5$ with $1,C,x,1,E$, where $x$ is depend on the color of $v_3'$ (if $f(v_3') \neq D$ then $x = D$; if $f(v_3') = D$ then we recolor $v_3$ with $1$ and $x = B$) and we reach a contradiction.

 \textbf{Case 2.2:} One of $v_2,v_3,v_4,v_5$ is colored with $1$. We have two cases up to symmetry.

 \textbf{Case 2.2.1:} $f(v_2) = 1$. Since $v_2v_5, v_2''v_4'', v_3v_4 \in G$, we have seven subcases depending on the colors of $v_3, v_4, v_5$ up to symmetry. One of $f(v_4')$ and $f(v_4'')$ must be $1$ since otherwise we recolor $v_4$ with $1$, which is solved in Case 2.1. Similarly, we know $f(v_3') = 1$.

 \textbf{Case 2.2.1.1:} $f(v_3) = A, f(v_4) = B, f(v_5) = A$. Since $|A(u_1)| \ge 3, |A(u_2)| \ge 2$, and $|A(u_4)| \ge 2$, we color $u_3,u_5$ with $1$ and $u_1,u_2,u_4$ with SDR, which is a contradiction.

 \textbf{Case 2.2.1.2:} $f(v_3) = A, f(v_4) = B, f(v_5) = C$. Color $u_3,u_5$ with $1$. Since $|A(u_1)| \ge 2, |A(u_2)| \ge 2$, and $|A(u_4)| \ge 2$, the only case we are not done with SDR is when $A(u_1) = A(u_2) = A(u_4)$ and $|A(u_1)| = |A(u_2)| = |A(u_4)| = 2$. However, we know by $v_2v_5 \in G'$ that $C \in A(u_2)$, which is a contradiction with $C \notin A(u_1)$.

 \textbf{Case 2.2.1.3:} $f(v_3) = B, f(v_4) = C, f(v_5) = A$. Since $|A(u_1)| \ge 3, |A(u_2)| \ge 1$, and $|A(u_4)| \ge 1$, the only case not done by SDR is when $A(u_2) = A(u_4)$ and both have size one. Say $A(u_2) = \{E\}$. Then $f(v_2') = C, f(v_2'') = f(v_4') = D$, and $f(v_4'') = 1$. We color $u_1,u_2,u_3,u_4,u_5$ with $1,E,A,1,B(D)$, which is a contradiction.

 \textbf{Case 2.2.1.4:} $f(v_3) = B, f(v_4) = A, f(v_5) = C$. We know $|A(u_2)| \ge 1$ and $|A(u_3)| \ge 2$.  In case $|A(u_5)| \ge 2$, by $v_2v_5 \in G'$ we know $C \in A(u_2)$. However, $C \notin A(u_5)$ and we can color $u_2,u_3,u_5$ with SDR ($u_1,u_4$ with $1$), which is a contradiction. In case $|A(u_5)| = 1$, we recolor $v_5$ with $1$. Now $|A(u_5)| \ge 2$ and thus $A(u_5) = A(u_3)$ with size two since otherwise we are done by SDR (color $u_1,u_4$ with $1$). Therefore, $f(v_5') = B$ and $f(v_4'') = D(E)$. Since $v_2''v_4'' \in G'$, $f(v_2'') \neq D(E)$ and we color $u_1, u_2, u_3, u_4, u_5$ with $1,C,D(E),1,E(D)$, which is a contradiction.

 \textbf{Case 2.2.1.5:} $f(v_3) = B, f(v_4) = A, f(v_5) = B$. We color $u_1$ and $u_4$ with $1$. We know $|A(u_5)| \ge 1$, $|A(u_2)| \ge 1$, and $|A(u_3)| \ge 2$. In case $|A(u_5)| = 1$, we recolor $v_5$ with $1$ and $B \in A(u_5)$ and $|A(u_5)| \ge 2$. We also know $B \notin A(u_3)$ and therefore $|A(u_5) \cup A(u_3)| \ge 3$. We color $u_2,u_3, u_5$ with SDR, which is a contradiction. Thus, we know at least one of $v_5'$ and $v_4''$ is colored with $1$. We also know $A(u_3) = A(u_5)$, both have size $2$, and $|A(u_2) \cup A(u_3) \cup A(u_5)| = 2$. Say $A(u_3) = A(u_5) = \{C,D\}$ and $A(u_2) = C$. Since we can also color $u_1,u_3$ with $1$, we similarly can also obtain $A(u_4) = A(u_5)$ and thus $A(u_4) = \{C,D\}$. By symmetry, we assume $f(v_2'') = E$ and $f(v_2') = D$. Since $v_2''v_4'' \in G$, we know $f(v_4') = f(v_5') = E$ and $f(v_4'') = 1 = f(w_5') = f(z_4')$. If we can recolor $v_5$ with $C$ or $D$, then it is already solved in Case 2.2.1.4. Thus, $f(w_5'') = f(z_4'') = D (C)$ and $f(w_4) = C (D)$. We switch the colors of $v_5$ and $v_4$, and color $u_1, u_2, u_3, u_4, u_5$ with $1,C,A,1,D$, which is a contradiction.

 \textbf{Case 2.2.1.6:} $f(v_3) = B, f(v_4) = C, f(v_5) = D$. We color $u_3,u_5$ with $1$. We know $|A(u_1)| \ge 2, |A(u_2)| \ge 1,$ and $|A(u_4)| \ge 1$. Since $v_2v_5 \in G'$, $D \in A(u_2)$. But $D \notin A(u_1)$ and thus $|A(u_1) \cup A(u_2)| \ge 3$. The only case we are not done by SDR is when $A(u_2) = A(u_4)$ and both have size $1$, which contradicts $D \notin A(u_4)$.

 \textbf{Case 2.2.1.7:} $f(v_3) = B, f(v_4) = C, f(v_5) = B$. We color $u_3,u_5$ with $1$. We know $|A(u_1)| \ge 2$, $|A(u_2)| \ge 1$, and $|A(u_4)| \ge 2$. We may assume $A(u_1) = A(u_4) = \{D,E\}$. Since $v_2''v_4'' \in G'$, we know $f(v_2') = C$ and $f(v_2'') = D$. We also know either $f(v_4') = 1$ and $f(v_4'') = A$ or $f(v_4') = A$ and $f(v_4'') = 1$. In the former case, $f(w_4) = 1$. We claim $f(v_5') = D$ since otherwise we color $u_1, u_2, u_3, u_4, u_5$ with $1,E,A,1,D$, which is a contradiction. Then we recolor $v_5$ with $1$ and color $u_1, u_2, u_3, u_4, u_5$ with $1,E,A,1,B$, which is a contradiction. In the latter case, we know $f(z_4') = 1$. Similarly to the former case, we know $f(v_5') = D$. We also know $\{f(w_5''), f(w_4) \} = \{A,E\}$ since otherwise we can recolor $v_5$ with $A(E)$ and it was solved in Case 2.2.1.6. If we can recolor $v_4$ with $D$, then we color $u_1, u_2, u_3, u_4, u_5$ with $1,E,A,1,C$, which is a contradiction. Thus, $f(z_4'') = D$. We switch the colors of $v_4$ and $v_5$, and color  $u_1, u_2, u_3, u_4, u_5$ with $B,E,1,D,1$, which is a contradiction.

 \textbf{Case 2.2.2:} $f(v_3) = 1$. Similarly to Case 2.2.1, we have seven subcases. We know $f(v_2') = 1$, since otherwise we recolor $v_2$ with $1$ and it was solved in Case 2.1. For the same reason, one of $f(v_4'')$ and $f(v_5')$ is colored with $1$.

 \textbf{Case 2.2.2.1:} $f(v_2) = A, f(v_4) = A, f(v_5) = B$. Color $u_1$ and $u_4$ with $1$. We know $|A(u_2)| \ge 3$, $|A(u_3)| \ge 2$, and $|A(u_5)| \ge 2$. We can color $u_2,u_3,u_5$ by SDR, which is a contradiction.

 \textbf{Case 2.2.2.2:} $f(v_2) = A, f(v_4) = B, f(v_5) = C$. Color $u_1$ and $u_4$ with $1$. We know $|A(u_2)| \ge 3$, $|A(u_3)| \ge 1$, and $|A(u_5)| \ge 1$. Thus, the only case we are not done by SDR is when $A(u_3) = A(u_5)$ and both have size one, say $A(u_3) = A(u_5) = \{E\}$. Since $v_2''v_4'' \in G'$, $f(v_2'') = f(v_5') = D$, $f(v_3') = C$, and $f(v_4'') = 1$. We also know $f(w_5') = 1$ since otherwise we can recolor $v_5'$ with $1$. Then we claim $\{f(w_5''), f(w_4)\} = \{A,E\}$ since otherwise we can recolor $v_4$ with $A$ or $E$ and can color $u_2, u_3, u_5$ with $B,E,C$, which is a contradiction. Say $f(w_4) = E (A)$. If we can switch the colors of $w_4,v_4''$, then we can recolor $v_5$ with $1$, which is a contradiction. Thus, either $ 1 \in \{f(w_4'), f(w_4'')\}$ or $f(v_4') = E (A)$. If we can switch the colors of $v_5$ and $v_4''$, then it was solved in Case 2.1. Thus, either $C \in \{f(w_4'), f(w_4'')\}$ or $f(v_4') = C$. Either $A \in \{f(w_4'), f(w_4'')\}$ or $f(v_4') = A (E)$ since otherwise we can recolor $v_4''$ with $A$ and $v_5$ with $1$. Thus, the colors on $v_4', w_4', w_4''$ can be $C,1,A(E)$ or $A(E),1,C$ or $E(A),A(E),C$. We know $\{f(z_4'), f(z_4'')\} = \{1,D\}$ since otherwise we either recolor $v_4'$ with $1$ or recolor $v_4$ with $D$ and color $u_2, u_3, u_5$ with $C,E,B$, which is a contradiction. In case $v_4', w_4', w_4''$ are colored with  $C,1,A(E)$ or $E(A),A(E),C$, we recolor $v_4$ with $A(E)$ and color $u_2, u_3, u_5$ with $C,B,E$, which is a contradiction. In case $v_4', w_4', w_4''$ are colored with $A(E),1,C$, we switch the colors of $v_4$ and $v_5$, and color $u_2, u_3, u_5$ with $C,B,E$, which is a contradiction.

 \textbf{Case 2.2.2.3:} $f(v_2) = C, f(v_4) = A, f(v_5) = B$. Color $u_1$ and $u_4$ with $1$. We know $|A(u_2)| \ge 2$, $|A(u_3)| \ge 1$, and $|A(u_5)| \ge 2$. Since $v_2v_5 \in G'$, $B \in A(u_2)$. However, $B \notin A(u_5)$ and $|A(u_2) \cup A(u_5)| \ge 3$. We are done by SDR. 

 \textbf{Case 2.2.2.4:} $f(v_2) = B, f(v_4) = B, f(v_5) = A$. Color $u_1$ and $u_4$ with $1$. We know $|A(u_2)| \ge 2$, $|A(u_3)| \ge 2$, and $|A(u_5)| \ge 2$. The only case we are not done by SDR is when $A(u_2) = A(u_3) = A(u_5)$ and all have size two. Say $A(u_2) = A(u_3) = A(u_5) = \{D,E\}$. Then $f(v_2'') = f(v_5') = C$, $f(v_3') = A$, and $f(v_4'') = 1$. If we can recolor $v_2$ with $A,D,E$, then we can color $u_2,u_3,u_5$ with SDR. Recall $N(v_2'') = \{v_2, v_3, y_2\}$ and $N(y_2) = \{v_2'', y_2', y_2''\}$. Thus, $f(N(v_2')-v_2) \cup f(y_2) = \{A,D,E\}$ and $1 \in \{f(y_2'), f(y_2'')\}$. If we can switch the colors of $v_2$ and $v_2''$, then we can color $u_1,u_2,u_3,u_4,u_5$ with $B,1,D,1,E$, which is a contradiction. Thus, $\{f(y_2'), f(y_2'')\} = \{1,B\}$. We recolor $v_2''$ with $D$ and color $u_1,u_2,u_3,u_4,u_5$ with $1,C,E,1,D$, which is again a contradiction.



 \textbf{Case 2.2.2.5:} $f(v_2) = B, f(v_4) = C, f(v_5) = A$. Color $u_1$ and $u_4$ with $1$. We know $|A(u_2)| \ge 2$, $|A(u_3)| \ge 1$, and $|A(u_5)| \ge 2$. Since $v_2v_5 \in G'$, $B \in A(u_5)$. However, $B \notin A(u_2)$ and $|A(u_2) \cup A(u_5)| \ge 3$. We are done by SDR.

 \textbf{Case 2.2.2.6:} $f(v_2) = C, f(v_4) = C, f(v_5) = B$. We first color $u_1$ and $u_4$ with $1$. Then $|A(u_2)| \ge 2$, $|A(u_3)| \ge 2$, and $|A(u_5)| \ge 1$. Since $v_2v_5 \in G'$, $B \in A(u_2)$. The only case we are not done by SDR is when $A(u_2) = A(u_3)$ and both have size two. Say $A(u_2) = A(u_3) = \{B,E\}$. Then we color $u_2$ and $u_5$ with $1$. Then $|A(u_1)| \ge 1$, $|A(u_3)| \ge 2$, and $|A(u_4)| \ge 2$. At least one of $A$ or $D$ is available at $u_4$ and thus $|A(u_3) \cup A(u_4)| \ge 3$. We are done by SDR. 

 \textbf{Case 2.2.2.7:} $f(v_2) = D, f(v_4) = C, f(v_5) = B$. Color $u_1$ and $u_4$ with $1$. We know $|A(u_2)| \ge 2$, $|A(u_3)| \ge 1$, and $|A(u_5)| \ge 1$. Since $v_2v_5 \in G'$, $B \in A(u_5)$ and $D \in A(u_5)$. Since $D \notin A(u_3)$, we know $|A(u_3) \cup A(u_5)| \ge 2$. Since $D \notin A(u_2)$, $|A(u_2) \cup A(u_5)| \ge 3$. We are done by SDR.

 \textbf{Case 2.3:} None of $v_2, v_3, v_4, v_5$ are colored with $1$. Since $v_2v_5, v_3v_4 \in G'$, we have nine subcases up to symmetry. Note that one of $f(v_1')$ and $f(v_1'')$, $f(v_2')$ and $f(v_2'')$, $f(v_2'')$ and $f(v_3')$, $f(v_5')$ and $f(v_4'')$, $f(v_4')$ and $f(v_4'')$ must be $1$.

 \textbf{Case 2.3.1:} $f(v_2) = A, f(v_3) = B, f(v_4) = A, f(v_5) = C$. Color $u_1$ and $u_4$ with $1$. Then $|A(u_2)| \ge 2$, $|A(u_3)| \ge 2$, and $|A(u_5)| \ge 2$. Since $v_2v_5 \in G'$, $C \in A(u_2)$ and $C \notin A(u_5)$. We are done by SDR since $|A(u_2) \cup A(u_5)| \ge 3$. 

 \textbf{Case 2.3.2:} $f(v_2) = A, f(v_3) = B, f(v_4) = A, f(v_5) = B$. Each $|A(u_i)| \ge 2$, where $i \in [5]$. Thus, the only case we are not done by SDR is when $A(u_1) = A(u_2) = A(u_3) = A(u_4) = A(u_5)$ and all have size two. Say $A(u_i) = \{D,E\}$, where $i \in [5]$. Since $v_2''v_4'' \in G'$, either $f(v_2'') = 1$ and $f(v_4'') = C$ or $f(v_2'') = C$ and $f(v_4'') = 1$. In the former case, $f(v_2') = f(v_3') = C$ and $f(v_4') = f(v_5') = f(w_4) = 1$; we also know $v_4''$ cannot be recolored with $D$ or $E$, and thus $\{f(w_4'), f(w_4'')\} = \{D,E\}$. If we can recolor $v_5$ with $D$ or $E$, then it is solved in Case 2.3.1. Thus, $\{f(w_5'), f(w_5'')\} = \{D,E\}$. We switch the colors of $v_4''$ and $v_5$, and color $u_1,u_2,u_3,u_4,u_5$ with $B,1,E,D,1$, which is a contradiction. In the latter case, $f(v_2') = f(v_3') = 1$ and $f(v_4') = f(v_5') = C$. We also know $1 \in \{f(w_5'), f(w_5'')\}$ and $1 \in \{f(z_4'), f(z_4'')\}$. If we can recolor $v_5$ with $D$ or $E$, then it is solved in Case 2.3.1. Similarly, we cannot recolor $v_4$ with $D$ or $E$ and thus we may assume $f(w_4) = D$ and $\{f(z_4'), f(z_4'')\} = \{f(w_5'), f(w_5'')\} = \{1,E\}$. If we can switch the colors of $v_4''$ and $v_5$, then it is solved in Case 2.2. Similarly, we cannot switch the colors of $v_4''$ and $v_5$ and thus $\{f(w_4'), f(w_4'')\} = \{A,B\}$. We recolor $w_4,v_4'',v_5$ with $1,E,1$, which is solved in Case 2.2.

 \textbf{Case 2.3.3:} $f(v_2) = A, f(v_3) = C, f(v_4) = B, f(v_5) = C$. Color $u_3$ and $u_5$ with $1$. Then $|A(u_1)| \ge 2$, $|A(u_2)| \ge 2$, and $|A(u_4)| \ge 2$. The only case we are not done by SDR is when $A(u_1) = A(u_2) = A(u_4)$ and all have size two. Thus, $A(u_1) = A(u_2) = A(u_4) = \{D,E\}$. Since $v_2''v_4'' \in G'$, $f(v_4'') = 1$, $f(v_2'') \neq 1,B$, $f(v_2') = f(v_3') = 1$, and $f(v_4') = A$. However, $B \in A(u_2)$ and $B \notin A(u_4)$, which is a contradiction.

 \textbf{Case 2.3.4:} $f(v_2) = A, f(v_3) = C, f(v_4) = B, f(v_5) = D$. Color $u_3$ and $u_5$ with $1$. Then $|A(u_1)| \ge 2$, $|A(u_2)| \ge 2$, and $|A(u_4)| \ge 1$. Since $v_2v_5 \in G'$, $D \in A(u_2)$ and $D \notin A(u_1)$. We are done by SDR since $|A(u_1) \cup A(u_2)| \ge 3$. 

 \textbf{Case 2.3.5:} $f(v_2) = C, f(v_3) = A, f(v_4) = C, f(v_5) = B$. Similarly to Case 2.3.1, we are done by SDR.

 \textbf{Case 2.3.6:} $f(v_2) = C, f(v_3) = A, f(v_4) = D, f(v_5) = B$. Color $u_1$ and $u_4$ with $1$. Each $|A(u_i)| \ge 1$, where $i \in [5]$. Since $v_2v_5 \in G'$, $B \in A(u_2)$ and $C \in A(u_5)$. Note that $C \notin A(u_3)$, we know $A(u_3) = \{B\}$ since otherwise we are done by SDR. Since $B,C \notin A(u_1)$, we can color $u_2, u_4$ with $1$ and $u_1,u_3,u_5$ by SDR, which is a contradiction.

 \textbf{Case 2.3.7:} $f(v_2) = B, f(v_3) = C, f(v_4) = B, f(v_5) = D$. Each $|A(u_i)| \ge 1$, where $i \in [5]$, and $|A(u_3)| \ge 2$. Since $v_2v_5 \in G'$, $D \in A(u_2)$. By $C,D \notin A(u_4)$, we claim $C \notin A(u_1)$ since otherwise we color $u_3,u_5$ with $1$ and $u_1, u_2, u_4$ with SDR. Thus, $\{f(v_1'), f(v_1'')\} = \{1,C\}$ and $A(u_4) = \{E\}$. It also implies $A(u_5) = \{E\}$ and $A(u_3) = \{D,E\}$ since otherwise we are done by SDR. Since $v_2''v_4'' \in G'$, $f(v_5') = C$ and $f(v_4'') = 1$. Then $f(v_4') = A$, $1 \in \{f(w_5'), f(w_5'')\}$, and $1 \in \{f(z_4'), f(z_4'')\}$. Say $f(w_5') = f(z_4') = 1$. If we can recolor $v_4$ with $C$ or $E$, then we can color $u_1, u_2, u_3, u_4, u_5$ with $E,D,1,B,1$, which is a contradiction. Therefore, $\{f(w_4), f(z_4'')\} = \{C,E\}$. If we can switch the colors of $v_4''$ and $v_5$, then it is solved in Case 2.2. Similarly, we cannot switch the colors of $v_4''$ and $v_4$, and we know $\{f(w_4'), f(w_4'')\} = \{B,D\}$. We recolor $w_4,v_4''$ with $1,E$, and recolor $v_5$ with $1$, which is solved in Case 2.2.

 \textbf{Case 2.3.8:} $f(v_2) = B, f(v_3) = C, f(v_4) = B, f(v_5) = C$. We know $|A(u_1)| \ge 1$, $|A(u_2)| \ge 1$, $|A(u_3)| \ge 2$, $|A(u_4)| \ge 2$, and $|A(u_5)| \ge 1$. Since $v_2''v_4'' \in G'$, at least one of $v_2''$ and $v_4''$ is colored with a $2$-color. Say $v_4''$ is colored with a $2$-color and this color can be $A$ or $D$ up to symmetry. We also know $f(v_5') = f(w_4) = f(v_4') = 1$. In the former case, $A(u_4) = A(u_5) = \{D,E\}$, $A(u_1) \subseteq \{D,E\}$. Thus, $A(u_3) = \{D,E\}$ since otherwise we color $u_2,u_5$ with $1$ and $u_1, u_3, u_4$ with SDR. If we can recolor $v_5$ with $D$ or $E$, then it is solved in Case 2.3.7. Therefore, $\{f(w_5'), f(w_5'')\} = \{D,E\}$. If we can recolor $v_4''$ with $D$ or $E$, then $A \in A(u_4)$ and $|A(u_3) \cup A(u_4)| \ge 3$ (done by SDR). Thus, $\{f(w_4'), f(w_4'')\} = \{D,E\}$. We switch the colors on $v_4''$ and $v_5$, and $C \in A(u_1)$, which implies $|A(u_1) \cup A(u_3) \cup A(u_4)| \ge 3$. We are done by SDR. In the latter case, $A(u_1) \subseteq \{D,E\}$, $A(u_5) = \{E\}$ and $A(u_4) = \{A,E\}$. Thus, $A(u_1) = \{E\}$ since otherwise we color $u_2,u_5$ with $1$ and $u_1,u_3,u_4$ with SDR. For a similar reason, $A(u_3) = \{A,E\}$. We are done by a similar argument to the former case.

 \textbf{Case 2.3.9:} $v_2,v_3,v_4,v_5$ has colors $B,D,C,E$. Similarly to Case 2.3.6, we are done by SDR. \hfill \qed


\subsection{Configuration ``$3|5|3$''}

\noindent\textbf{Proof of Lemma~\ref{no353}:}
Let the triangle be $u_1u_2u_3$ and their neighbors out of the triangle be $v_1, v_2, v_3$ respectively; let $v_2w_1w_2w_3w_4v_2$ be the five cycle and it is joined with $u_1u_2u_3$ by the edge $u_2v_2$; let $w_1'z_1z_2$ be the other triangle joined with $v_2w_1w_2w_3w_4v_2$ by the edge $w_1w_1'$ Let $N(v_1) = \{u_1, v_1', v_1''\}$, $N(v_2) = \{u_2, w_1, w_4\}$, $N(v_3) = \{u_3, v_3', v_3''\}$, $N(w_1) = \{v_2, w_1', w_2\}$, $N(w_2) = \{w_1, w_3, w_2'\}$, $N(w_3) = \{w_2, w_4\, w_3'\}$, $N(w_4) = \{v_2, w_3, w_4'\}$; let $N(w_1') = \{w_1, z_1, z_2\}$, $N(z_1) = \{w_1', z_2, z_1'\}$, and $N(z_2) = \{w_1', z_1, z_2'\}$. We contract $u_1, u_2, u_3$ to a single vertex $u$ to obtain a subcubic planar graph $G'$. By the inductive hypothesis, $G'$ has a good coloring $f$. By Lemma~\ref{only-case} and Remark~\ref{neighbor}, we can assume $f(v_1) = A$, $f(v_2) = B$, $f(v_3) = C$, $\{f(v_1'), f(v_1'')\} = \{f(w_1), f(w_4)\} = \{f(v_3'), f(v_3'')\} = \{1,D\}$. We extend $f$ to $G$.

 \textbf{Case 1:} $f(w_1) = D$ and $f(w_4) = 1$. If $1 \notin \{f(w_1'), f(w_2)\}$ then we can recolor $w_1$ with $1$ and it contradicts our assumption. Thus, $1 \in \{f(w_2), f(w_1')\}$.

 \textbf{Case 1.1:} $f(w_2) = 1$. If $v_2$ can be recolored with $A$ or $C$ then we can color $u_1, u_2, u_3$ with $1,E,B$, which is a contradiction. Thus, $A,C \in \{f(w_1'), f(w_3), f(w_4')\}$. By symmetry, we have the following cases.

 \textbf{Case 1.1.1:} $f(w_1') = A$ and $f(w_3) = C$. Then $1 \in \{f(z_1), f(z_2)\}$ since otherwise we recolor $w_1'$ with $1$, which contradicts our assumption. Without loss of generality, say $f(z_1) = 1$. Then $f(w_4')$ can be $A, D$ or $E$.

 \textbf{Case 1.1.1.1:} $f(w_4') = A$. If we can switch the colors on $v_2$ and $w_1$, then we can color $u_1, u_2, u_3$ with $B,1,E$, which is a contradiction. If we can recolor $w_1$ with $E$, then we can color $u_1, u_2, u_3$ with $E,D,1$, which is a contradiction. Therefore, $\{f(w_2'), f(z_2)\} = \{B,E\}$. Furthermore, we know $\{f(z_1'), f(z_2')\} = \{C,D\}$ since otherwise we can recolor $w_1'$ with $C$ or $D$, $w_1$ with $A$, $v_2$ with $E$, and color $u_1, u_2, u_3$ with $1, D, B$. When $f(w_2') = B (E)$ and $f(z_2) = E (B)$, then we recolor $w_1'$ with $B (E)$, $w_1$ with $A$, $v_2$ with $E (B)$, and color $u_1, u_2, u_3$ with $1,D,B (E)$, which is a contradiction.

 \textbf{Case 1.1.1.2:} $f(w_4') = D$. If we can recolor $w_1$ with $B$ and $v_2$ with $E$, then we can color $u_1, u_2, u_3$ with $B,D,1$. Thus, $B \in \{f(z_2), f(w_2')\}$. Similarly to Case 1.1.1.1, we know $\{f(z_2), f(w_2')\} = \{B, E\}$, $\{f(z_1'), f(z_2')\} = \{C,D\}$. When $f(w_2') = B (E)$ and $f(z_2) = E (B)$, we recolor $w_1'$ with $B (E)$, $w_1$ with $A$, $v_2$ with $E (B)$, and color $u_1, u_2, u_3$ with $1, D, B (E)$, which is a contradiction.

 \textbf{Case 1.1.1.3:} $f(w_4') = E$. Similarly to Case 1.1.1.1, $\{f(w_2'), f(z_2)\} = \{B,E\}$ and $\{f(z_1'), f(z_2')\} = \{C,D\}$. When $f(w_2') = B (E)$ and $f(z_2) = E (B)$, then we recolor $w_1'$ with $B (E)$, $w_1$ with $A$, $v_2$ with $D (B)$, and color $u_1, u_2, u_3$ with $1,E (D),B (E)$, which is a contradiction. 

 \textbf{Case 1.1.2:} $f(w_1') = A$ and $f(w_4') = C$. Then $w_3$ can be $A$ or $E$. We know $1 \in \{f(z_1), f(z_2)\}$ since otherwise we can recolor $w_1'$ with $1$, which contradicts our assumption. Say $f(z_1) = 1$.

 \textbf{Case 1.1.2.1:} $f(w_3) = A$. Similarly to Case 1.1.1, we know $\{f(z_2), f(w_2')\} = \{B, E\}$. Then we can recolor $w_1$ with $C$, and color $u_1, u_2, u_3$ with $1,D,E$, which is a contradiction.

 \textbf{Case 1.1.2.2:} $f(w_3) = E$. Similarly to Case 1.1.1, we know $B \in \{f(z_2), f(w_2')\}$. If $C \notin \{f(z_2), f(w_2')\}$, then we can recolor $w_1$ with $C$, and color $u_1, u_2, u_3$ with $D,1,E$. Therefore, $\{f(z_2), f(w_2')\} = \{C,B\}$. If we can recolor $w_1'$ with $D (E)$ and $w_1$ with $A$, then it contradicts our assumption of the case. Therefore, $\{f(z_1'), f(z_2')\} = \{D,E\}$. When $f(z_2) = C (B)$ and $f(w_2') = B (C)$, we recolor $w_1'$ with $B (C)$, $w_1$ with $A$, $v_2$ with $D$, and color $u_1, u_2, u_3$ with $E,1,B$, which is a contradiction.

 \textbf{Case 1.1.3:} $f(w_3) = A$ and $f(w_4') = C$. By symmetry, the cases when $f(w_1') = A,C$ were already solved in previous cases. Thus, $f(w_1')$ can be $1$ or $E$.

 \textbf{Case 1.1.3.1:} $f(w_1') = 1$. Similarly to previous cases, we know $\{B,E\} \subseteq \{f(z_1), f(z_2), f(w_2')\}$. If we can recolor $w_1$ with $C$, then we can color $u_1, u_2, u_3$ with $D,1,E$, which is a contradiction. Thus, $\{f(z_1), f(z_2), f(w_2')\} = \{C,B,E\}$.

If $\{f(z_1), f(z_2)\} = \{C,B\}$ (say $f(z_1) = C$ and $f(z_2) = B$) and $f(w_2') = E$, then $f(z_1') = 1$ since otherwise we can recolor $z_1$ with $1$, $w_1'$ with $C$, and it was already solved in Case 1.1.1 by symmetry. We also know $f(z_2') = A$ since otherwise we can recolor $w_1'$ with $A$ and it was already solved in Case 1.1.2.1. Then we recolor $z_2$ with $1$, $w_1'$ with $E$, $w_1$ with $B$, $v_2$ with $D$, and color $u_1, u_2, u_3$ with $E,1,B$, which is a contradiction.

If $\{f(z_1), f(z_2)\} = \{C,E\}$ (say $f(z_1) = C$ and $f(z_2) = E$) and $f(w_2') = B$, then similarly to previous paragraph we know that $f(z_1') = 1$ and $f(z_2') = A$. Then we recolor $z_2$ with $1$, $w_1'$ with $B$, $w_1$ with $E$, $v_2$ with $D$, and color $u_1, u_2, u_3$ with $E,1,B$, which is a contradiction. 

If $\{f(z_1), f(z_2)\} = \{B,E\}$ (say $f(z_1) = B$ and $f(z_2) = E$) and $f(w_2') = C$, then similarly to previous paragraph we know that $\{f(z_1'), f(z_2')\} = \{A,C\}$. Then we recolor $z_2$ with $1$, $w_1'$ with $D$, $w_1$ with $E$, and color $u_1, u_2, u_3$ with $E,D,1$, which is a contradiction.

 \textbf{Case 1.1.3.2:} $f(w_1') = E$. Then $1 \in \{f(z_1), f(z_2)\}$ since otherwise we recolor $w_1'$ with $1$ and it was solved in Case 1.1.3.1. Say $f(z_1) = 1$. We also know $C \in \{f(z_2), f(w_2')\}$ since otherwise we recolor $w_1$ with $C$ and it was solved in Case 1.1.1.1. We claim $B \in \{f(z_2), f(w_2')\}$ since otherwise we switch the colors of $v_2$ and $w_1$, and color $u_1, u_2, u_3$ with $E,1,B$. Furthermore, $A \in \{f(z_1'), f(z_2')\}$ since otherwise we recolor $w_1'$ with $A$ and it was solved in Case 1.1.2.1. We claim $D \in \{f(z_1'), f(z_2')\}$ since otherwise we can recolor $w_1'$ with $D$, $w_1$ with $B$, $v_2$ with $E$, and color $u_1, u_2, u_3$ with $B,D,1$, which is a contradiction. When $f(z_2) = C (B)$ and $f(w_2') = B (C)$, we recolor $w_1'$ with $B (C)$, $w_1$ with $E$, $v_2$ with $D$, and color $u_1, u_2, u_3$ with $E,1,B$, which is a contradiction.

 \textbf{Case 1.2:} $f(w_1') = 1$ and $f(w_2) \neq 1$. We know $f(w_2') = 1$ since otherwise we can recolor $f(w_2)$ with $1$ and it contradicts our assumption. If $v_2$ can be recolored with $A$ or $C$ then we color $u_1, u_2, u_3$ with $1,E,B$, which is a contradiction. Thus, $A,C \in \{f(w_2), f(w_3), f(w_4')\}$. By symmetry, we have the following cases.

 \textbf{Case 1.2.1:} $f(w_2) = A$ and $f(w_3) = C$ (or $f(w_2) = C$ and $f(w_3) = A$). Then $w_4'$ can be $A,D$, or $E$. If we can recolor $w_1$ with $E$, then it contradicts our assumption. Thus, $E \in \{f(z_1), f(z_2)\}$. If we can switch the colors of $w_1$ and $w_1'$, then it contradicts our assumption. Thus, $D \in \{f(z_1'), f(z_2')\}$. 

 \textbf{Case 1.2.1.1:} $f(w_4') = A$ or $E$. If we can switch the colors of $v_2$ and $w_1$, then we can color $u_1, u_2, u_3$ with $B,1,E$, which is a contradiction. Thus, $\{f(z_1), f(z_2)\} = \{B,E\}$. Say $f(z_1) = B$ and $f(z_2) = E$. If we can recolor $w_1'$ with $C$, then we can recolor $w_1$ with $1$, and it contradicts our assumption. Thus, $\{f(z_1'), f(z_2')\} = \{C,D\}$. Then we switch the colors of $w_1'$ and $z_2$, recolor $w_1$ with $1$, and color $u_1,u_2,u_3$ with $E,D,1$, which is a contradiction.

 \textbf{Case 1.2.1.2:} $f(w_4') = D$. If $B \notin \{f(z_1), f(z_2)\}$, then we can recolor $w_1'$ with $B$, $w_1$ with $1$, $v_2$ with $E$, and color $u_1, u_2, u_3$ with $B,D,1$, which is a contradiction. Thus, $\{f(z_1), f(z_2)\} = \{B,E\}$. The remaining proof is the same with Case 1.2.1.1.


 \textbf{Case 1.2.2:} $f(w_2) = A$ and $f(w_4') = C$. Then $w_3$ must be $E$. Similarly to Case 1.2.1, we know $D \in \{f(z_1'), f(z_2')\}$ and $B \in \{f(z_1), f(z_2)\}$. Furthermore, if we can recolor $w_1$ with $C$, then it contradicts our assumption. Thus, $\{f(z_1), f(z_2)\} = \{C,B\}$. Say $f(z_1) = C$ and $f(z_2)$ is $B$. If we can recolor $w_1'$ with $E$, then we can recolor $w_1$ with $1$ and it contradicts our assumption. Thus, $\{f(z_1'), f(z_2')\} = \{D,E\}$. Then we switch the colors of $z_1$ and $w_1'$, recolor $w_1$ with $1$, and color $u_1, u_2, u_3$ with $1,D,E$.

 \textbf{Case 1.2.3:} $f(w_3) = A$ and $f(w_4') = C$. Then $w_2$ can be $C$ or $E$. The case when $f(w_2) = C$ was solved in Case 1.2.1 by symmetry. Now, $f(w_2) = E$. Similarly to Case 1.2.2, we know that $\{f(z_1), f(z_2)\} = \{C,B\}$. Say $f(z_1) = C$ and $f(z_2)$ is $B$. Similarly to Case 1.2.1, $D \in \{f(z_1'), f(z_2')\}$. If we can recolor $w_1'$ with $A$, then we can recolor $w_1$ with $1$ and it contradicts our assumption. Therefore, $\{f(z_1'), f(z_2')\} = \{A,D\}$. Then we switch the colors of $z_1$ and $w_1'$, recolor $w_1$ with $1$, and color $u_1, u_2, u_3$ with $1,D,E$.

 \textbf{Case 2:} $f(w_1) = 1$ and $f(w_4) = D$. Then we know $1 \in \{f(w_4'), f(w_3)\}$ since otherwise we can recolor $w_4$ with $1$ and it contradicts our assumption. 

 \textbf{Case 2.1:} $f(w_3) = 1$. Similarly to Case 1, we also know $\{A,C\} \subseteq \{f(w_1'), f(w_2), f(w_4')\}$.

 \textbf{Case 2.1.1:} $f(w_1') = A$ and $f(w_2) = C$ (or $f(w_1') = C$ and $f(w_2) = A$). If we can switch the colors of $v_2$ and $w_1$, then it contradicts our assumption; if we can recolor $w_1$ with $E$, then we recolor $v_2$ with $1$, and color $u_1, u_2, u_3$ with $1,B,E$. Therefore, $B,E \in \{f(z_1), f(z_2), f(w_2')\}$. If we can switch the colors of $w_1$ and $w_1'$, then we can recolor $v_2$ with $1$, and color $u_1, u_2, u_3$ with $1,B,E$. Therefore, either $f(w_2') = A$ or $1 \in \{f(z_1), f(z_2)\}$. Therefore, We know by symmetry $f(z_1), f(z_2), f(w_2')$ can be $1,E,B$ or $1,B,E$ or $B,E,A$, and $1 \in f(N(w_2')-w_2)$ since otherwise we can recolor $w_2'$ with 1. We also know $f(w_4')$ can be $1,A,C,$ or $E$.

 \textbf{Case 2.1.1.1:} $f(w_4') = 1$. If we can recolor $w_1'$ with $D$, then it contradicts our assumption that $\{A,C\} \subseteq \{f(w_1'), f(w_2), f(w_4')\}$. Thus, $D \in \{f(z_1'), f(z_2')\}$. When $f(z_1), f(z_2), f(w_2')$ is $B,E,A$, we know $f(z_1') = f(z_2') = 1$ since otherwise we can recolor $z_1$ or $z_2$ with $1$, and it contradicts our assumption on $f(z_1), f(z_2), f(w_2')$. However, it is a contradiction with $D \in \{f(z_1'), f(z_2')\}$. 

When $f(z_1), f(z_2), f(w_2')$ is $1,E (B),B (E)$,  we know $B (E) \in \{f(z_1'), f(z_2')\}$ since otherwise we can recolor $w_1'$ with $B (E)$ and it contradicts our assumption that $\{A,C\} \subseteq \{f(w_1'), f(w_2), f(w_4')\}$. We know $A, B, E \in f(N(w_4')-w_4) \cup f(w_3')$ since otherwise we can recolor $w_4$ with $A$, or $B,$ or $E$, recolor $v_2$ with $D$, and color $u_1, u_2, u_3$ with $E,1,B$, which is a contradiction. Therefore, we know $f(N(w_4')-w_4) \cup f(w_3') = \{A,B,E\}$. If we can switch the colors of $w_2$ and $w_4$, then we can color $u_1, u_2, u_3$ with $1,D,E$, which is a contradiction. Therefore, $f(N(w_2')-w_2) = \{1,D\}$. If we can switch the colors of $w_1'$ and $w_2$, then we can recolor $w_4$ with $C$ and color $u_1, u_2, u_3$ with $1,D,E$, which is a contradiction. Therefore, we must have $f(w_3') = A$ and $f(N(w_4')-w_4) = \{B,E\}$. If we can recolor $w_3$ with $A$ and recolor $w_3'$ with $1$ or recolor $w_3$ with $E$, then we can switch the colors of $w_1$ and $w_2$, and recolor $v_2$ with $1$, and color $u_1,u_2, u_3$ with $1,B,E$, which is a contradiction. Therefore, $f(N(w_3')-w_3) = \{1,E\}$. We recolor $w_1'$ with $C$, $w_1$ with $A$, $w_2$ with $1$, $w_3$ with $C$, $v_2$ with $1$, and color $u_1, u_2, u_3$ with $1,B,E$, which is a contradiction.

 \textbf{Case 2.1.1.2:} $f(w_4') = A$. We know $1 \in f(N(w_4')-w_4)$ since otherwise we can recolor $w_4'$ with $1$ and it was solved in Case 2.1.1.1. By Case 2.1.1, we know $f(z_1), f(z_2), f(w_2')$ can be $1,B,E$ or $1,E,B$ or $B,E,A$. We also know either $f(w_3') = B$ and $f(N(w_4')-w_4) = \{1,E\}$ or $f(w_3') = E$ and $f(N(w_4')-w_4) = \{1,B\}$, since otherwise we can recolor $w_4$ with $B$ or $E$, $v_2$ with $D$, and color $u_1, u_2, u_3$ with $E,1,B$, which is a contradiction. Furthermore, we must have $D \in f(N(w_3')-w_3)$ since otherwise we can switch the colors of $w_4$ and $w_3$, and color $u_1, u_2, u_3$ with $1,D,E$, which is a contradiction. 

When $f(z_1), f(z_2), f(w_2')$ is $B,E,A$, we recolor $v_2,w_4,w_3,w_2,w_1$ with $1,C,E(B),1,D$, and color $u_1,u_2,u_3$ with $E,B,1$, which is a contradiction. When $f(z_1), f(z_2), f(w_2')$ is $1,B (E), E(B)$, by a recoloring similarly to the previous sentence, we know $f(w_3') = E (B)$; we have $\{f(z_1'), f(z_2')\} = \{E(B),D\}$ since otherwise we can recolor $w_1',w_1,v_2$ with $D (E,B),A,1$, and color $u_1, u_2, u_3$ with $E,B,1$, which is a contradiction. Therefore, we can recolor $v_2,w_4,w_3,w_2,w_1,w_1'$ with $1,D,C,1,A,C$, and color $u_1, u_2, u_3$ with $E,B,1$, which is a contradiction. This implies $f(N(w_3')-w_3) = \{C,D\}$. However, we can recolor $v_2,w_4,w_3,w_2,w_1,w_1'$ with $1,D,B(E),1,A,C$, and color $u_1, u_2, u_3$ with $E,B,1$, which still reaches a contradiction.



 \textbf{Case 2.1.1.3:} $f(w_4') = E$. Then we know $1 \in f(N(w_4') - w_4)$ since otherwise we can recolor $w_4'$ with $1$ and it was solved in Case 2.1.1.1. Similarly to Case 2.1.1.1, we know by symmetry $f(z_1), f(z_2), f(w_2')$ can be $1,B,E$ or $1,E,B$ or $B,E,A$. Furthermore, either $f(w_3) = A$ and $f(N(w_4') - w_4) = \{1,B\}$ or $f(w_3) = B$ and $f(N(w_4') - w_4) = \{1,A\}$. Similarly to Case 2.1.1.1, when $f(z_1), f(z_2), f(w_2')$ is $B,E,A$, $f(z_1') = f(z_2') = 1$; we recolor $w_1'$ with $D$ and $v_2$ with $A$, and color $u_1, u_2, u_3$ with $E,1,B$, which is a contradiction.

Similarly to Case 2.1.1.1, when $f(z_1), f(z_2), f(w_2')$ is $1,B (E),E (B)$, $\{f(z_1'), f(z_2')\} = \{D,E(B)\}$. If we can switch the colors of $w_4$ and $w_3$, then we can color $u_1, u_2, u_3$ with $1,D,E$, which is a contradiction. Therefore, $D \in f(N(w_3')-w_3)$. We also claim that $C \in f(N(w_3')-w_3)$ since otherwise we recolor $w_1',w_1, w_2, w_3, v_2$ with $C,A,1,C,1$, and color $u_1, u_2, u_3$ with $1,B,E$, which is a contradiction. Thus, $f(N(w_3')-w_3) = \{C,D\}$ and we recolor $w_3,w_3',w_4$ with $A,1,1$, and color $u_1, u_2, u_3$ with $1,D,E$, which is a contradiction.

 \textbf{Case 2.1.1.4:} $f(w_4') = C$. Similarly to Case 2.1.1.1, we know $f(N(w_4')-w_4) \cup f(w_3') = \{A,B,E\}$. Then we recolor $w_4'$ with $1$, $v_2$ with $C$, and color $u_1, u_2, u_3$ with $E,B,1$, which is a contradiction.

 \textbf{Case 2.1.2:} $f(w_1') = A$ and $f(w_4') = C$. Then $f(w_2)$ can be $C$ or $E$. The case when $f(w_2) = C$ was already covered in Case 2.1.1.4. Thus, $f(w_2) = E$. Similarly to Case 2.1.1.1, we know by symmetry $f(z_1), f(z_2), f(w_2')$ can be $1,B,C$ or $1,C,B$ or $B,C,A$. Similarly to Case 2.1.1.1, when $f(z_1), f(z_2), f(w_2')$ is $B,C,A$, $f(z_1') = f(z_2') = 1$; we recolor $w_1'$ with $D$ and $v_2$ with $A$, and color $u_1, u_2, u_3$ with $E,1,B$.

Similarly to Case 2.1.1.1, when $f(z_1), f(z_2), f(w_2')$ is $1,B (C),C (B)$, $\{f(z_1'), f(z_2')\} = \{D,C(B)\}$. We also know either $f(N(w_4')-w_4) = \{1,B\}$ and $f(w_3') = A$  or $f(N(w_4')-w_4) = \{1,A\}$ and $f(w_3') = B$. No matter which case occur, $D \in \{N(w_3')-w_3\}$ since otherwise we can switch the colors of $w_3$ and $w_4$, and color $u_1, u_2, u_3$ with $E,D,1$, which is a contradiction; we also know $1 \in \{N(w_3')-w_3\}$ since otherwise we recolor $w_3', w_3, w_4$ with $1,A,1$, and color $u_1, u_2, u_3$ with $E,D,1$, which is a contradiction. Thus, $\{N(w_3')-w_3\} = \{1,D\}$. We recolor $w_1',w_1, w_2, w_3, v_2$ with $E,A,1,B,1$, and color $u_1, u_2, u_3$ with $E,B,1$, which is a contradiction.

 \textbf{Case 2.1.3:} $f(w_2) = A$ and $f(w_4') = C$. Then $f(w_1')$ can be $C,D,E$. The case when $f(w_1') = C$ was covered in Case 2.1.1.1 by symmetry. Thus, we have two cases.

 \textbf{Case 2.1.3.1:} $f(w_1') = D$. Similarly to Case 2.1.1.1, we know $f(z_1), f(z_2), f(w_2')$ can be $C,B,E$ or $C,E,B$ or $B,E,C$, and $f(z_1') = f(z_2') = 1$. Furthermore, we know $1 \in f(N(w_2')-w_2)$ since otherwise we can recolor $w_2'$ with $1$ and it is a contradiction with our assumption on $f(w_2')$. Similarly to Case 2.1.1.1, we also know either $f(N(w_4')-w_4) = \{1,B\}$ and $f(w_3') = E$ or $f(N(w_4')-w_4) = \{1,E\}$ and $f(w_3') = B$. 

When $f(z_1), f(z_2), f(w_2')$ is $B,E,C$, we recolor $w_1'$ with $C$ and it was solved in Case 2.1.1.1 by symmetry. Thus, we may assume $f(z_1), f(z_2), f(w_2')$ is $C,B (E),E (B)$. If $f(N(w_4')-w_4) = \{1,B\}$ and $f(w_3') = E$, then $f(N(w_2')-w_2) = \{1,C\}$ since otherwise we recolor $w_2$ with $C$ and it contradicts our assumption that $\{A,C\} \subseteq \{f(w_1'), f(w_2), f(w_4')\}$; we recolor $w_1',w_1, w_2,v_2$ with $1,A,B (E),1$, and color $u_1, u_2, u_3$ with $B,E,1$, which is a contradiction. If $f(N(w_4')-w_4) = \{1,E\}$ and $f(w_3') = B$, then we again know $f(N(w_2')-w_2) = \{1,C\}$; then we recolor $w_1',w_2,w_4$ with $A,D,A$, and color $u_1, u_2, u_3$ with $1,D,E$, which is a contradiction.

 \textbf{Case 2.1.3.2:} $f(w_1') = E$. Similarly to Case 2.1.1.1, we know $f(z_1), f(z_2), f(w_2')$ can be $C,B,E$ or $1,C,B$ or $1,B,C$. Furthermore, when $\{f(z_1), f(z_2)\} = \{C,B\}$ we know $f(z_1') = f(z_2') = 1$; then we recolor $w_1'$ with $D$ and it was solved in Case 2.1.3.1. Now $f(z_1), f(z_2), f(w_2')$ is $1,C (B),B (C)$. We know $D \in \{f(z_1'), f(z_2')\}$ since otherwise we recolor $w_1'$ with $D$ and it was already solved in Case 2.1.3.1. Similarly to Case 2.1.1.1, we know $\{f(z_1'), f(z_2')\} = \{B (C), D\}$, and either $f(N(w_4')-w_4) = \{1,B\}$ and $f(w_3') = E$ or $f(N(w_4')-w_4) = \{1,E\}$ and $f(w_3') = B$. We must have $f(N(w_4')-w_4) = \{1,E\}$ and $f(w_3') = B$ since otherwise we can switch the colors of $v_2$ and $w_2$, and color $u_1, u_2, u_3$ with $B,1,E$, which is a contradiction. Then we recolor $w_1',w_2, w_4$ with $A,E,A$, and color $u_1, u_2, u_3$ with $1,D,E$, which is a contradiction.

 \textbf{Case 2.2:} $f(w_4') = 1$. Similarly to Case 1, we also know $\{A,C\} \subseteq \{f(w_1'), f(w_2), f(w_3)\}$. 

 \textbf{Case 2.2.1:} $f(w_1') = A$ and $f(w_2) = C$. We know $f(w_3) \neq 1$ since it was solved in Case 2.1.1.1; we also know $f(w_3') = 1$ since otherwise we recolor $w_3$ with $1$ and it was solved in Case 2.1.1.1. Thus, $f(w_3)$ can be $A$ or $E$. If we can switch the colors of $v_2$ and $w_1$, then we color $u_1, u_2, u_3$ with $B,E,1$, which is a contradiction. Therefore, $B \in \{f(z_1), f(z_2), f(w_2')\}$.  If we can switch the colors of $w_1$ and $w_2$, then we recolor $v_2$ with $1$, and color $u_1, u_2, u_3$ with $B,E,1$, which is a contradiction. Thus, either $f(w_2') = 1$ or $C \in \{f(z_1), f(z_2)\}$. 

 \textbf{Case 2.2.1.1:} $f(w_3) = A$. If we can recolor $w_1$ with $E$, then we recolor $v_2$ with $1$, and color $u_1, u_2, u_3$ with $1,B,E$, which is a contradiction. Thus, $E \in \{f(z_1), f(z_2), f(w_2')\}$ and we know by symmetry $f(z_1), f(z_2), f(w_2')$ can be $B,E,1$ or $C,E,B$ or $C,B,E$.  If we can recolor $w_4$ with $B (E)$ and $v_2$ with $D$, then we color $u_1, u_2, u_3$ with $B,1,E$, which is a contradiction. Thus, $f(N(w_4')-w_4) = \{B,E\}$. When $f(z_1), f(z_2), f(w_2')$ is $B,E,1$, $f(N(w_2')-w_2) = \{B,E\}$ since otherwise we recolor $w_2$ with $E (B)$, $v_2$ with $C$, and color $u_1, u_2, u_3$ with $B,1,E$, which is a contradiction. We recolor $w_2$ with $D$, $w_4$ with $C$, and color $u_1, u_2, u_3$ with $B,1,E$, which is a contradiction. Similarly, when $f(z_1), f(z_2), f(w_2')$ is $C,E(B),B(E)$,$f(N(w_2')-w_2) = \{1,E (B)\}$, we recolor $w_2$ with $D$, $w_4$ with $C$, and color $u_1, u_2, u_3$ with $B,1,E$, which is a contradiction.

 \textbf{Case 2.2.1.2:} $f(w_3) = E$. Similarly to Case 2.2.1.1, $f(N(w_4')-w_4) = \{A,B\}$. If we can switch the colors of $w_1$ and $w_1'$, then we recolor $v_2$ with $1$, and color $u_1, u_2,u_3$ with $1,B,E$, which is a contradiction. Thus, either $f(w_2') = A$ or $1 \in \{f(z_1), f(z_2)\}$. Recall either $f(w_2') = 1$ or $C \in \{f(z_1), f(z_2)\}$ and thus $f(z_1),f(z_2),f(w_2')$ can be $1,B,1$ or $1,C,B$ or $B,C,A$ by symmetry. When $f(z_1),f(z_2),f(w_2')$ is $B,C,A$, $f(z_1) = f(z_2) = 1$ and we recolor $w_1'$ with $D$, $v_2$ with $A$, and color $u_1, u_2, u_3$ with $B,1,E$, which is a contradiction. When $f(z_1),f(z_2),f(w_2')$ is $1,B,1$, $\{f(z_1'),f(z_2')\} = \{D,E\}$ since otherwise we recolor $w_1'$ with $D$ or $E$ and it contradicts our assumption that $\{A,C\} \subseteq \{f(w_1'), f(w_2), f(w_3)\}$. If we can switch the colors of $v_2$ and $w_2$, then we color $u_1, u_2, u_3$ with $B,1,E$, which is a contradiction; if we can switch the colors of $w_2$ and $w_2$, then we color $u_1, u_2, u_3$ with $1,D,E$, which is a contradiction. Thus, $f(N(w_2')-w_2) = \{B,D\}$. We claim $f(N(w_3')-w_3) = \{A,D\}$ since otherwise we recolor $w_3$ with $A$ or $D$, $w_4$ with $E$, and color $u_1, u_2, u_3$ with $1,D,E$, which is a contradiction. We recolor $w_1',w_1,w_2,w_3,v_2$ with $C,A,E,C,1$, and color $u_1, u_2, u_3$ with $B,E,1$, which is a contradiction. When $f(z_1),f(z_2),f(w_2')$ is $1,C,B$, we know again $\{f(z_1'),f(z_2')\} = \{D,E\}$; we recolor $w_1',w_1, v_2$ with $B,A,1$, and color $u_1, u_2, u_3$ with $B,E,1$, which is a contradiction. 

 \textbf{Case 2.2.2:} $f(w_1') = A$ and $f(w_3) = C$. Then $f(w_2)$ must be $E$. Similarly to Case 2.2.1, we know $f(z_1), f(z_2),f(w_2')$ can be $1,B,1$ or $1,E,B$ or $B,E,A$ by symmetry. When $f(z_1), f(z_2),f(w_2')$ is $B,E,A$, similarly to Case 2.2.1 we know $f(z_1') = f(z_2') = 1$; we recolor $w_1'$ with $D$, $v_2$ with $A$, and color $u_1, u_2, u_3$ with $B,1,E$, which is a contradiction. When $f(z_1), f(z_2),f(w_2')$ is $1,B,1$, similarly to Case 2.2.1.2 we know $\{f(z_1'), f(z_2')\} = \{C,D\}$, $D \in f(N(w_2')-w_2) $, $f(N(w_4') - w_4) = \{A,B\}$. If we can recolor $w_3$ with $B (A)$, $v_2$ with $C$, then we color $u_1, u_2, u_3$ with $B,1,E$, which is a contradiction. Thus, $f(N(w_3') - w_3) = \{A,B\}$. We also know $f(N(w_2')-w_2) = \{B,D\}$ since otherwise we recolor $w_2,w_3,v_2$ with $B,E,C$, and color $u_1, u_2, u_3$ with $B,1,E$, which is a contradiction. We recolor $w_1',w_1, w_2, w_3, v_2$ with $E,A,C,E,1$, and color $u_1, u_2, u_3$ with $1,B,E$, which is a contradiction. When $f(z_1), f(z_2),f(w_2')$ is $1,E,B$, we again know $\{f(z_1'), f(z_2')\} = \{C,D\}$; we recolor $w_1',w_1,v_2$ with $B,A,1$, and color $u_1, u_2, u_3$ with $1,B,E$, which is a contradiction.

 \textbf{Case 2.2.3:} $f(w_2) = A$ and $f(w_3) = C$. Then $f(w_1')$ can be $D$ or $E$. 

 \textbf{Case 2.2.3.1:} $f(w_1') = D$. Similarly to Case 2.2.1, we know $f(z_1), f(z_2), f(w_2')$ can be $B,E,1$ or $A,B,E$ or $A,E,B$ by symmetry. Similarly to Case 2.2.1, we know $f(z_1') = f(z_2') = 1$. In all possibilities of $f(z_1), f(z_2), f(w_2')$, we recolor $w_1'$ with $C$ and it was solved in Case 2.2.1.1 by symmetry.

 \textbf{Case 2.2.3.2:} $f(w_1') = E$. Similarly to Case 2.2.1, we know $f(z_1), f(z_2), f(w_2')$ can be $A,B,E$ or $1,A,B$ or $1,B,1$ by symmetry. When $f(z_1), f(z_2), f(w_2')$ is $A,B,E$, we know $f(z_1') = f(z_2') = 1$; we recolor $w_1'$ with $C$ and it was solved in Case 2.2.1.1. When $f(z_1), f(z_2), f(w_2')$ is $1,A,B$, similarly to Case 2.2.1 we know $\{f(z_1'), f(z_2')\} = \{C,D\}$, we recolor $w_1',w_1,v_2$ with $B,E,1$, and color $u_1, u_2, u_3$ with $1,B,E$, which is a contradiction. When $f(z_1), f(z_2), f(w_2')$ is $1,B,1$, we again have $\{f(z_1'), f(z_2')\} = \{C,D\}$; similarly to Case 2.2.2, $f(N(w_2') - w_2) = \{B,D\}$, $f(N(w_3') - w_3) = \{D,E\}$, $f(N(w_4') - w_4) = \{B,E\}$; we recolor $w_1',w_1, w_2, w_3, v_2$ with $A,E,C,B,1$, and color $u_1, u_2, u_3$ with $1,B,E$, which is a contradiction. \hfill \qed

\subsection{Boundary vertices, pendant vertices, and separating cycles}
We prove the non-existence of $1$-cuts, $2$-cuts, and configurations in Lemmas~\ref{no3|3}-\ref{no4|6}, which guarantee that the boundary vertices and pendant vertices are different from those in configurations in Class One and Class Two. All cases that cannot be covered are listed in Figure~\ref{separating}. Computer check shows that none of them exists. Note that all the cases in Figure~\ref{separating} except the first case have a $3$-cut.
\begin{figure}
\vspace{-20mm}
\begin{center}
\hspace{-15mm}
\includegraphics[scale=0.9]{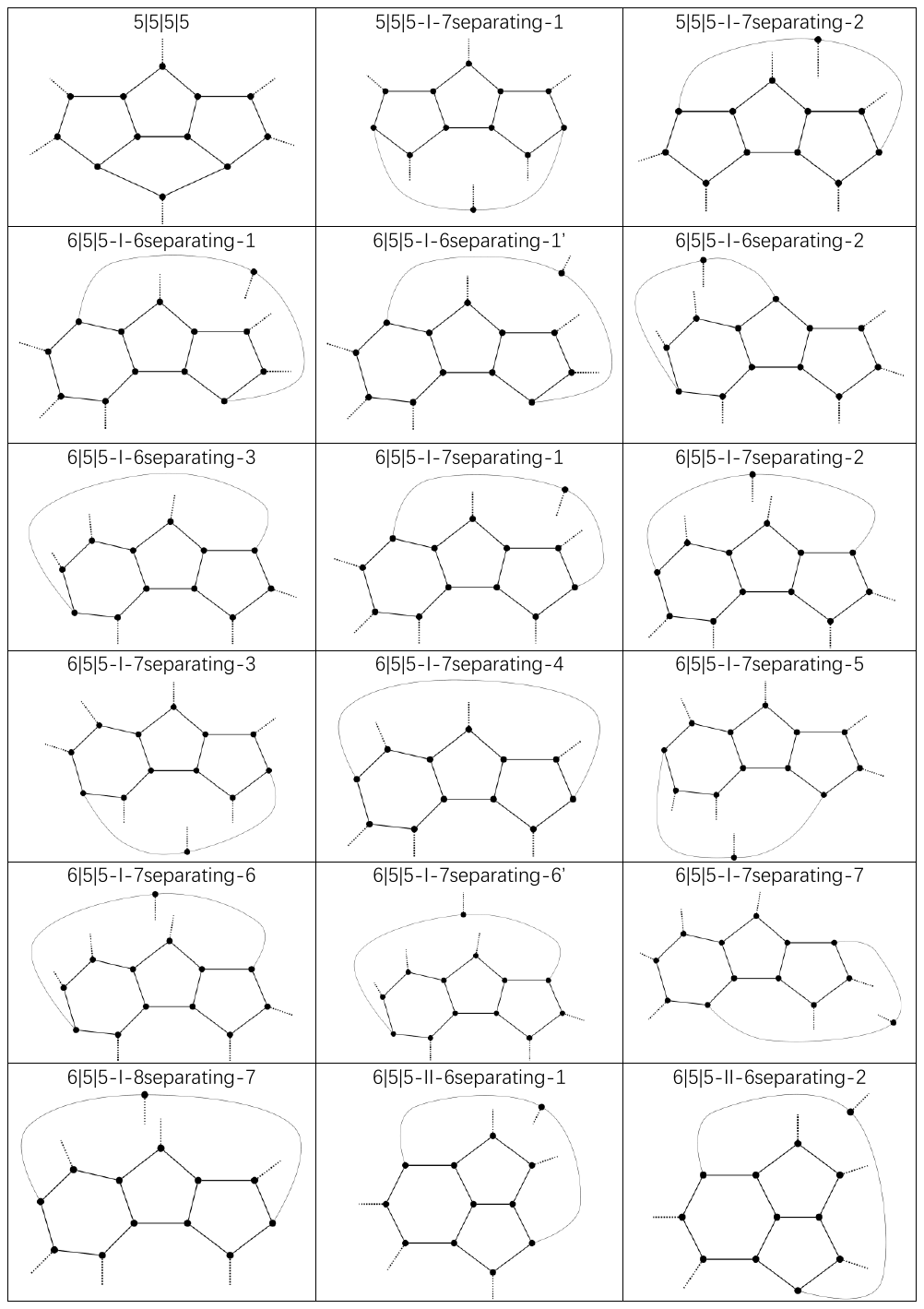}
\end{center}
\end{figure}

\begin{figure}
\vspace{-20mm}
\begin{center}
\hspace{-15mm}
\includegraphics[scale=0.9]{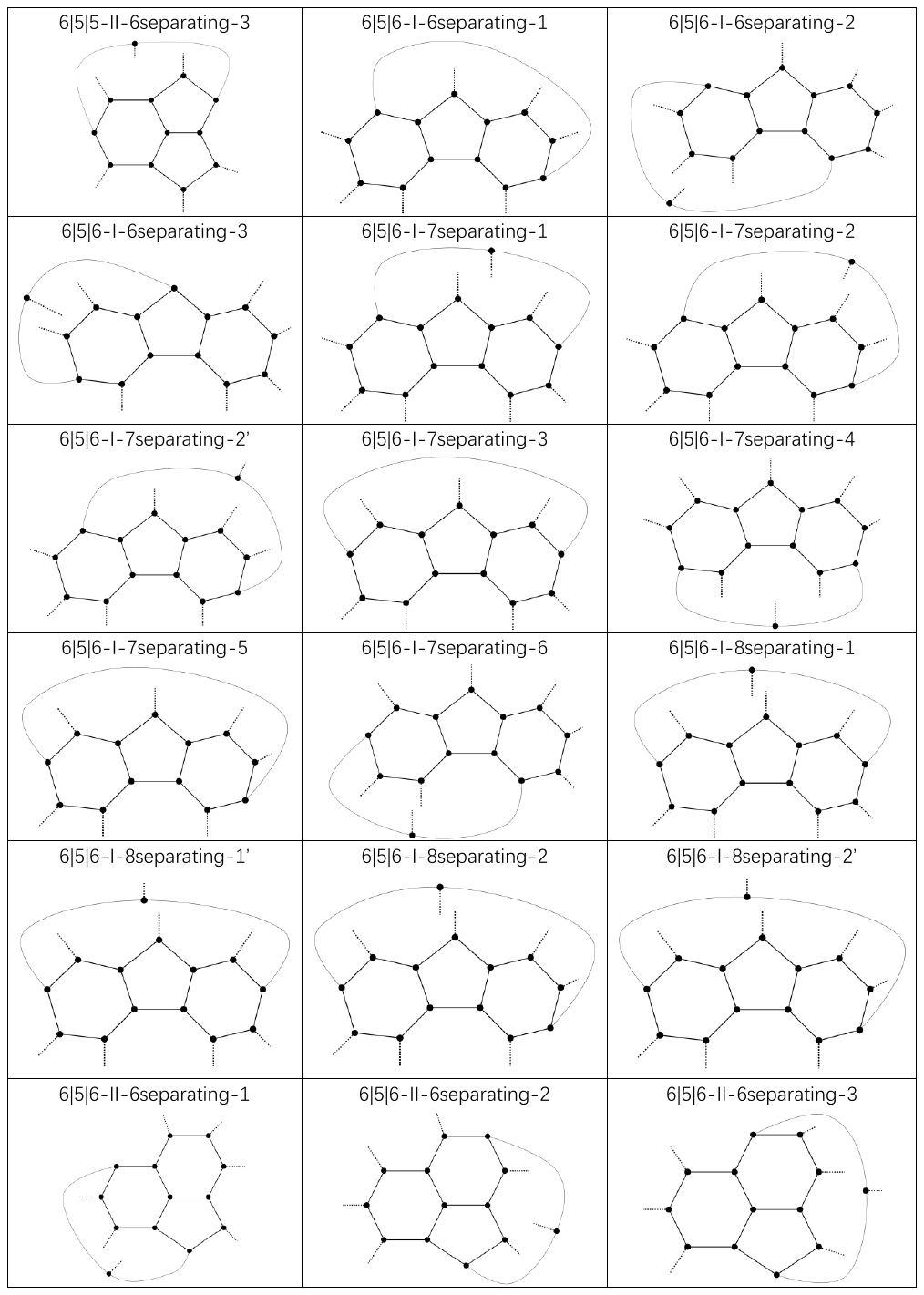}
\end{center}
\end{figure}

\begin{figure}
\vspace{-20mm}
\begin{center}
\hspace{-15mm}
\includegraphics[scale=0.9]{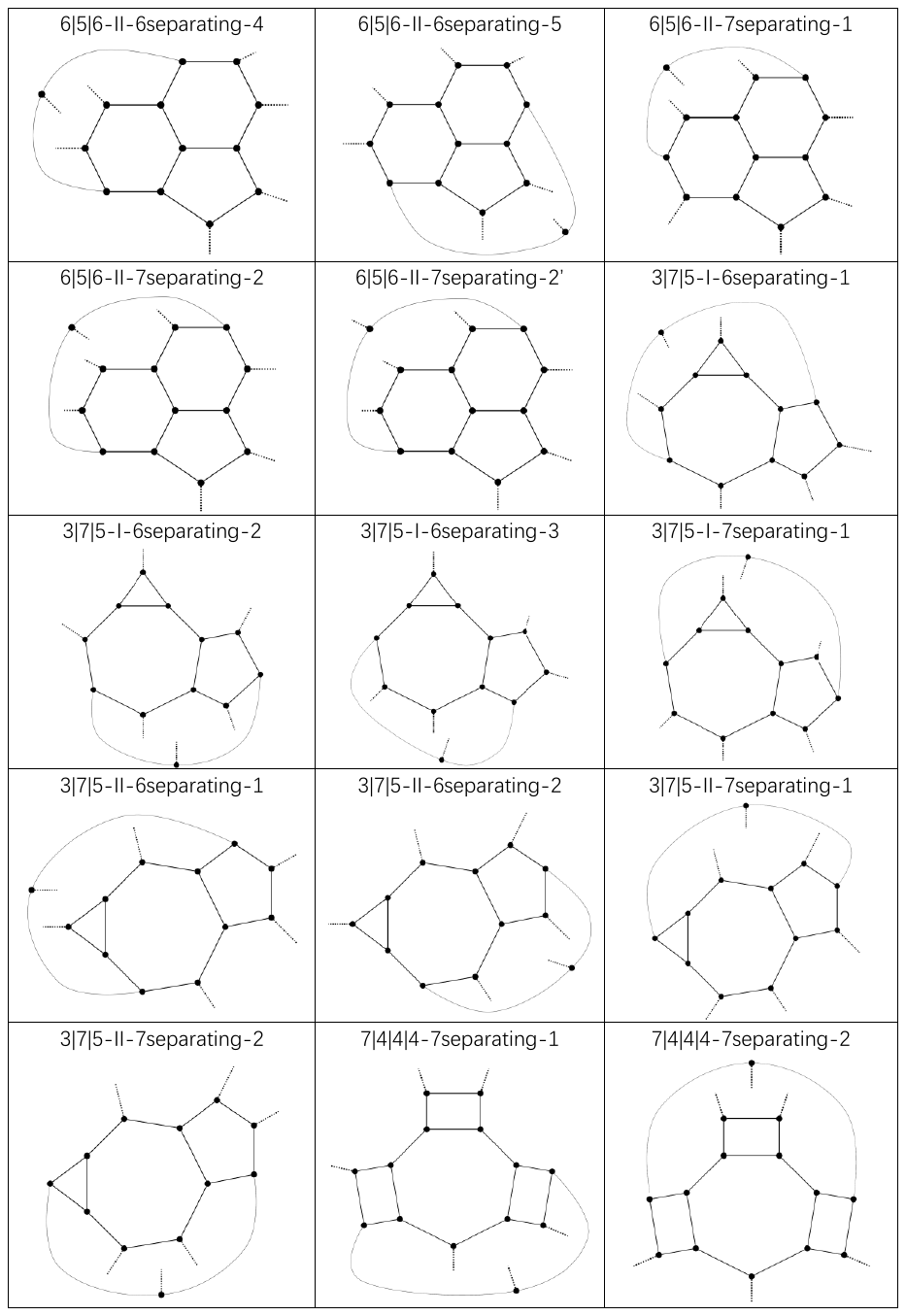}
\end{center}
\end{figure}

\begin{figure}
\vspace{-20mm}
\begin{center}
\hspace{-15mm}
\includegraphics[scale=0.9]{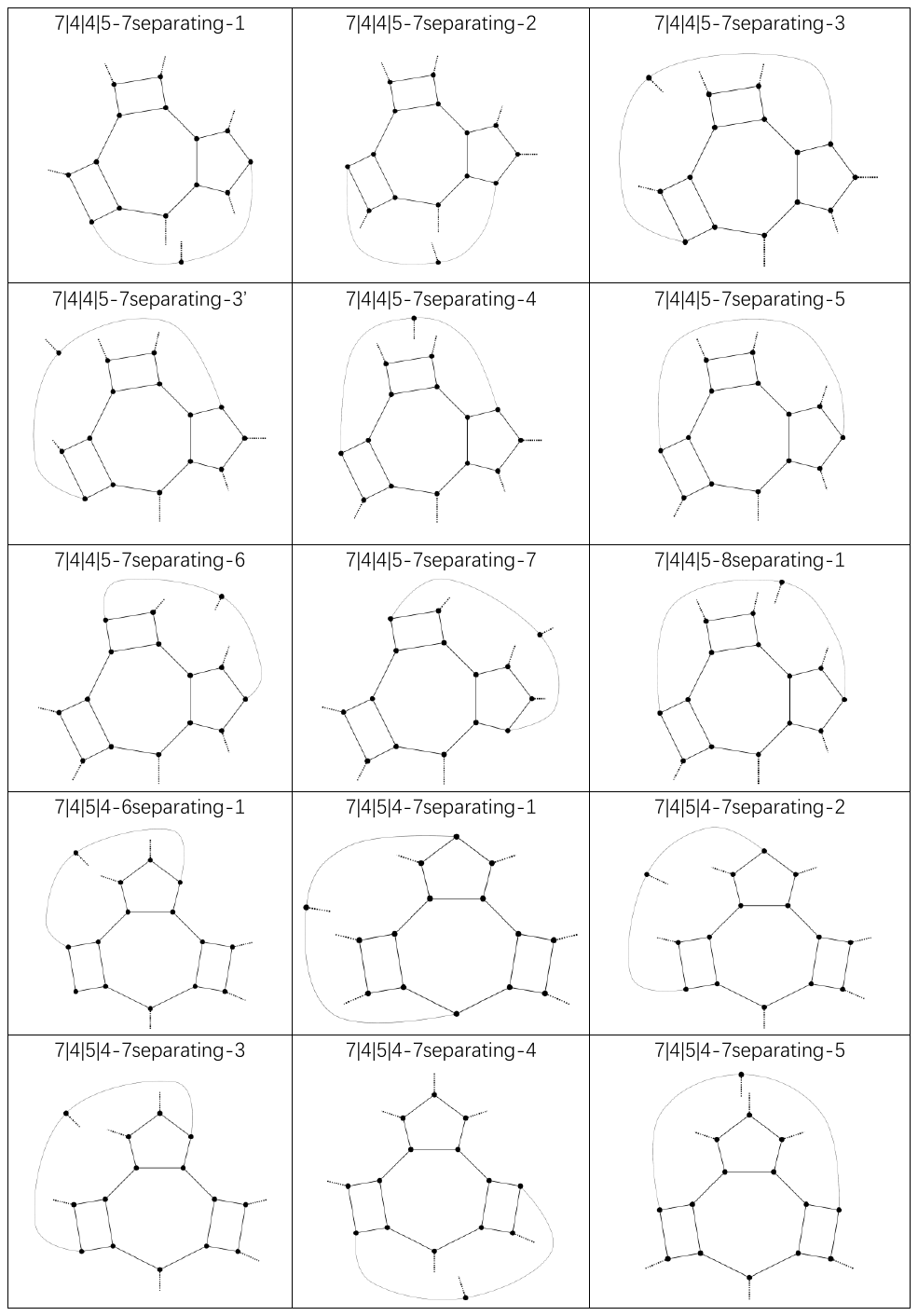}
\end{center}
\end{figure}

\begin{figure}
\vspace{-20mm}
\begin{center}
\hspace{-15mm}
\includegraphics[scale=0.9]{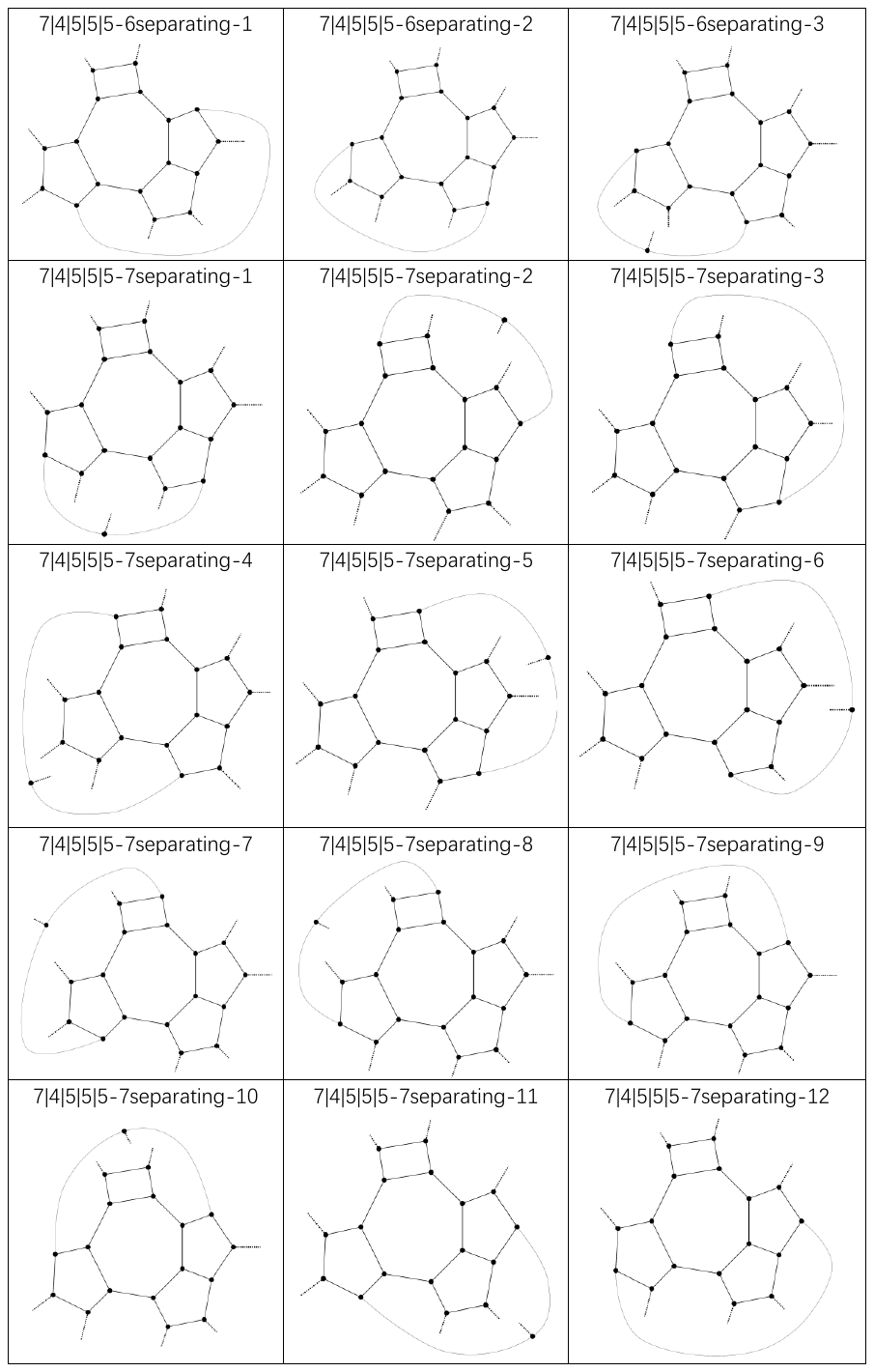}
\end{center}
\end{figure}

\newpage
\begin{figure}
  \vspace{-25mm}
  \hspace{-15mm}
  \includegraphics[scale=0.9]{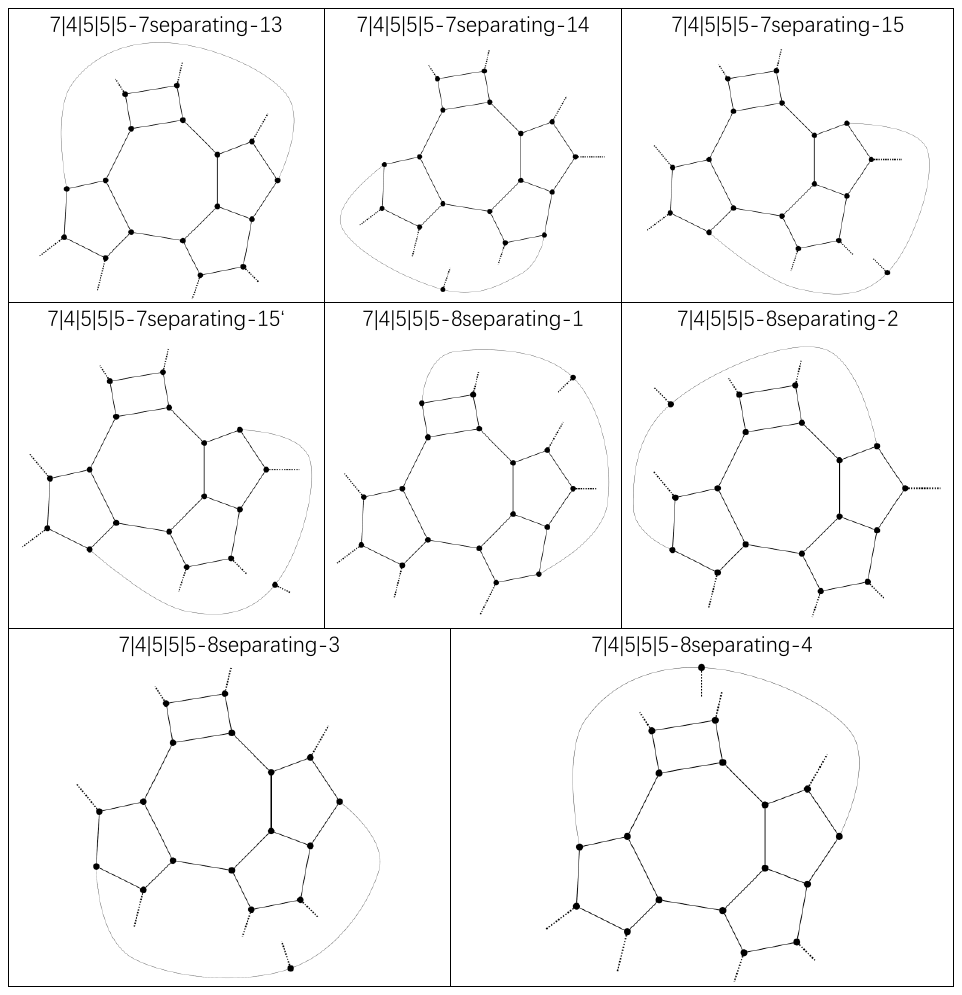}
  \vspace{-93mm}
  \caption{Boundary vertices, pendant vertices, and separating cycles.}\label{separating}
  \vspace{-3mm}
\end{figure}

\subsection{C++ Program}

\begin{lstlisting}
#include <iostream>
#include <fstream>
#include <vector>
#include <cassert>
#include <string>
#include <stdlib.h>
#include <time.h>

using namespace std;

const int V = 100;

#define COLORING_122222

/*
Colors used in packing coloring
0 . . uncolored
1 . . distance 1
2 . . distance 2
3 . . distance 2
4 . . distance 2
5 . . distance 2
6 . . distance 2
*/

#ifdef COLORING_122222
const int COLORS = 6 ;
const int MAX_DIST = 2 ; // Largest distance
// In cubic graphs, pending vertices are boundary vertices in a reducible  
// configuration plus its two neighbors not in the reducible configuration.
const int PENDANT_VERTICES = 3 ; 

// Pending vertices are represented as an array of size 3 where
// the first vertex is the boundary vertex in the reducible configuration;
// the second and third vertices are the neighbors of the first vertex 
// not in the reducible configuration.
// Note we can always recolor the first vertex with color 1 if the second
// and third neighbors do not have color 1.
const int outer_colors[][ PENDANT_VERTICES]= {
	{6,1,2},{6,1,3},{6,1,4},{6,1,5},
	{5,1,2},{5,1,3},{5,1,4},{5,1,6},
	{4,1,2},{4,1,3},{4,1,5},{4,1,6},
	{3,1,2},{3,1,4},{3,1,5},{3,1,6},
	{2,1,3},{2,1,4},{2,1,5},{2,1,6},
	{1,2,3},{1,2,4},{1,2,5},{1,2,6},{1,3,4},{1,3,5},{1,3,6},{1,4,5},{1,4,6},{1,5,6}
} ;

const int  outer_cnt = 30 ;
const int  outer_cnt_first = 2 ;
const int  coldist[] = { 0 , 0 , 1 , 1 , 1 , 1, 1 } ; // Index of color to conflict graph
const char* filename = "C6C5C6_typeII_extra_edge"; // Input file name
#endif



vector<vector<int> > triples;

// Adjacency matrix
int G[MAX_DIST][V][V] ; // Index 0 and index 1 are G and G^2
int v = 0 ; // Number of vertices

// Adjacency matrix with extra edges
int G_extra_edge[MAX_DIST][V][V] ; // Index 0 and index 1 are G and G^2

// List of degrees
int deg[MAX_DIST][V]; // Index 0 and index 1 are G and G^2
// List of neighbors
int N[MAX_DIST][V][V] ; // Index 0 and index 1 are G and G^2

// List of degrees with extra edges
int deg_extra_edge[MAX_DIST][V] ; // Index 0 and index 1 are G and G^2
// List of neighbors with extra edges
int N_extra_edge[MAX_DIST][V][V] ; // Index 0 and index 1 are G and G^2

int coloring[V] ;

clock_t clock_start,clock_end; // Time count

// Print current coloring
void print_coloring()
{
	for(int x = 0 ; x < v ; x++) {
		cout << coloring[ x ] << " " ;
	}
}

// Recursively check whether the current partial coloring is extendable
// x is the index of current vertex to color
bool is_coloring_extendable( int x ){
	// We have colored all the vertices
	if ( x >= v )
	{
		return true;
	}

	// If the current vertex x has been colored, then try next vertex x+1
	if ( coloring[ x ] != 0) return is_coloring_extendable(x+1) ;

	// Now color vertex x from color 1 to color COLORS
	for ( int c = 1 ; c <= COLORS; c++)
	{
		// Test if color c is possible
		bool found_conflict = false;
		for ( int i = 0 ; i < deg[ coldist[c] ][ x ] ; i++)
		{
			if ( coloring[ N[ coldist[ c ] ][ x ][ i ] ] == c )
			{
				found_conflict = true ;
				break ;
			}
		}
		
		if ( found_conflict ) continue ;

		coloring[ x ] = c ;
		if ( is_coloring_extendable(x+1) )
		{
			coloring[ x ] = 0 ;
			return true ;
		}
	}

	coloring[ x ] = 0 ;
	return false;
}

// Recursively generate partial colorings for pendant vertices triples from index to_color to the last one
// to_color is the index of the next pendant vertices triple
bool generate_precolorings( int to_color )
{
	// Remove any color for pendant vertiecs triples from to_color to the last one
	for (int next_triple = to_color; next_triple < ( int ) triples.size( ); next_triple++ )
	{
		for ( int i = 0 ; i < PENDANT_VERTICES; i++)
		{
			coloring[ triples[ next_triple ][ i ] ] = 0 ;
		}
	}

	// Before starting coloring vertices, first check whether the color assignment
	// of the pendant vertices have no conflict as we add extra edges between them.
	// If there is conflict, it means the color assignment of pendant vertices is 
	// not valid. So we do not have to check this case and can just return true.
	for ( int x = 0 ; x <= v-1; x++)
	{
		int c = coloring[ x ];
		if (c == 0)
		{
			continue;
		}

		for ( int i = 0 ; i < deg_extra_edge[ coldist[c] ][ x ] ; i++)
		{
			if ( coloring[ N_extra_edge[ coldist[ c ] ][ x ][ i ] ] == c )
			{
				return true;
			}
		}
	}

	// When all the pendant vertices have colors, call is_coloring_extendable to 
	// extend this coloring to all the vertices.
	if ( to_color >= ( int ) triples.size( ) )
	{
		if ( !is_coloring_extendable( 0 ) )
		{
			cout << " Precoloring " ;
			print_coloring( ) ;
			cout << " does not extend " << endl ;
			return false;
		}
		return true;
	}

	// Color the current pendant vertices by all possibilities,
	// then recursively generate partial colorings
	for ( int oc = 0 ; oc < outer_cnt ; oc++)
	{
		for ( int i = 0 ; i < PENDANT_VERTICES; i++)
		{
			coloring[ triples[ to_color ][ i ] ] = outer_colors[ oc ][ i ] ;
		}
		if ( ! generate_precolorings( to_color+1) ) return false;
	}

	return true ;
}

// Initialize generate coloring for pendant vertices and 
// extend this coloring to all vertices
bool generate_precolorings_first( )
{
	if (triples.size() > 0) 
	{
		for ( int oc = 0 ; oc < outer_cnt_first; oc++)
		{
			for ( int i = 0 ; i < PENDANT_VERTICES; i++) {
				coloring[ triples[ 0 ][ i ] ] = outer_colors[ oc ][ i ] ;
			}
			if ( ! generate_precolorings(1) ) return false;
		}
	}
	else
	{
		if ( ! generate_precolorings( 0 ) ) return false;
	}

	return true;
}




int main ( int argc , char* argv [ ] )
{
	// Start time count
	clock_start = clock();

	// Read input file
	ifstream input ;
	input.open( filename ) ;
	cerr << "Opening file : " << filename << endl ;

	while( input.good( ) )
	{
		// First line of input file is the name of reducible configuration.
		string name;
		while ( input.good( ) && (name.length( )==0 || name == " " || name == "  " || name == "   " ) )
		{
			std::getline( input, name ) ;
			cerr << name << endl ;
		}
		if ( ! input.good( ) ) break ;
		cout << " ' " << name << " ' " ;

		// First line of input file is the number of pendant vertices triples.
		int triples_count = 0 ;
		input >> triples_count;

		// Next triples_count lines are the pendant vertices triples.
		cout << "Loading " << triples_count << " triples " << endl ;
		triples.clear();
		for ( int i = 0 ; i < triples_count ; i++)
		{
			vector<int> abc;
			for ( int i = 0 ; i < PENDANT_VERTICES; i++)
			{
				int x;
				input >> x ;
				abc.push_back( x ) ;
			}
			triples.push_back( abc ) ;
		}

		// Load graph without extra edges
		// Next line is the number of vertices
		input >> v ;
		cout << "Loading graph without extra edges on " << v << " vertices " << endl ;
		assert( v < V);

		// Input is adjacency matrix, edge is 2, non-edge is 1
		for ( int x = 0 ; x < v ; x++)
		{
			G[ 0 ][ x ][ x ] = 0 ;
			for ( int y = x+1; y < v ; y++)
			{
				int e ;
				input >> e ;
				G[ 0 ][ x ][ y ] = e - 1;
				G[ 0 ][ y ][ x ] = e - 1;
			}
		}

		// Generate distances of G and G^2 without extra edges
		for ( int g = 1 ; g < MAX_DIST; g++)
		{
			for ( int x = 0 ; x < v ; x++)
			{
				G[ g ][ x ][ x ] = 0 ;
				for ( int y = x+1; y < v ; y++)
				{
					G[ g ][ y ][ x ] = G[ g ][ x ] [ y ]= G[ g - 1 ][ x ][ y ] ;
					for ( int z = 0 ; z < v ; z++)
					{
						if (G[ g - 1 ] [ x ] [ z ] == 1 && G[ 0 ] [ y ] [ z ] == 1)
						{
							G[ g ] [ x ] [ y ] = 1 ;
							G[ g ] [ y ] [ x ] = 1 ;
						}
					}
				}
			}
		}


		// Create list of neighbors without extra edges
		for ( int g = 0 ; g < MAX_DIST; g++)
		{
			for ( int x = 0 ; x < v ; x++)
			{
				deg[ g ][ x ] = 0 ;
				for ( int y = 0 ; y < v ; y++)
				{
					if ( y == x || G[ g ][ x ][ y ] == 0) continue ;
					N[ g ][ x ][ deg[ g ][ x ]++ ] = y ;
				}
			}
		}


		// Load graph with extra edges
		// Next line is the number of vertices
		input >> v ;
		cout << "Loading graph with extra edges on " << v << " vertices " << endl ;
		assert( v < V);

		// Input is adjacency matrix, edge is 2, non-edge is 1
		for ( int x = 0 ; x < v ; x++)
		{
			G_extra_edge[ 0 ][ x ][ x ] = 0 ;
			for ( int y = x+1; y < v ; y++)
			{
				int e ;
				input >> e ;
				G_extra_edge[ 0 ][ x ][ y ] = e - 1;
				G_extra_edge[ 0 ][ y ][ x ] = e - 1;
			}
		}

		// Generate distances of G and G^2 with extra edges
		for ( int g = 1 ; g < MAX_DIST; g++)
		{
			for ( int x = 0 ; x < v ; x++)
			{
				G_extra_edge[ g ][ x ][ x ] = 0 ;
				for ( int y = x+1; y < v ; y++)
				{
					G_extra_edge[ g ][ y ][ x ] = G_extra_edge[ g ][ x ] [ y ] = G_extra_edge[ g - 1 ][ x ][ y ] ;
					for ( int z = 0 ; z < v ; z++)
					{
						if (G_extra_edge[ g - 1 ] [ x ] [ z ] == 1 && G_extra_edge[ 0 ] [ y ] [ z ] == 1)
						{
							G_extra_edge[ g ] [ x ] [ y ] = 1 ;
							G_extra_edge[ g ] [ y ] [ x ] = 1 ;
						}
					}
				}
			}
		}


		// Create list of neighbors with extra edges
		for ( int g = 0 ; g < MAX_DIST; g++)
		{
			for ( int x = 0 ; x < v ; x++)
			{
				deg_extra_edge[ g ][ x ] = 0 ;
				for ( int y = 0 ; y < v ; y++)
				{
					if ( y == x || G_extra_edge[ g ][ x ][ y ] == 0) continue ;
					N_extra_edge[ g ][ x ][ deg_extra_edge[ g ][ x ]++ ] = y ;
				}
			}
		}


		// Uncolor everything
		for ( int x = 0 ; x < v ; x++) 
			coloring[ x ] = 0 ;

		if ( generate_precolorings_first( ) )
		{
			cout << "Reducible " << endl ;
		}
	}

	input.close();

	// Time count end
	clock_end = clock();
	cout<<"time = "<<double(clock_end-clock_start)/CLOCKS_PER_SEC<<" seconds"<<endl;

	return 0;
}
\end{lstlisting}

\subsection{Sample input file}

\begin{lstlisting}
C6C5C6_typeII_extra_edge
9
13 14 15
16 17 18
19 20 21
22 23 24
25 26 27
28 29 30
31 32 33
34 35 36
37 38 39
40  2 1 1 1 1 1 1 1 1 1 1 2 2 1 1 1 1 1 1 1 1 1 1 1 1 1 1 1 1 1 1 1 1 1 1 1 1 1 1  2 1 1 1 1 1 1 1 1 1 1 1 1 1 2 1 1 1 1 1 1 1 1 1 1 1 1 1 1 1 1 1 1 1 1 1 1 1  2 1 1 1 1 1 1 2 1 1 1 1 1 1 1 1 1 1 1 1 1 1 1 1 1 1 1 1 1 1 1 1 1 1 1 1 1  2 1 1 1 1 1 1 1 1 1 1 1 1 1 1 2 1 1 1 1 1 1 1 1 1 1 1 1 1 1 1 1 1 1 1 1  2 1 1 1 2 1 1 1 1 1 1 1 1 1 1 1 1 1 1 1 1 1 1 1 1 1 1 1 1 1 1 1 1 1 1  2 1 1 1 1 1 1 1 1 1 1 1 1 1 1 1 2 1 1 1 1 1 1 1 1 1 1 1 1 1 1 1 1 1  2 1 1 1 1 1 1 1 1 1 1 1 1 1 1 1 1 1 2 1 1 1 1 1 1 1 1 1 1 1 1 1 1  2 1 1 1 1 1 1 1 1 1 1 1 1 1 1 1 1 1 1 1 2 1 1 1 1 1 1 1 1 1 1 1  2 1 1 1 1 1 1 1 1 1 1 1 1 1 1 1 1 1 1 1 1 1 2 1 1 1 1 1 1 1 1  2 1 1 1 1 1 1 1 1 1 1 1 1 1 1 1 1 1 1 1 1 1 1 1 1 1 1 1 1 1  2 1 1 1 1 1 1 1 1 1 1 1 1 1 1 1 1 1 1 1 1 1 1 1 1 1 1 1 1  2 1 1 1 1 1 1 1 1 1 1 1 1 1 1 1 1 1 1 1 1 1 2 1 1 1 1 1  1 1 1 1 1 1 1 1 1 1 1 1 1 1 1 1 1 1 1 1 1 1 1 1 2 1 1  2 2 1 1 1 1 1 1 1 1 1 1 1 1 1 1 1 1 1 1 1 1 1 1 1 1  1 1 1 1 1 1 1 1 1 1 1 1 1 1 1 1 1 1 1 1 1 1 1 1 1  1 1 1 1 1 1 1 1 1 1 1 1 1 1 1 1 1 1 1 1 1 1 1 1  2 2 1 1 1 1 1 1 1 1 1 1 1 1 1 1 1 1 1 1 1 1 1  1 1 1 1 1 1 1 1 1 1 1 1 1 1 1 1 1 1 1 1 1 1  1 1 1 1 1 1 1 1 1 1 1 1 1 1 1 1 1 1 1 1 1  2 2 1 1 1 1 1 1 1 1 1 1 1 1 1 1 1 1 1 1  1 1 1 1 1 1 1 1 1 1 1 1 1 1 1 1 1 1 1  1 1 1 1 1 1 1 1 1 1 1 1 1 1 1 1 1 1  2 2 1 1 1 1 1 1 1 1 1 1 1 1 1 1 1  1 1 1 1 1 1 1 1 1 1 1 1 1 1 1 1  1 1 1 1 1 1 1 1 1 1 1 1 1 1 1  2 2 1 1 1 1 1 1 1 1 1 1 1 1  1 1 1 1 1 1 1 1 1 1 1 1 1  1 1 1 1 1 1 1 1 1 1 1 1  2 2 1 1 1 1 1 1 1 1 1  1 1 1 1 1 1 1 1 1 1  1 1 1 1 1 1 1 1 1  2 2 1 1 1 1 1 1  1 1 1 1 1 1 1  1 1 1 1 1 1  2 2 1 1 1  1 1 1 1  1 1 1  2 2  1  
40  2 1 1 1 1 1 1 1 1 1 1 2 2 1 1 1 1 1 1 1 1 1 1 1 1 1 1 1 1 1 1 1 1 1 1 1 1 1 1  2 1 1 1 1 1 1 1 1 1 1 1 1 1 2 1 1 1 1 1 1 1 1 1 1 1 1 1 1 1 1 1 1 1 1 1 1 1  2 1 1 1 1 1 1 2 1 1 1 1 1 1 1 1 1 1 1 1 1 1 1 1 1 1 1 1 1 1 1 1 1 1 1 1 1  2 1 1 1 1 1 1 1 1 1 1 1 1 1 1 2 1 1 1 1 1 1 1 1 1 1 1 1 1 1 1 1 1 1 1 1  2 1 1 1 2 1 1 1 1 1 1 1 1 1 1 1 1 1 1 1 1 1 1 1 1 1 1 1 1 1 1 1 1 1 1  2 1 1 1 1 1 1 1 1 1 1 1 1 1 1 1 2 1 1 1 1 1 1 1 1 1 1 1 1 1 1 1 1 1  2 1 1 1 1 1 1 1 1 1 1 1 1 1 1 1 1 1 2 1 1 1 1 1 1 1 1 1 1 1 1 1 1  2 1 1 1 1 1 1 1 1 1 1 1 1 1 1 1 1 1 1 1 2 1 1 1 1 1 1 1 1 1 1 1  2 1 1 1 1 1 1 1 1 1 1 1 1 1 1 1 1 1 1 1 1 1 2 1 1 1 1 1 1 1 1  2 1 1 1 1 1 1 1 1 1 1 1 1 1 1 1 1 1 1 1 1 1 1 1 1 1 1 1 1 1  2 1 1 1 1 1 1 1 1 1 1 1 1 1 1 1 1 1 1 1 1 1 1 1 1 1 1 1 1  2 1 1 1 1 1 1 1 1 1 1 1 1 1 1 1 1 1 1 1 1 1 2 1 1 1 1 1  1 1 1 1 1 1 1 1 1 1 1 1 1 1 1 1 1 1 1 1 1 1 1 1 2 1 1  2 2 2 1 1 1 1 1 1 1 1 1 1 1 1 1 1 1 1 1 1 1 1 1 1 1  1 1 1 1 1 1 1 1 1 1 1 1 1 1 1 1 1 1 1 1 1 1 1 1 1  1 1 1 1 1 1 1 1 1 1 1 1 1 1 1 1 1 1 1 1 1 1 1 1  2 2 1 1 1 1 1 1 1 1 1 1 1 1 1 1 1 1 1 1 1 1 1  1 1 1 1 1 1 1 1 1 1 1 1 1 1 1 1 1 1 1 1 1 1  1 1 1 1 1 1 1 1 1 1 1 1 1 1 1 1 1 1 1 1 1  2 2 2 1 1 1 1 1 1 1 1 1 1 1 1 1 1 1 1 1  1 1 1 1 1 1 1 1 1 1 1 1 1 1 1 1 1 1 1  1 1 1 1 1 1 1 1 1 1 1 1 1 1 1 1 1 1  2 2 1 1 1 1 1 1 1 1 1 1 1 1 1 1 1  1 1 1 1 1 1 1 1 1 1 1 1 1 1 1 1  1 1 1 1 1 1 1 1 1 1 1 1 1 1 1  2 2 2 1 1 1 1 1 1 1 1 1 1 1  1 1 1 1 1 1 1 1 1 1 1 1 1  1 1 1 1 1 1 1 1 1 1 1 1  2 2 1 1 1 1 1 1 1 1 1  1 1 1 1 1 1 1 1 1 1  1 1 1 1 1 1 1 1 1  2 2 2 1 1 1 1 1  1 1 1 1 1 1 1  1 1 1 1 1 1  2 2 1 1 1  1 1 1 1  1 1 1  2 2  1
\end{lstlisting}

\end{document}